\documentclass[english, a4paper, 12pt]{article}
\usepackage{multicol,amsmath,t1enc,a4,babel}
\usepackage{amsthm}
\usepackage{amssymb}
\usepackage[dvips]{graphicx}
\usepackage[dvips]{epsfig}
\usepackage{times}
\usepackage{amsfonts}
\usepackage{wasysym}
\usepackage{mathtools}
\usepackage[colorinlistoftodos]{todonotes}
\usepackage{mathrsfs}
\usepackage{relsize}
\usepackage{bbm}
\usepackage{graphicx}
\usepackage{subcaption}
\usepackage{float}

\newcommand{\e}{\mathbb{E}}

\newcommand{\p}{\mathbb{P}}

\newcommand{\ha}{\hat{\gamma}_N}

\newcommand{\bo}{{\pmb{\gamma}}}
\newcommand{\boh}{{\hat{\pmb{\gamma}}_N}}
\newcommand{\bohh}{{\hat{\pmb{\gamma}}_{0,N}}}

\newcommand{\bxi}{{\pmb{\xi}}}

\newcommand{\indi}{\mathbbm{1}_{L_{t-n} \geq 1}}
\newtheorem{defi}{Definition}[section]
\newtheorem{lemma}[defi]{Lemma}

\newtheorem{theorem}[defi]{Theorem}
\newtheorem{kor}[defi]{Corollary}

\newtheorem{rem}[defi]{Remark}

\begin{document}

\title{On generalized ARCH model with stationary liquidity}

\renewcommand{\thefootnote}{\fnsymbol{footnote}}

\author{Pauliina Ilmonen\footnotemark[1] \, and \, Soledad Torres\footnotemark[2]\, and \, Ciprian Tudor\footnotemark[3] \\ \, and \, Lauri Viitasaari\footnotemark[4] \, and \, Marko Voutilainen\footnotemark[1]}

\footnotetext[1]{Department of Mathematics and Systems Analysis, Aalto University School of Science, Finland}

\footnotetext[2]{CIMFAV, Facultad de Ingenier\'ia, Universidad de Valpara\'iso,
 Valparaiso, Chile}

\footnotetext[3]{UFR Math\'ematiques, Universit\'e de Lille 1, France}

\footnotetext[4]{Department of Mathematics and Statistics, University of Helsinki, 
Finland}

\maketitle

\begin{abstract}
\noindent
We study a  generalized ARCH model with liquidity given by a general stationary  process. We provide minimal assumptions that ensure the existence and uniqueness of the stationary solution. In addition, we provide consistent estimators for the model parameters by using AR(1) type characterisation. We illustrate our results with several examples and simulation studies.
\end{abstract}

{\small
\medskip


\medskip

}




\section{Introduction}

The ARCH and GARCH models have become important tools in time series analysis. The ARCH model has been introduced by Engle in \cite{En} and then it has been generalized by Bollerslev to the GARCH model in \cite{Bol1}. Since, a large collection of variants and extensions of these models has been  produced by many authors.  See for example \cite{Bol2}  for a glossary of models derived from ARCH and GARCH. 

In this work, we also focus on a generalized ARCH model, namely the model (\ref{arch}).  Our contribution proposes to include in the expression of the squared volatility $\sigma _{t} ^{2} $ a factor $L_{t-1}$, which we will call {\it liquidity.} The motivation to consider such a model comes from mathematical finance, where the factor $L_{t}$, which constitutes a proxi for the trading volume at day $t$, has been included in order to capture the fluctuations of the intra-day price in financial markets. A more detailed explanation can be found in  \cite{BTT} or \cite{TT}. In the work \cite{BTT} we considered the particular case when $L_{t}$ is the squared increment of the fractional Brownian motion (fBm in the sequel), i.e. $L_{t}= (B ^{H}_{t+1} -B^{H}_{t}) ^{2} $, where $B ^{H}$ is a fBm with Hurst parameter $ H\in (0,1)$. 

In this work, our purpose is twofold. Firstly, we enlarge the ARCH with fBm liquidity in \cite{BTT} by considering, as a proxi for the liquidity, a general positive (strictly) stationary process $(L_{t}) _{t\in \mathbb{Z}}$.  This includes, besides the above mentioned case of the squared increment of the fBm,  many other examples. 

 The second purpose is to provide a method to estimate the parameters of the model.  As mentioned in \cite{BTT},  in the case when $L$ is a process without independent increments, the usual approaches for the parameter estimation in ARCH models (such as least squares method and maximum likelihood method) do not work, in the sense that the estimators obtained by these classical methods are biased and not consistent. Here we adopt a different technique, based on the AR(1) characterization of the ARCH process, which has also been used in \cite{vouti}. The AR(1) characterization leads to Yule-Walker type equations for the parameters of the model. These equations are of quadratic form and  then we are able to find explicit formulas for the estimators. We prove that the estimators are consistent by using extended version of the law of large numbers and by assuming enough regularity for the correlation structure of the liquidity process. We also provide a numerical analysis of the estimators.

The rest of the paper is organised as follows. In Section \ref{sec:model} we introduce our model and prove the existence and uniqueness of the stationary solution. We also provide necessary and sufficient conditions for the existence of the autocovariance function. We derive the AR(1) characterization and  Yule-Walker type equations for the parameters of the model.  Section \ref{sec:estimation} is devoted to the estimation of the model parameters.  We construct estimators in a closed form and we prove their consistency via extended versions of the law of large numbers and a control of the behaviour of the covariance of the liquidity process.  Several examples are discussed in details. In particular, we study squared increments of the fBm, squared increments of the compensated Poisson process, and the squared increments of the Rosenblatt process.
We end the paper with a numerical analysis of our estimators.

\section{The model}
\label{sec:model}
The generalized ARCH model is defined for every $t\in\mathbb{Z}$ as 

\begin{equation}
\label{arch}
X_t = \sigma_t\epsilon_t, \qquad \sigma_t^2 = \alpha_0+\alpha_1X_{t-1}^2 + l_1L_{t-1},
\end{equation}
where $\alpha_0\geq 0$, $\alpha_1,l_1>0$, and $(\epsilon_t)_{t\in\mathbb{Z}}$ is an i.i.d. process with $\e(\epsilon_0) = 0$ and $\e(\epsilon_0^2)=1$. Moreover, $(L_t)_{t\in\mathbb{Z}}$ is a strictly stationary positive process with $\e(L_0) = 1$ and independent of $(\epsilon_t)_{t\in\mathbb{Z}}$. We first give sufficient conditions to ensure the existence of a stationary solution.
\noindent
Note that we have a recursion

\begin{equation}
\label{recursiveformula}
\sigma_t^ 2 = \alpha_0 + \alpha_1 \epsilon_{t-1}^ 2 \sigma_{t-1}^ 2 + l_1 L_{t-1}.
\end{equation}
Let us denote

\begin{equation*}
A_t = \alpha_1\epsilon_t^ 2\quad\text{and} \quad B_t = \alpha_0 + l_1L_t \quad\text{for every }t\in\mathbb{Z}.
\end{equation*}
Using \eqref{recursiveformula} $k+1$ times we get

\begin{equation}
\label{iteration}
\begin{split}
\sigma_{t+1}^ 2 &= A_t \sigma_t^ 2 + B_t\\
&= A_tA_{t-1} \sigma_{t-1}^ 2 + A_tB_{t-1} + B_t\\
&=\hdots\\
&= \left(\prod_{i=0}^k A_{t-i}\right) \sigma_{t-k}^ 2 + \sum_{i=0}^ k\left(\prod_{j=0}^ {i-1} A_{t-j}\right) B_{t-i},
\end{split}
\end{equation}
with the convention $\prod_0^ {-1} = 1$.

The following lemma ensures that we are able to continue the recursion infinitely many times.
\begin{lemma}
\label{lemma:stationarity}
Suppose $\alpha_1<1$ and $\sup_{t\in\mathbb{Z}}\e(\sigma_t^ 2) \leq M_1<\infty$. Then, as $k\to \infty$, we have

\begin{equation*}
 \left(\prod_{i=0}^k A_{t-i}\right) \sigma_{t-k}^ 2 \to 0
\end{equation*}
in $L^ 1$. Furthermore, if $\alpha_1 <  \frac{1}{\sqrt{\e(\epsilon_0^ 4)}}$ and $\sup_{t\in\mathbb{Z}}\e(\sigma_t^4)\leq M_2 < \infty$, then the convergence holds also almost surely.

\begin{proof}
By independence of $\epsilon$, we have
\begin{equation*}
\begin{split}
\e\left| \left(\prod_{i=0}^k A_{t-i}\right) \sigma_{t-k}^ 2\right| = \alpha_1^ {k+1} \e(\sigma_{t-k}^ 2)\leq \alpha_1^ {k+1} M_1\to 0
\end{split}
\end{equation*}
proving the first part of the claim. 
For the second part, Chebysev's inequality implies

\begin{footnotesize}
\begin{equation*}
\begin{split}
\p\left(\left|\left(\prod_{i=0}^k A_{t-i}\right) \sigma_{t-k}^ 2 - \alpha_1^ {k+1}\e(\sigma_{t-k}^ 2)\right|>\varepsilon\right) &\leq \frac{\mathrm{Var}\left(\left(\prod_{i=0}^k A_{t-i}\right) \sigma_{t-k}^ 2\right)}{\varepsilon^ 2}\\
&=\frac{\alpha_1^ {2k+2} \e\left(\left(\prod_{i=0}^ k \epsilon^ 4_{t-i}\right) \sigma_{t-k}^ 4\right) - \alpha_1^ {2k+2} \e(\sigma_{t-k}^2)^2}{\varepsilon^ 2} \\
&\leq \frac{\left(\alpha_1^ 2 \e(\epsilon^ 4_0)\right)^ {k+1} M_2 - \alpha_1^ {2k+2} M_1^2}{\varepsilon^ 2},
\end{split}
\end{equation*}
\end{footnotesize}
which is summable by assumptions. Borel-Cantelli then implies 

\begin{equation*}
\left(\prod_{i=0}^k A_{t-i}\right) \sigma_{t-k}^ 2 - \alpha_1^ {k+1}\e(\sigma_{t-k}^ 2)\to 0
\end{equation*}
almost surely proving the claim.
\end{proof}
\end{lemma}

\subsection{Existence of a stationary solution}
The following theorem gives the existence of a stationary solution under relatively weak assumptions (we only assume the existence of the second moment of $L$ and the usual condition $\alpha_{1} <1 $ (see e.g. \cite{FZ})). 

\begin{theorem}
\label{theo:stationarity}
Assume that $\e(L_0^ 2) < \infty$ and  $\alpha_1< 1$. Then \eqref{arch} has the following strictly stationary solution 

\begin{equation}
\label{uniquesolution}
\sigma_{t+1}^ 2 = \sum_{i=0}^ \infty\left(\prod_{j=0}^ {i-1} A_{t-j}\right) B_{t-i}.
\end{equation}

\begin{proof}
We begin by showing that \eqref{uniquesolution} is well-defined. That is, we prove that 

\begin{equation*}
\lim_{k\to\infty} \sum_{i=0}^ k\left(\prod_{j=0}^ {i-1} A_{t-j}\right) B_{t-i}
\end{equation*}
defines an almost surely finite random variable. 
First we observe that the summands above are non-negative and hence, the pathwise limits exist in $[0,\infty]$. Write

\begin{equation}
\label{twoseries}
\sum_{i=0}^ \infty\left(\prod_{j=0}^ {i-1} A_{t-j}\right) B_{t-i} = \alpha_0 \sum_{i=0}^ \infty \left(\prod_{j=0}^ {i-1} A_{t-j}\right) + l_1 \sum_{i=0}^ \infty \left(\prod_{j=0}^ {i-1} A_{t-j}\right) L_{t-i}
\end{equation}
and denote

\begin{equation*}
a_n =\left (\prod_{j=0}^ {n-1} A_{t-j}\right) L_{t-n}, \quad b_n =\left (\prod_{j=0}^ {n-1} A_{t-j}\right).
\end{equation*}
By the root test it suffices to prove that 
\begin{equation}
\label{eq:root-a}
\limsup_{n\to\infty} a_n^ {\frac{1}{n}} < 1
\end{equation}
and 
\begin{equation}
\label{eq:root-b}
\limsup_{n\to\infty} b_n^ {\frac{1}{n}} < 1.
\end{equation}
Here

\begin{equation*}
a_n^{\frac{1}{n}} = e^{\frac{1}{n}\log a_n} = L_{t-n}^ {\frac{1}{n}} e^{\frac{1}{n} \sum_{j=0}^{n-1} \log A_{t-j}},
\end{equation*}
where

\begin{equation*}
e^{\frac{1}{n} \sum_{j=0}^{n-1} \log A_{t-j}} \overset{a.s}{\longrightarrow}  e^{\e\log A_0} = \alpha_1 e^{\e \log \epsilon_0^ 2}
\end{equation*}
by the law of large numbers and continuous mapping theorem. By Jensen's inequality we obtain that

\begin{equation*}
 \alpha_1 e^{\e \log \epsilon_0^ 2} \leq \alpha_1 e^{\log \e(\epsilon_0^ 2)} = \alpha_1 < 1.
\end{equation*}
That is 

\begin{equation*}
\lim_{n\to\infty} e^{\frac{1}{n} \sum_{j=0}^{n-1} \log A_{t-j}} < 1
\end{equation*}
almost surely. This proves \eqref{eq:root-b} which implies that the first series in \eqref{twoseries} is almost surely convergent. To obtain \eqref{eq:root-a}, it remains to show $\limsup_{n\to\infty} L_{t-n}^ {\frac{1}{n}} \leq 1$ almost surely. We have

\begin{equation}
\label{indicators}
L_{t-n}^ {\frac{1}{n}} = \mathbbm{1}_{L_{t-n} < 1}L_{t-n}^ {\frac{1}{n}} + \indi L_{t-n}^ {\frac{1}{n}} \leq 1 + \indi L_{t-n}^ {\frac{1}{n}} - \mathbbm{1}_{L_{t-n} \geq 1} 
\end{equation}
where we have used
$$
\mathbbm{1}_{L_{t-n} < 1}L_{t-n}^ {\frac{1}{n}} \leq \mathbbm{1}_{L_{t-n} < 1} = 1 - \mathbbm{1}_{L_{t-n} \geq 1}.
$$
Now

\begin{equation*}
\begin{split}
\p\left(\left|\indi L_{t-n}^ {\frac{1}{n}} - \indi\right| \geq \varepsilon\right) &\leq \frac{\e\left|\indi L_{t-n}^ {\frac{1}{n}} - \indi\right|^ 2}{\varepsilon^ 2}\\
&= \frac{\e\left(\indi\left(L_{t-n}^ {\frac{1}{n}} -1\right)^ 2\right)}{\varepsilon^ 2}.
\end{split}
\end{equation*}
Consider now the function $f_x(a) \coloneqq x^ a$ for  $x\geq 1$ and $a\geq 0$. Since $f_x'(a) = x^a\log x$ we obtain by the mean value theorem that

\begin{equation*}
\left|f_x(a) - f_x(0)\right| \leq \max_{0\leq b \leq a} \left|f_x'(b)\right|a = ax^a\log x.
\end{equation*}
Hence

\begin{equation*}
\indi \left(L_{t-n}^ {\frac{1}{n}} -1\right)^ 2 \leq\indi  \frac{1}{n^ 2}L_{t-n}^ {\frac{2}{n}} \left(\log L_{t-n}\right)^ 2.
\end{equation*}
On the other hand, for $n\geq 2$ and $L_{t-n} \geq 1$ it holds that

\begin{equation*}
\frac{ L_{t-n}^ {\frac{2}{n}} \left(\log L_{t-n}\right)^ 2}{L_{t-n}^ 2} \leq  \frac{ \left(\log L_{t-n}\right)^ 2}{L_{t-n}}< 1,
\end{equation*}
since for $x\geq 1$, the function $g(x) \coloneqq \left(\log x\right)^ 2x^ {-1}$ has the maximum $g(e^2) = 4e^ {-2}$. 
Consequently,

\begin{equation*}
 \frac{\e\left(\indi\left(L_{t-n}^ {\frac{1}{n}} -1\right)^ 2\right)}{\varepsilon^ 2} < \frac{\e\left(\indi L_{t-n}^ 2\right)}{\varepsilon^ 2 n^ 2} \leq \frac{\e\left(L_{t-n}^ 2\right)}{\varepsilon^ 2 n^ 2}.
\end{equation*}
Hence Borel-Cantelli implies

\begin{equation*}
\indi L_{t-n}^{\frac{1}{n}} - \indi \overset{\text{a.s.}}{\to} 0
\end{equation*}
which by \eqref{indicators} shows \eqref{eq:root-a}. Let us next show that \eqref{uniquesolution} satisfies \eqref{recursiveformula}.

\begin{equation*}
\begin{split}
A_t \sigma_t^ 2 + B_t &= \sum_{i=0}^\infty \left(\prod_{j=0}^i A_{t-j}\right) B_{t-i-1} + B_t\\
&= \sum_{i=1}^\infty \left(\prod_{j=0}^{i-1} A_{t-j}\right) B_{t-i} + B_t\\
&=  \sum_{i=0}^ \infty\left(\prod_{j=0}^ {i-1} A_{t-j}\right) B_{t-i} = \sigma_{t+1}^ 2.
\end{split}
\end{equation*}
It remains to prove that \eqref{uniquesolution} is stationary. However, since $(A_t, B_t)$ is stationary, we have

\begin{equation*}
\sum_{i=0}^ k\left(\prod_{j=0}^ {i-1} A_{t-j}\right) B_{t-i} \overset{\text{law}}{=} \sum_{i=0}^ k\left(\prod_{j=0}^ {i-1} A_{-j}\right) B_{-i}
\end{equation*}
for every $t$ and $k$. Since the limits of the both sides exist as $k\to\infty$ we have

\begin{equation*}
\sigma_{t+1}^ 2 = \sum_{i=0}^ \infty\left(\prod_{j=0}^ {i-1} A_{t-j}\right) B_{t-i} \overset{\text{law}}{=} \sum_{i=0}^ \infty\left(\prod_{j=0}^ {i-1} A_{-j}\right) B_{-i} = \sigma_1^ 2.
\end{equation*}
Treating multidimensional distributions similarly concludes the proof.
\end{proof}
\end{theorem}

We show below that the stationary solution is unique in some class of processes.

\begin{kor}
\label{kor:unique-stat}
Suppose $\alpha_1 < 1$ and $\e(L_0^2)<\infty$. Then \eqref{arch} has a unique solution given by \eqref{uniquesolution} in the class of processes satisfying $\sup_{t\in\mathbb{Z}}\e (\sigma_t^2) < \infty$. 
\end{kor}
\begin{proof}
By Theorem \ref{theo:stationarity} \eqref{uniquesolution} provides a stationary solution. Hence it remains to prove the uniqueness. 
By \eqref{iteration} we have for every $t\in\mathbb{Z}$ and $k\in\{0,1,\hdots\}$ that
\begin{equation*}
\sigma_{t+1}^2 =  \left(\prod_{i=0}^k A_{t-i}\right) \sigma_{t-k}^ 2 + \sum_{i=0}^ k\left(\prod_{j=0}^ {i-1} A_{t-j}\right) B_{t-i}.
\end{equation*}
Suppose now that there exists two solutions $\sigma_{t}^2$ and $\tilde{\sigma}_t^2$ satisfying $\sup_{t\in\mathbb{Z}}\e(\sigma_t^2) < \infty$ and $\sup_{t\in\mathbb{Z}}\e(\tilde{\sigma}_t^2) < \infty$. Then
$$
|\sigma_{t+1}^2 - \tilde{\sigma}_{t+1}^2| \leq  \left(\prod_{i=0}^k A_{t-i}\right) \sigma_{t-k}^ 2 + \left(\prod_{i=0}^k A_{t-i}\right) \tilde{\sigma}_{t-k}^ 2.
$$
As both terms on the right-side converges in $L^1$ to zero by Lemma \ref{lemma:stationarity}, we observe that 
$$
\e|\sigma_{t+1}^2 - \tilde{\sigma}_{t+1}^2| = 0
$$
for all $t\in\mathbb{Z}$ which implies the result.
\end{proof}

\begin{rem}
We assumed that the liquidity $(L_{t})_{t\in \mathbb{Z}}$ is a strictly stationary sequence. Nevertheless, the results in this section can be obtained by assuming that $(L_{t})_{t\in \mathbb{Z}}$ is weakly stationary (i.e., we have the shift-invariance in time of the first and second moments of the process). That is, by assuming weak stationarity of the noise, we obtain weak stationarity of the volatility $(\sigma _{t} ^{2} )_{t\in \mathbb{Z}}$ in Theorem \ref{theo:stationarity}. We prefer to keep the assumption of strict stationarity because it is needed later to simplify the third and fourth order assumptions of Lemma \ref{lemma:consistency2} and also because our main examples of liquidities are strictly stationary processes (see Section \ref{examples})

\end{rem}

In the sequel, we consider the stationary solution $(\sigma^2_t)_{t\in\mathbb{Z}}$ given by Theorem \ref{theo:stationarity}. Therefore, we will always implicitly assume that 

$$\e(L_0^2) <\infty \mbox{ and  }\alpha_1<1.$$ 

In order to study covariance function of the solution \eqref{uniquesolution}, we need that the moments $\e (\sigma_t^4)$ exists. Necessary and sufficient conditions for this are given in the following lemma. 

\begin{lemma}
\label{thefourthmoment}
Suppose $\e(\epsilon_0^4) < \infty$. Then $\e(\sigma^4_0) < \infty$ if and only if $\alpha_1 < \frac{1}{\sqrt{\e(\epsilon_0^4)}}$.
\end{lemma}
\begin{proof}
Denote $\e(\epsilon_0^4) = C_\epsilon$ and $\e(L_0^2) = C_L$. By the definition \eqref{uniquesolution} of the strictly stationary solution
\begin{equation*}
\e(\sigma_{t+1}^4) = \e\left(\sum_{i=0}^{\infty} \left(\prod_{j=0}^{i-1} A_{t-j}\right)B_{t-i}\right)^2,
\end{equation*}
and since all the terms above are positive, both sides are simultaneously finite or infinite. Note also that, as the terms all positive, we may apply Tonelli's theorem to change the order of summation and integration obtaining

\begin{equation}
\label{diagonals}
\begin{split}
\e(\sigma_{t+1}^4) &= \sum_{i=0}^{\infty} \e\left(\left(\prod_{j=0}^{i-1} A_{t-j}\right)^2B_{t-i}^2\right)\\
&\ + \sum_{\mathclap{\begin{subarray}{c}
i,k=0\\ 
i \neq k
\end{subarray}}}^{\infty} \e \left(\left(\prod_{j=0}^{i-1} A_{t-j}\right)B_{t-i}\left(\prod_{j=0}^{k-1} A_{t-j}\right)B_{t-k}\right).
\end{split}
\end{equation}
Let us begin with the first term above. By independence, we obtain 

\begin{equation}
\label{firstanal}
\begin{split}
\sum_{i=0}^{\infty} \e\left(\left(\prod_{j=0}^{i-1} A_{t-j}^2\right)B_{t-i}^2\right) &= \sum_{i=0}^{\infty} \left(\prod_{j=0}^{i-1} \alpha_1^2 C_\epsilon\right)\e(B_{t-i}^2)\\
 &= \e(B_0^2)\sum_{i=0}^{\infty} (\alpha_1^{2} C_\epsilon)^i.
\end{split}
\end{equation}
Consequently, $\e(\sigma_0^4) < \infty$ implies $\alpha_1 < \frac{1}{\sqrt{C_\epsilon}}$, since it is the radius of convergence of the series above. For the converse, consider the latter term in \eqref{diagonals}. By Cauchy-Schwarz inequality we obtain

\begin{equation*}
\begin{split}
&\ \sum_{\mathclap{\begin{subarray}{c}
i,k=0\\ 
i \neq k
\end{subarray}}}^{\infty} \e \left(\left(\prod_{j=0}^{i-1} A_{t-j}\right)B_{t-i}\left(\prod_{j=0}^{k-1} A_{t-j}\right)B_{t-k}\right)\\
&\leq \sum_{\mathclap{\begin{subarray}{c}
i,k=0\\ 
i \neq k
\end{subarray}}}^{\infty}\sqrt{\e\left(\left(\prod_{j=0}^{i-1} A_{t-j}^2\right)B_{t-i}^2\right)}  \sqrt{\e\left(\left(\prod_{j=0}^{k-1} A_{t-j}^2\right)B_{t-k}^2\right)}\\
&=\e(B_0^2)\sum_{\mathclap{\begin{subarray}{c}
i,k=0\\ 
i \neq k
\end{subarray}}}^{\infty}  (\alpha_1^2 C_\epsilon)^{\frac{i}{2}} (\alpha_1^2 C_\epsilon)^{\frac{k}{2}},
\end{split}
\end{equation*}
where

\begin{equation*}
\begin{split}
\sum_{\mathclap{\begin{subarray}{c}
i,k=0\\ 
i \neq k
\end{subarray}}}^{\infty}  (\alpha_1^2 C_\epsilon)^{\frac{i}{2}} (\alpha_1^2 C_\epsilon)^{\frac{k}{2}} &<  \sum_{i=0}^\infty (\alpha_1 C_\epsilon^{\frac{1}{2}})^i\sum_{k=0}^\infty (\alpha_1  C_\epsilon^{\frac{1}{2}})^k.
\end{split}
\end{equation*}
Together with \eqref{firstanal} this shows that if $\alpha_1 < \frac{1}{\sqrt{C_\epsilon}}$, all the series are convergent and thus $\e(\sigma_0^4) < \infty$.
\end{proof}
\begin{rem}
As expected, in order to have finite moments of higher order we needed to pose more restrictive assumption $\alpha_1 < \frac{1}{\sqrt{\e(\epsilon_0^4)}} \leq 1$ as $\e (\epsilon_0^2) = 1$. For example, in the case of Gaussian innovations we obtain  the well-known condition $\alpha_1 < \frac{1}{\sqrt{3}}$ (see e.g. \cite{FZ} or \cite{Li}). An explicit expression of the fourth moment can be obtained when $L$ is the squared increment of fBm (see Lemma 4 in \cite{BTT}).
\end{rem}

\subsection{Computation of the model parameters}
\label{subsec:parameters}

In this section we compute the parameters $\alpha_{0}, \alpha _{1}, l_{1}$ in (\ref{arch}) by using the aucovariance functions of $X^{2}$ and $L$. To this end, we use an AR(1) characterization of the ARCH process. From this characterization, we derive, using an idea from \cite{vouti},  a Yule -Walker equation of quadratic form for the parameters, that we can solve explicitly.  This constitutes the basis of the construction of the estimators in the next section.  From \eqref{arch} it follows that if $(\sigma^2_t)_{t\in\mathbb{Z}}$ is stationary, then so is $(X^2_t)_{t\in\mathbb{Z}}$. In addition 
\begin{equation}
\label{X_t}
\begin{split}
X_t^2 &= \sigma_t^2\epsilon_t^2 - \sigma_t^2 + \alpha_0+\alpha_1X_{t-1}^2+l_1L_{t-1}\\
&= \alpha_0 + \alpha_1X_{t-1}^2+ \sigma_t^2(\epsilon_t^2-1) + l_1L_{t-1}.
\end{split}
\end{equation}
Now

\begin{equation*}
\e(X_t^2) = \alpha_0+ \alpha_1 \e(X_{t-1}^2) + l_1
\end{equation*}
and hence

\begin{equation}
\label{meanofX}
\e(X_t^2) = \frac{\alpha_0 + l_1}{1-\alpha_1}.
\end{equation}
Let us define an auxiliary process $(Y_t)_{t\in\mathbb{Z}}$ by

\begin{equation*}
Y_t = X_t^2 - \frac{\alpha_0+l_1}{1-\alpha_1}.
\end{equation*}
Now $Y$ is a zero-mean stationary process satisfying

\begin{equation}
\label{Y_t}
\begin{split}
Y_t &= \alpha_1Y_{t-1} +\alpha_0+\sigma_t^2(\epsilon_t^2-1) +  l_1L_{t-1} -\frac{\alpha_0+l_1}{1-\alpha_1} + \alpha_1\frac{\alpha_0+l_1}{1-\alpha_1}\\
&=  \alpha_1Y_{t-1} +\sigma_t^2(\epsilon_t^2-1) +  l_1(L_{t-1} -1).
\end{split}
\end{equation}
By denoting 
\begin{equation*}
Z_t = \sigma_t^2(\epsilon_t^2-1) + l_1(L_{t-1} -1) 
\end{equation*}
we may write

\begin{equation*}
Y_t = \alpha_1Y_{t-1} + Z_t
\end{equation*}
corresponding to the AR$(1)$ characterization (\cite{vouti}) of $Y_t$ for $0 <\alpha_1 <1$.\\ 
In what follows, we denote the autocovariance functions of $X^2$ and $L$ with $\gamma(n)=\mathbb{E} (X_{t} ^{2} X_{t+n} ^{2}) - ( \frac{\alpha_0 + l_1}{1-\alpha_1})^2$ and $s(n)= \mathbb{E} (L_{n} L_{t+n} ) - 1$ respectively.

\begin{lemma}
\label{lemma:equations}
Suppose $\e(\epsilon_0^4) < \infty$ and $\alpha_1 < \frac{1}{\sqrt{\e(\epsilon_0^4)}}$. Then for any $n\neq 0$ we have

\begin{equation}
\label{quadratic}
\alpha_1^2\gamma(n) - \alpha_1(\gamma(n+1) + \gamma(n-1)) + \gamma(n) - l_1^2s(n) = 0
\end{equation}
and for $n=0$ it holds that

\begin{equation}
\label{quadratic2}
\alpha_1^2 \gamma(0) - 2\alpha_1\gamma(1) + \gamma(0) - \frac{\e(X_0^4) Var(\epsilon_0^2)}{\e(\epsilon_0^4)} - l_1^2 s(0) =0.
\end{equation}

\begin{proof}
First we notice that
\begin{equation}
\label{fourthofX}
\e(X_0^4) = \e(\sigma_0^4\epsilon_0^4) = \e(\sigma_0^4)\e(\epsilon_0^4) < \infty
\end{equation}
by Lemma \ref{thefourthmoment}. Hence, the stationary processes $Y$ and $Z$ have finite second moments. Furthermore, the covariance of $Y$ coincides with the one of $X^2$. Applying Lemma 1 of \cite{vouti} we get 

\begin{equation*}
\alpha_1^2\gamma(n) - \alpha_1(\gamma(n+1) + \gamma(n-1)) + \gamma(n) - r(n) = 0
\end{equation*}
for every $n\in\mathbb{Z}$, where $r(\cdot)$ is the autocovariance function of $Z$. For $r(n)$ with $n\geq 1$ we obtain

\begin{equation}
\label{r(n)}
\begin{split}
r(n) &= \e(Z_1Z_{n+1})\\
&= \e[(\sigma_1^2(\epsilon_1^2-1) + l_1(L_0-1))(\sigma_{n+1}^2(\epsilon_{n+1}^2-1)+l_1(L_n-1))]\\
&= l_1^2\e[(L_0-1)(L_n-1)] = l_1^2s(n),
\end{split}
\end{equation}
since the sequences $(\epsilon_t)_{t\in\mathbb{Z}}$ and $(L_t)_{t\in\mathbb{Z}}$ are independent of each other, and $\epsilon_t$ is independent of $\sigma_s$ for $s\leq t$. By the same arguments, for $n=0$ we have

\begin{equation}
\label{r(0)}
\begin{split}
r(0) &= \e\left[\left(\sigma_1^2(\epsilon_1^2-1) + l_1(L_0-1)\right)^2\right]\\ 
&= \e\left[\sigma_1^4(\epsilon_1^2-1)^2\right] + l_1^2\e\left[(l_0-1)^2\right]\\
&= \e(\sigma_1^4) Var(\epsilon_0^2) + l_1^2 s(0).
\end{split}
\end{equation}
Now using \eqref{fourthofX} and $\gamma(-1) = \gamma(1)$ completes the proof.

\end{proof}
\end{lemma}


\noindent
Now, let first $n\in\mathbb{Z}$ with $n\neq0$. Then

\begin{align}
\label{twoquadratic2}
\alpha_1^2 \gamma(0) - 2\alpha_1\gamma(1) + \gamma(0) - \frac{\e(X_0^4) Var(\epsilon_0^2)}{\e(\epsilon_0^4)} - l_1^2 s(0) =0 \nonumber\\
\alpha_1^2\gamma(n) - \alpha_1(\gamma(n+1) + \gamma(n-1)) + \gamma(n) - l_1^2s(n) = 0.
\end{align}
From the first equation we get

\begin{equation*}
l_1^2 = \frac{1}{s(0)}\left(\alpha_1^2 \gamma(0) - 2\alpha_1 \gamma(1)+ \gamma(0) - \frac{\e(X_0^4) Var(\epsilon_0^2)}{\e(\epsilon_0^4)} \right).
\end{equation*}
Substitution to \eqref{twoquadratic2} yields

\begin{scriptsize}
\begin{equation*}
\alpha_1^2\left(\gamma(n) - \frac{s(n)}{s(0)}\gamma(0)\right) + \alpha_1\left(2 \frac{s(n)}{s(0)} \gamma(1) - (\gamma(n+1) + \gamma(n-1))\right) + \gamma(n) + \frac{s(n)}{s(0)} \left(\frac{\e(X_0^4) Var(\epsilon_0^2)}{\e(\epsilon_0^4)} - \gamma(0) \right) = 0
\end{equation*}
\end{scriptsize}
Let us denote $\bo_0 = [\gamma(n+1), \gamma(n), \gamma(n-1), \gamma(1), \gamma(0), \e(X_0^4)]$ and

\begin{equation}
\begin{split}
\label{quadraticterms3}
a_0(\bo_0) &= \gamma(n) - \frac{s(n)}{s(0)}\gamma(0)\\
b_0(\bo_0) &= 2 \frac{s(n)}{s(0)} \gamma(1) - (\gamma(n+1) + \gamma(n-1))\\
c_0(\bo_0) &= \gamma(n) + \frac{s(n)}{s(0)} \left(\frac{\e(X_0^4) Var(\epsilon_0^2)}{\e(\epsilon_0^4)} - \gamma(0) \right).
\end{split}
\end{equation}
Assuming that $a_0(\bo_0) \neq 0$ we have the following solutions for the model parameters $\alpha_1$ and $l_1$:

\begin{equation}
\label{alphaofgamma2}
\alpha_1(\bo_0) = \frac{-b_0(\bo_0) \pm \sqrt{b_0(\bo_0)^2 - 4a_0(\bo_0)c_0(\bo_0)}}{2a_0(\bo_0)}
\end{equation}
and

\begin{equation}
\label{lofgamma2}
l_1(\bo_0) = \sqrt{\frac{1}{s(0)}\left(\alpha_1(\bo_0)^2 \gamma(0) - 2\alpha_1(\bo_0) \gamma(1)+ \gamma(0) - \frac{\e(X_0^4) Var(\epsilon_0^2)}{\e(\epsilon_0^4)} \right)}.
\end{equation}
Finally, denoting $\mu = \e(X_0^2)$ and using \eqref{meanofX} we may write

\begin{equation}
\label{alpha0ofgamma2}
\alpha_0(\bo_0, \mu) = \mu(1-\alpha_1(\bo_0)) - l_1(\bo_0).
\end{equation}
Now, let $n_1, n_2 \in\mathbb{Z}$ with $n_1\neq n_2$ and $n_1,n_2\neq 0$. Then

\begin{align}
\label{twoquadratic}
\alpha_1^2\gamma(n_1) - \alpha_1(\gamma(n_1+1) + \gamma(n_1-1)) + \gamma(n_1) - l_1^2s(n_1)  &= 0\\
\alpha_1^2\gamma(n_2) - \alpha_1(\gamma(n_2+1) + \gamma(n_2-1)) + \gamma(n_2) - l_1^2s(n_2)  &= 0. \nonumber
\end{align}
Assuming that $n_2$ is chosen in such a way that $s(n_2) \neq 0$ we have

\begin{equation}
\label{l_1}
l_1^2 = \frac{\alpha_1^2\gamma(n_2) - \alpha_1(\gamma(n_2+1) + \gamma(n_2-1)) + \gamma(n_2)}{s(n_2)}.
\end{equation}
Substitution to \eqref{twoquadratic} yields

\begin{tiny}
\begin{equation*}
\alpha_1^2\left(\gamma(n_1)-\frac{s(n_1)}{s(n_2)}\gamma(n_2)\right) - \alpha_1\left(\gamma(n_1+1) + \gamma(n_1-1)- \frac{s(n_1)}{s(n_2)}\left(\gamma(n_2+1) + \gamma(n_2-1)\right)\right) + \gamma(n_1) - \frac{s(n_1)}{s(n_2)}\gamma(n_2) = 0.
\end{equation*}
\end{tiny}
Let us denote $\bo = [\gamma(n_1+1), \gamma(n_2+1), \gamma(n_1), \gamma(n_2), \gamma(n_1-1), \gamma(n_2-1)]$ and

\begin{align}
\label{quadraticterms2}
\begin{split}
a(\bo) &= \gamma(n_1) - \frac{s(n_1)}{s(n_2)}\gamma(n_2)\\
b(\bo) &=  \frac{s(n_1)}{s(n_2)}\left(\gamma(n_2+1) + \gamma(n_2-1)\right) - (\gamma(n_1+1) + \gamma(n_1-1)).\\
\end{split}
\end{align}
Assuming $a(\bo) \neq 0$ we obtain the following solutions for the model parameters $\alpha_1$ and $l_1$:

\begin{equation}
\label{alphaofgamma}
\alpha_1(\bo) = \frac{-b(\bo) \pm \sqrt{b(\bo)^2 - 4a(\bo)^2}}{2a(\bo)},
\end{equation} 
and

\begin{equation}
\label{lofgamma}
l_1(\bo) = \sqrt{\frac{\alpha_1^2(\bo)\gamma(n_2) - \alpha_1(\bo)(\gamma(n_2+1) + \gamma(n_2-1)) + \gamma(n_2)}{s(n_2)}}.
\end{equation}
Again, $\alpha_0$ is given by

\begin{equation}
\label{alpha0ofgamma}
\alpha_0(\bo, \mu) = \mu(1-\alpha_1(\bo)) - l_1(\bo).
\end{equation}

\begin{rem}
Note that here we assumed $s(n_2)\neq 0$ and $a(\bo) \neq 0$ which means that we choose $n_1,n_2$ in a suitable way. Notice however, that these assumptions are not a restriction. Firstly, the case where $s(n_2)=0$ for all $n_2\neq 0$ corresponds to the more simple case where $L$ is a sequence of uncorrelated random variables. Secondly, if $s(n_2)\neq 0$ and $a(\bo)=0$, the second order term vanishes and 
we get a linear equation for $\alpha_1$. For detailed discussion on this phenomena, we refer to \cite{vouti}.
\end{rem}

\begin{rem}
\label{rem:sign}
At first glimpse Equations \eqref{alphaofgamma2} and \eqref{alphaofgamma} may seem useless as one needs to choose between signs. However, it usually suffices to know additional values of the covariance of the noise (see \cite{vouti}). In particular, it suffices that $s(n) \to 0$ (see \cite{vouti2}).
\end{rem}

\section{Parameter estimation}
\label{sec:estimation}
In this section we discuss how to estimate the model parameters consistently from the observations provided that the covariance of the liquidity $L$ is known. Based on formulas for the parameters provided in Subsection \ref{subsec:parameters}, it suffices that the covariances of $X^2$ can be estimated consistently. 
\subsection{Consistency of autocovariance estimators}
Throughout this section we denote 
\begin{equation*}
f(t-s) = \e(L_tL_s) = Cov(L_t,L_s) + 1 = s(t-s)+1.
\end{equation*}

\begin{lemma}
\label{covariance}
Let $t,s \in\mathbb{Z}$. Then

\begin{equation*}
\e(\sigma_t^ 2 L_s) = \frac{\alpha_0}{1-\alpha_1} + l_1\sum_{i=0}^ \infty \alpha_1^ i f(t-s-i-1)
\end{equation*}
\begin{proof}
By \eqref{uniquesolution} and Fubini-Tonelli

\begin{equation*}
\begin{split}
\e(\sigma_t^ 2 L_s) &= \sum_{i=0}^\infty\left(\prod_{j=0}^{i-1} \alpha_1 \e(\epsilon_{t-1-j}^ 2)\right)  \e\left((\alpha_0+l_1L_{t-1-i})L_s\right)\\
&= \alpha_0 \sum_{i=0}^ \infty \alpha_1^ i + l_1\sum_{i=0}^ \infty \alpha_1^ i \e(L_{t-1-i} L_s)\\
&=  \alpha_0 \sum_{i=0}^ \infty \alpha_1^ i + l_1\sum_{i=0}^ \infty \alpha_1^i f(t-s-i-1),
\end{split}
\end{equation*}
where the series converges since $\alpha_1 <1$ and $\e(L_0^2)< \infty$.
\end{proof}

\end{lemma}

The following variant of the law of large number is needed for the proof of the consistency of the estimators.

\begin{lemma}
\label{lawoflarge}
Let $(U_1, U_2, ...)$ be a sequence of random variables with a mutual expectation. In addition, assume that $\mathrm{Var}(U_j) \leq C$ and $\left|\mathrm{Cov}(U_j,U_k)\right| \leq g(|k-j|)$, where $g(i)\to 0$ as $i\to\infty$. Then 

\begin{equation*}
\frac{1}{n} \sum_{k=1}^ n U_k \to \e(U_1)
\end{equation*}
in probability.
\begin{proof}

By Chebyshev's inequality
\begin{equation*}
\begin{split}
\p\left(\left|\frac{1}{n} \sum_{k=1}^ n U_k - \e(U_1)\right| > \varepsilon\right)  &\leq \frac{\mathrm{Var}\left( \sum_{k=1}^ n U_k\right)}{\varepsilon^2 n^2},
\end{split}
\end{equation*}
where

\begin{equation*}
\begin{split}
\mathrm{Var}\left(\sum_{k=1}^ n U_k\right) &= \sum_{k,j=1}^n \mathrm{Cov}\left(U_k, U_j\right)\\
&= \sum_{k=1}^n \mathrm{Var}(U_k) + 2 \sum_{k=1}^ n \sum_{j=1}^{k-1} \mathrm{Cov} (U_k,U_j)\\
&\leq nC + 2 \sum_{k=1}^ n \sum_{j=1}^{k-1}\left| \mathrm{Cov} (U_k,U_j)\right|.
\end{split}
\end{equation*}
Fix $\delta >0$. Then, there exists $N_\delta\in\mathbb{N}$ such that $g(|k-j|) <\delta$ whenever $|k-j| \geq N_\delta$. Note also that by Cauchy-Schwarz it holds that $\left|\mathrm{Cov}(U_k,U_j)\right| \leq C$. Assume that $n> N_\delta$. Now

\begin{equation*}
\begin{split}
\sum_{k=1}^ n \sum_{j=1}^{k-1}\left| \mathrm{Cov} (U_k,U_j)\right| &\leq \sum_{k=1}^n \sum_{j=1}^ {k-N_\delta} g(|k-j|) + \sum_{k=1}^n \sum_{j=k-N_\delta +1}^ {k-1}\mathclap{C}\\
&\leq n^2\delta + nN_\delta C.
\end{split}
\end{equation*}
Hence

\begin{equation*}
\p\left(\left|\frac{1}{n} \sum_{k=1}^ n U_k - \e(U_1)\right| > \varepsilon\right) \leq \frac{nC + 2n^ 2\delta + 2nN_\delta C}{\varepsilon^2 n^ 2} = \frac{2\delta}{\varepsilon^ 2} + \mathcal{O}\left(\frac{1}{n}\right)
\end{equation*}
concluding the proof, since $\delta$ was arbitrary small.
\end{proof}
\end{lemma}

\begin{rem}
Note that the convergence in Lemma \ref{lawoflarge} actually takes place also in $L^2$. However, to obtain consistency of our estimators, the convergence in probability suffices.
\end{rem}

Assume that $(X^2_1, X^2_2, \hdots ,X^2_N)$ is an observed series from an generalized ARCH process $(X_t)_{t\in\mathbb{Z}}$. We use the following estimator of the autocovariance function of $X_t^2$

\begin{equation*}
\hat{\gamma}_N(n) = \frac{1}{N} \sum_{t=1}^{N-n} \left(X^2_t-\bar{X^2}\right)\left(X^2_{t+n}-\bar{X^2}\right)\quad\text{for }n\geq 0,
\end{equation*}
where $\bar{X^2}$ is the sample mean of the observations. We show that the estimator above is consistent in two steps. Namely, we consider the sample mean and the term

\begin{equation*}
\frac{1}{N}\sum_{t=1}^{N-n} X^2_tX_{t+n}^2
\end{equation*}
separately. If the both terms are consistent, consistency of the autocovariance estimator follows.

\begin{lemma}
\label{lemma:consistency1}
Suppose $\e(\epsilon_0^4) < \infty$ and $s(t) = cov(L_0L_t) \to 0$ as $t\to\infty$. If $\alpha_1 < \frac{1}{\sqrt{\e(\epsilon_0^ 4)}}$, then the sample mean
\begin{equation*}
\hat{\mu}_N = \frac{1}{N}\sum_{t=1}^N X_t^2
\end{equation*}
converges in probability to $\e(X_0^2)$.
\begin{proof}
By Lemma \ref{lawoflarge} it suffices to show that $cov(X_1^2, X_{t+1}^2)$ converges to zero as $t$ tends to infinity. For simplicity, let us assume that $t\geq 2$. Now by fixing $k=t-1$ in \eqref{iteration} we have

\begin{equation*}
\begin{split}
X_{t+1}^ 2 &= \epsilon_{t+1}^2 \left(\left(\prod_{i=0}^{t-1} A_{t-i} \right) \sigma_1^2+ \sum_{i=0}^{t-1}\left(\prod_{j=0}^{i-1} A_{t-j}\right)B_{t-i}\right)\\ 
&= \epsilon_{t+1}^2 \left(\left(\prod_{i=0}^{t-2} A_{t-i} \right) \alpha_1X_1^2 + \sum_{i=0}^{t-1}\left(\prod_{j=0}^{i-1} A_{t-j}\right)B_{t-i}\right).
\end{split}
\end{equation*}
Hence

\begin{equation*}
\begin{split}
X_{t+1}^2X_1^2=\left(\prod_{i=0}^ {t-2} \alpha_1 \epsilon_{t-i}^ 2\right) \alpha_1 X_1^4 \epsilon_{t+1}^ 2 + \epsilon_{t+1}^ 2 X_1^2\sum_{i=0}^ {t-1}\left(\prod_{j=0}^ {i-1} \alpha_1 \epsilon_{t-j}^2\right)(\alpha_0+l_1L_{t-i}).
\end{split}
\end{equation*}
Taking expectations yields

\begin{equation*}
\begin{split}
\e(X_{t+1}^2X_1^ 2) = \alpha_1^ t \e(X_1^ 4) + \alpha_0\e(X_1^ 2) \sum_{i=0}^ {t-1} \alpha_1^ i + l_1\sum_{i=0}^ {t-1} \alpha_1^ i \e(X_1^2L_{t-i}).
\end{split}
\end{equation*}
By Lemma \ref{covariance}, and since $\alpha_1 < 1$ we obtain that

\begin{footnotesize}
\begin{equation*}
\begin{split}
\e(X_{t+1}^2X_1^ 2) &= \alpha_1^ t \e(X_0^ 4) + \alpha_0\e(X_0^ 2) \sum_{i=0}^ {t-1} \alpha_1^ i +l_1\sum_{i=0}^{t-1} \alpha_1^ i\left(\frac{\alpha_0}{1-\alpha_1 }+ l_1\sum_{j=0}^ \infty \alpha_1^ j f(i-t-j)\right)\\
&=  \alpha_1^ t \e(X_0^ 4) + \left(\alpha_0\e(X_0^ 2) +\frac{l_1\alpha_0}{1-\alpha_1}\right) \sum_{i=0}^ {t-1} \alpha_1^ i + l_1^ 2\sum_{i=0}^ {t-1} \sum_{j=0}^ \infty \alpha_1^ {i+j} f(i-t-j).
\end{split}
\end{equation*}
\end{footnotesize}
As $t$ tends to infinity

\begin{equation*}
\begin{split}
\lim_{t\to\infty} \e(X_{t+1}^2X_1^ 2) &= \frac{\alpha_0\e(X_0^ 2)}{1-\alpha_1} + \frac{l_1\alpha_0}{(1-\alpha_1)^ 2}+l_1^ 2\lim_{t\to\infty}\sum_{i=0}^ {t-1} \sum_{j=0}^ \infty \alpha_1^ {i+j} f(i-t-j)\\
&= \frac{\alpha_0^ 2+2\alpha_0 l_1}{(1-\alpha_1)^ 2} + l_1^ 2\lim_{t\to\infty}\sum_{i=0}^ {\infty} \sum_{j=0}^ \infty \alpha_1^ {i+j} f(i-t-j),
\end{split}
\end{equation*}
where we have used \eqref{meanofX} for expectation of $X_0^2$. Note that $f(t) = s(t) +1$. Hence, there exists $M >0$ such that for the terms  in the double sum it holds that

\begin{equation*}
\left|\alpha_1^ {i+j} f(i-t-j)\right| \leq M \alpha_1^ {i+j}\qquad\text{for every }i, j, t.
\end{equation*}
Thus we have a uniform integrable upper bound and consequently, dominated convergence theorem yields

\begin{equation*}
\lim_{t\to\infty}\sum_{i=0}^ {\infty} \sum_{j=0}^ \infty \alpha_1^ {i+j} f(i-t-j) = \sum_{i=0}^ \infty \sum_{j=0}^ {\infty} \alpha_1^ {i+j} = \frac{1}{(1-\alpha_1)^ 2}.
\end{equation*}
Finally, we may conclude that

\begin{equation*}
\lim_{t\to\infty} \e(X_{t+1}^2X_1^ 2) = \left(\frac{\alpha_0+l_1}{1-\alpha_1}\right)^ 2 = \e(X_1^ 2)^2.
\end{equation*}
\end{proof}
\end{lemma}

\begin{lemma}
\label{lemma:consistency2}
Suppose $\e(L_0^ 4)<\infty$ and $\e(\epsilon_0^ 8) <\infty$. In addition, assume that for every fixed $n, n_1$ and $n_2$ it holds that $cov(L_0,L_t)\to 0$, $cov(L_0L_n, L_{\pm t}) \to 0$ and $cov(L_0L_{n_1}, L_tL_{t+n_2})\to 0$ as $t\to\infty$. If $\alpha_1 < \frac{1}{\e(\epsilon_0^8)^{\frac{1}{4}}}$, then
\begin{equation*}
\frac{1}{N-n}\sum_{t=1}^{N-n} X^2_tX_{t+n}^2
\end{equation*}
converges in probability to $\e(X_0^2X_n^2)$ for every $n\in\mathbb{Z}$.
\begin{proof}
Again, by Lemma \ref{lawoflarge} it suffices to show that $cov(X_0^2X_n^2, X_t^2X_{t+n}^2)$ converges to zero as $t$ tends to infinity. Hence we assume that $t > n$. By \eqref{uniquesolution} 

\begin{footnotesize}
\begin{equation}
\label{whatsthelimit}
\begin{split}
\e(X_0^ 2X_n^ 2X_t^ 2X_{t+n}^ 2) &= \e \sum_{i_1=0}^ \infty  \sum_{i_2=0}^ \infty  \sum_{i_3=0}^ \infty  \sum_{i_4=0}^ \infty \left(\prod_{j=0}^ {i_1-1} A_{-1-j}\right)B_{-1-i_1}\epsilon_0^ 2 \left(\prod_{j=0}^ {i_2-1} A_{n-1-j}\right)B_{n-1-i_2}\epsilon_n^ 2\\ 
&\quad \left(\prod_{j=0}^ {i_3-1} A_{t-1-j}\right)B_{t-1-i_3}\epsilon_t^ 2 \left(\prod_{j=0}^ {i_4-1} A_{t+n-1-j}\right)B_{t+n-1-i_4}\epsilon_{t+n}^ 2.
\end{split}
\end{equation}
\end{footnotesize}
Since the summands are non-negative, we can take the expectation inside. Furthermore, by independence of the sequences $\epsilon_t$ and $L_t$ we observe
\begin{equation}
\label{whatsthelimit2}
\begin{split}
\e(X_0^ 2X_n^ 2X_t^ 2X_{t+n}^ 2) &= \sum_{i_1=0}^ \infty  \sum_{i_2=0}^ \infty  \sum_{i_3=0}^ \infty  \sum_{i_4=0}^ \infty \e\left(B_{-1-i_1} B_{n-1-i_2} B_{t-1-i_3} B_{t+n-1-i_4}\right)\\
&\e\left(\epsilon_0^ 2\epsilon_n^ 2\epsilon_t^ 2\epsilon_{t+n}^ 2\prod_{j=0}^ {i_1-1} A_{-1-j} \prod_{j=0}^ {i_2-1} A_{n-1-j} \prod_{j=0}^ {i_3-1} A_{t-1-j}\prod_{j=0}^ {i_4-1} A_{t+n-1-j}\right).
\end{split}
\end{equation}
Next we justify the use of the dominated convergence theorem in order to change the order of the summations and taking the limit. Consequently, it suffices to study the limits of the terms 
\begin{equation}
\label{twoterms}
\begin{split}
&\e\left(\epsilon_0^ 2\epsilon_n^ 2\epsilon_t^ 2\epsilon_{t+n}^ 2\prod_{j=0}^ {i_1-1} A_{-1-j} \prod_{j=0}^ {i_2-1} A_{n-1-j} \prod_{j=0}^ {i_3-1} A_{t-1-j}\prod_{j=0}^ {i_4-1} A_{t+n-1-j}\right)\cdot\\
&\quad \e\left(B_{-1-i_1} B_{n-1-i_2} B_{t-1-i_3} B_{t+n-1-i_4}\right).
\end{split}
\end{equation}
\textbf{Step 1: finding summable upper bound.}\\
First note that the latter term is bounded by a constant. Indeed,
by stationarity of $(B_t)_{t\in\mathbb{Z}}$ we can write
\begin{equation}
\label{latterterm}
\begin{split}
\e\left(B_{-i_1} B_{n-i_2} B_{t-i_3} B_{t+n-i_4}\right) &= \alpha_0^ 4 + 4\alpha_0^ 3l_1 + \alpha_0^ 2l_1^ 2\big(\e(L_{-i_1}L_{n-i_2}) + \e(L_{-i_1}L_{t-i_3})\\
& + \e(L_{-i_1}L_{t+n-i_4}) + \e(L_{n-i_2}L_{t-i_3}) + \e(L_{n-i_2}L_{t+n-i_4})\\
& + \e(L_{t-i_3}L_{t+n-i_4})\big) + \alpha_0l_1^ 3\big(\e(L_{-i_1}L_{n-i_2}L_{t-i_3})\\
& + \e(L_{-i_1}L_{t-i_3}L_{t+n-i_4}) + \e(L_{-i_1}L_{n-i_2}L_{t+n-i_4})\\ 
&+ \e(L_{n-i_2}L_{t-i_3}L_{t+n-i_4})\big) + l_1^ 4\e(L_{-i_1}L_{n-i_2}L_{t-i_3}L_{t+n-i_4}),
\end{split}
\end{equation}
which is bounded by a repeated application of Cauchy-Schwarz inequality and the fact that the fourth moment of $L_0$ is finite. \\
Consider now the first term in \eqref{twoterms}. First we recall the elementary fact
\begin{equation}
\label{holders}
1 = \e(\epsilon_0^ 2) \leq \sqrt{\e(\epsilon_0^ 4)} \leq \e(\epsilon_0^ 6)^{\frac{1}{3}} \leq \e(\epsilon_0^ 8)^{\frac{1}{4}} < \infty. 
\end{equation}
Next note that the first term in \eqref{twoterms} is bounded for every set of indices. Indeed, this follows from the independence of $\epsilon$ and the observation that we obtain terms up to power 8 at most. That is, terms of form $\epsilon_t^8$ and by assumption, $\e (\epsilon_t^8) < \infty$. Let now $n>0$. Then 

\begin{equation*}
\begin{split}
&\e\left(\epsilon_0^ 2\epsilon_n^ 2\epsilon_t^ 2\epsilon_{t+n}^ 2\prod_{j=0}^ {i_1-1} A_{-1-j} \prod_{j=0}^ {i_2-1} A_{n-1-j} \prod_{j=0}^ {i_3-1} A_{t-1-j}\prod_{j=0}^ {i_4-1} A_{t+n-1-j}\right)\\
&= \e\left(\epsilon_{t+n}^2\right)\e\left(\epsilon_0^ 2\epsilon_n^ 2\epsilon_t^ 2\prod_{j=0}^ {i_1-1} A_{-1-j} \prod_{j=0}^ {i_2-1} A_{n-1-j} \prod_{j=0}^ {i_3-1} A_{t-1-j}\prod_{j=0}^ {i_4-1} A_{t+n-1-j}\right)\\
&\leq \e\left(\epsilon_{t+n}^2\right)\e\left(\epsilon_{t}^4\right)\e\left(\epsilon_0^ 2\epsilon_n^ 2\prod_{j=0}^ {i_1-1} A_{-1-j} \prod_{j=0}^ {i_2-1} A_{n-1-j} \prod_{j=0}^ {i_3-1} A_{t-1-j}\prod_{\mathclap{\begin{subarray}{c}
j=0\\ 
j \neq n-1
\end{subarray}}}^ {i_4-1} A_{t+n-1-j}\right)\\
&\leq\e\left(\epsilon_{t+n}^2\right)\e\left(\epsilon_{t}^4\right)\e\left(\epsilon_{n}^6\right)\e\left(\epsilon_0^ 2\prod_{j=0}^ {i_1-1} A_{-1-j} \prod_{j=0}^ {i_2-1} A_{n-1-j} \prod_{\mathclap{\begin{subarray}{c}
j=0\\ 
j \neq t-1-n
\end{subarray}}}^ {i_3-1} A_{t-1-j}\prod_{\mathclap{\begin{subarray}{c}
j=0\\ 
j \neq n-1\\
j\neq t-1
\end{subarray}}}^ {i_4-1} A_{t+n-1-j}\right)\\
&\leq\e\left(\epsilon_{t+n}^2\right)\e\left(\epsilon_{t}^4\right)\e\left(\epsilon_{n}^6\right)\e\left(\epsilon_{0}^8\right)
\e\left(\prod_{j=0}^ {i_1-1} A_{-1-j} \prod_{\mathclap{\begin{subarray}{c}
j=0\\ 
j \neq n-1
\end{subarray}}}^ {i_2-1} A_{n-1-j} \prod_{\mathclap{\begin{subarray}{c}
j=0\\ 
j \neq t-1-n\\
j\neq t-1
\end{subarray}}}^ {i_3-1} A_{t-1-j}\prod_{\mathclap{\begin{subarray}{c}
j=0\\ 
j \neq n-1\\
j \neq t-1\\
j \neq t+n-1
\end{subarray}}}^ {i_4-1} A_{t+n-1-j}\right).
\end{split}
\end{equation*}
Computing similarly for $n=0$, using stationarity of $A$,  and observing that 
$$
1=\e(\epsilon_0^ 2) \leq \e(\epsilon_0^ 4) \leq \e(\epsilon_0^ 6) \leq \e(\epsilon_0^ 8) 
$$
we hence deduce 
\begin{equation}
\label{loose}
\begin{split}
&\e\left(\epsilon_0^ 2\epsilon_n^ 2\epsilon_t^ 2\epsilon_{t+n}^ 2\prod_{j=0}^ {i_1-1} A_{-1-j} \prod_{j=0}^ {i_2-1} A_{n-1-j} \prod_{j=0}^ {i_3-1} A_{t-1-j}\prod_{j=0}^ {i_4-1} A_{t+n-1-j}\right)\\
&\leq C\e\left(\prod_{j=0}^ {i_1-1} A_{-j} \prod_{j=0}^ {i_2-1} A_{n-j} \prod_{j=0}^ {i_3-1} A_{t-j}\prod_{j=0}^ {i_4-1} A_{t+n-j}\right),
\end{split}
\end{equation}
where $C$ is a constant. Moreover, by using similar arguments we observe 
\begin{footnotesize}
\begin{equation*}
\e\left(\prod_{j=0}^ {i_1-1} A_{-j} \prod_{j=0}^ {i_2-1} A_{n-j} \prod_{j=0}^ {i_3-1} A_{t-j}\prod_{j=0}^ {i_4-1} A_{t+n-j}\right) \leq \e\left(\prod_{j=0}^ {i_1-1} A_{-j} \prod_{j=0}^ {i_2-1} A_{-j} \prod_{j=0}^ {i_3-1} A_{-j}\prod_{j=0}^ {i_4-1} A_{-j}\right).
\end{equation*}
\end{footnotesize}
Combining all the estimates above, it thus suffices to prove that 
\begin{equation*}
\begin{split}
&\sum_{i_1=0}^ \infty  \sum_{i_2=0}^ \infty  \sum_{i_3=0}^ \infty  \sum_{i_4=0}^ \infty\e\left(\prod_{j=0}^ {i_1-1} A_{-j} \prod_{j=0}^ {i_2-1} A_{-j} \prod_{j=0}^ {i_3-1} A_{-j}\prod_{j=0}^ {i_4-1} A_{-j}\right)\\ 
&\leq4!  \sum_{i_4=0}^ \infty  \sum_{i_3=0}^{i_4}  \sum_{i_2=0}^{i_3}  \sum_{i_1=0}^{i_2}\e\left(\prod_{j=0}^ {i_1-1} A_{-j} \prod_{j=0}^ {i_2-1} A_{-j} \prod_{j=0}^ {i_3-1} A_{-j}\prod_{j=0}^ {i_4-1} A_{-j}\right) < \infty.
\end{split}
\end{equation*}
Now for $i_1\leq i_2 \leq i_3 \leq i_4$ we have

\begin{equation*}
\e\left(\prod_{j=0}^ {i_1-1} A_{-j} \prod_{j=0}^ {i_2-1} A_{-j} \prod_{j=0}^ {i_3-1} A_{-j}\prod_{j=0}^ {i_4-1} A_{-j}\right) = \alpha_1^ {i_1+i_2+i_3+i_4} \e(\epsilon_0^ 8)^ {i_1}\e(\epsilon_0^ 6)^ {i_2-i_1} \e(\epsilon_0^4)^{i_3-i_2}
\end{equation*}
which yields

\begin{equation*}
\begin{split}
&4!  \sum_{i_4=0}^ \infty  \sum_{i_3=0}^{i_4}  \sum_{i_2=0}^{i_3}  \sum_{i_1=0}^{i_2}\e\left(\prod_{j=0}^ {i_1-1} A_{-j} \prod_{j=0}^ {i_2-1} A_{-j} \prod_{j=0}^ {i_3-1} A_{-j}\prod_{j=0}^ {i_4-1} A_{-j}\right) \\
&= 4!  \sum_{i_4=0}^ \infty  \sum_{i_3=0}^{i_4}  \sum_{i_2=0}^{i_3}  \sum_{i_1=0}^{i_2} \alpha_1^ {i_1+i_2+i_3+i_4} \e(\epsilon_0^ 8)^ {i_1}\e(\epsilon_0^ 6)^ {i_2-i_1} \e(\epsilon_0^4)^{i_3-i_2}\\
&= 4! \sum_{i_4=0}^\infty  \alpha_1^{i_4} \sum_{i_3=0}^{i_4}  \left(\alpha_1 \e(\epsilon_0^ 4)\right)^ {i_3}     \sum_{i_2=0}^{i_3}     \left(\alpha_1  \frac{\e(\epsilon_0^ 6)}{\e(\epsilon_0^ 4)}\right)^ {i_2}     \sum_{i_1=0}^{i_2} \left(\alpha_1 \frac{\e(\epsilon_0^ 8)}{\e(\epsilon_0^ 6)}\right)^{i_1}.
\end{split}
\end{equation*}
Denote 

\begin{equation*}
a_1 = \alpha_1 \frac{\e(\epsilon_0^8)}{\e(\epsilon_0^6)}, \quad a_2 =  \alpha_1 \frac{\e(\epsilon_0^6)}{\e(\epsilon_0^4)} \quad\text{and}\quad a_3 = \alpha_1 \e(\epsilon_0^4). 
\end{equation*}
Then we need to show that
\begin{equation}
\label{quadruple}
\begin{split}
\sum_{i_4=0}^\infty  \alpha_1^{i_4} \sum_{i_3=0}^{i_4}  a_3^ {i_3}     \sum_{i_2=0}^{i_3}     a_2^ {i_2}     \sum_{i_1=0}^{i_2} a_1^{i_1} < \infty.
\end{split}
\end{equation} 
For this suppose first that $1 \notin S \coloneqq \{a_1, a_2, a_3, a_1a_2, a_2a_3, a_1a_2a_3\}.$ Then we are able to use geometric sums to obtain 

\begin{equation*}
 \sum_{i_1=0}^{i_2} a_1^{i_1} = \frac{1-a_1^{i_2+1}}{1-a_1}\qquad\text{for }a_1\neq 1.
\end{equation*}
Continuing like this in the iterated sums in \eqref{quadruple} we deduce 

\begin{equation*}
\begin{split}
\sum_{i_2=0}^{i_3} a_2^{i_2} (1-a_1^{i_2+1}) &= \sum_{i_2=0}^{i_3} a_2^{i_2} - a_1\sum_{i_2=0}^{i_3} (a_1a_2)^{i_2} = \frac{1-a_2^{i_3+1}}{1-a_2} - a_1\frac{1-(a_1a_2)^{i_3+1}}{1-a_1a_2},
\end{split}
\end{equation*}

\begin{equation*}
\sum_{i_3=0}^{i_4} a_3^{i_3} (1-a_2^{i_3+1}) = \frac{1-a_3^{i_4+1}}{1-a_3} - a_2\frac{1-(a_2a_3)^{i_4+1}}{1-a_2a_3},
\end{equation*}
and

\begin{equation*}
\sum_{i_3=0}^{i_4} a_3^{i_3} (1- (a_1a_2)^{i_3+1}) = \frac{1-a_3^{i_4+1}}{1-a_3} - a_1a_2\frac{1-(a_1a_2a_3)^{i_4+1}}{1-a_1a_2a_3}.
\end{equation*}
Consequently, it suffices that the following three series converge 
\begin{equation*}
\sum_{i_4=0}^\infty \alpha_1^{i_4}a_3^{i_4+1}, \quad \sum_{i_4=0}^\infty \alpha_1^{i_4} (a_2a_3)^{i_4+1}\quad\text{and}\quad \sum_{i_4=0}^\infty \alpha_1^{i_4} (a_1a_2a_3)^{i_4+1}
\end{equation*}
yielding constraints

\begin{equation*}
\alpha_1 < \frac{1}{\sqrt{\e(\epsilon_0^4)}}, \quad \alpha_1 < \frac{1}{\e(\epsilon_0^6)^{\frac{1}{3}}} \quad\text{and}\quad \alpha_1 <\frac{1}{\e(\epsilon_0^8)^{\frac{1}{4}}}.
\end{equation*}
However, these follow from the assumption $\alpha_1 <\frac{1}{\e(\epsilon_0^8)^{\frac{1}{4}}}$. Finally, 
if $1 \in S$ it simply suffices to replace $a_1,a_2, a_3$ with
\begin{equation*}
\tilde{a}_1 = \alpha_1 \left(\frac{\e(\epsilon_0^8)}{\e(\epsilon_0^6)}+\delta\right), \quad \tilde{a}_2 =  \alpha_1 \left(\frac{\e(\epsilon_0^6)}{\e(\epsilon_0^4)}+\delta\right) \quad\text{and}\quad \tilde{a}_3 = \alpha_1 \left(\e(\epsilon_0^4)+\delta\right) 
\end{equation*}
such that 
$$
1 \notin \{\tilde{a}_1, \tilde{a}_2, \tilde{a}_3, \tilde{a}_1\tilde{a}_2, \tilde{a}_2\tilde{a}_3, \tilde{a}_1\tilde{a}_2\tilde{a}_3\}.
$$
Choosing $\delta<0$ small enough the claim follows from the fact that the inequality $\alpha_1 <\frac{1}{\e(\epsilon_0^8)^{\frac{1}{4}}}$ is strict.\\
\textbf{Step 2: computing the limit of \eqref{whatsthelimit}.}\\
By step 1 we can apply dominated convergence theorem in \eqref{whatsthelimit}. For this let us analyze the limit behaviour of \eqref{twoterms}. For the latter term we use \eqref{latterterm}. By assumptions, we have e.g. the following identities:

\begin{align*}
\lim_{t\to\infty} \e(L_{t-i_3}L_{t+n-i_4}) &=1\\ 
\lim_{t\to\infty} \e(L_{-i_1} L_{t-i_3}L_{t+n-i_4}) &= f(n+i_3-i_4)\\
\lim_{t\to\infty} \e(L_{-i_1}L_{n-i_2}L_{t-i_3}L_{t+n-i_4}) &= f(n+i_1-i_2)f(n+i_3-i_4).
\end{align*} 
Therefore the limit of the latter term of \eqref{twoterms} is given by

\begin{scriptsize}
\begin{equation*}
\begin{split}
\lim_{t\to\infty}\e\left(B_{-i_1} B_{n-i_2} B_{t-i_3} B_{t+n-i_4}\right) &= \alpha_0^ 4 + 4\alpha_0^ 3l_1 + \alpha_0^ 2l_1^ 2\big(4+ f(n+i_1-i_2)+f(n+i_3-i_4)\big)\\
& +\alpha_0l_1^ 3\big(f(n+i_1-i_2) + f(n+i_3-i_4) + f(n+i_1-i_2)\\
&+ f(n+i_3-i_4)\big) +l_1^4f(n+i_1-i_2)f(n+i_3-i_4)\\
&=  (\alpha_0^ 2 + 2\alpha_0l_1+l_1^ 2f(n+i_1-i_2))(\alpha_0^2+2\alpha_0l_1+l_1^ 2f(n+i_3-i_4))
\end{split}
\end{equation*}
\end{scriptsize}
The first term of \eqref{twoterms} can be divided into two independent parts whenever $t$ is large enough. More precisely, for $t > \max \{n+i_3, i_4\}$, we have

\begin{equation*}
\begin{split}
&\e\left(\epsilon_0^ 2\epsilon_n^ 2\epsilon_t^ 2\epsilon_{t+n}^ 2\prod_{j=0}^ {i_1-1} A_{-1-j} \prod_{j=0}^ {i_2-1} A_{n-1-j} \prod_{j=0}^ {i_3-1} A_{t-1-j}\prod_{j=0}^ {i_4-1} A_{t+n-1-j}\right)\\ 
=&  \e\left(\epsilon_0^ 2\epsilon_n^ 2 \prod_{j=0}^ {i_1-1} A_{-1-j} \prod_{j=0}^ {i_2-1} A_{n-1-j}\right)\e\left(\epsilon_t^ 2\epsilon_{t+n}^ 2 \prod_{j=0}^ {i_3-1} A_{t-1-j} \prod_{j=0}^ {i_4-1} A_{t+n-1-j}\right)\\ 
=& \e\left(\epsilon_0^ 2\epsilon_n^ 2 \prod_{j=0}^ {i_1-1} A_{-1-j} \prod_{j=0}^ {i_2-1} A_{n-1-j}\right)\e\left(\epsilon_0^ 2\epsilon_{n}^ 2 \prod_{j=0}^ {i_3-1} A_{-1-j} \prod_{j=0}^ {i_4-1} A_{n-1-j}\right),
\end{split}
\end{equation*}
where the last equality follows from stationarity of $A_t$. Hence

\begin{footnotesize}
\begin{equation*}
\begin{split}
&\lim_{t\to\infty} \e(X_0^ 2X_n^ 2X_t^ 2X_{t+n}^ 2)\\
 =& \sum_{i_1=0}^ \infty  \sum_{i_2=0}^ \infty  \sum_{i_3=0}^ \infty  \sum_{i_4=0}^ \infty \e\left(\epsilon_0^ 2\epsilon_n^ 2 \prod_{j=0}^ {i_1-1} A_{-1-j} \prod_{j=0}^ {i_2-1} A_{n-1-j}\right)\e\left(\epsilon_0^ 2\epsilon_{n}^ 2 \prod_{j=0}^ {i_3-1} A_{-1-j} \prod_{j=0}^ {i_4-1} A_{n-1-j}\right)\cdot\\
& (\alpha_0^ 2 + 2\alpha_0l_1+l_1^ 2f(n+i_1-i_2))(\alpha_0^2+2\alpha_0l_1+l_1^ 2f(n+i_3-i_4)).
\end{split}
\end{equation*}
\end{footnotesize}
On the other hand, by \eqref{uniquesolution}

\begin{footnotesize}
\begin{equation*}
\begin{split}
\e(X_0^ 2X_n^ 2) &= \sum_{i_1 = 0}^\infty \sum_{i_2=0}^ \infty \e\left(\epsilon_0^ 2 \epsilon_n^ 2\prod_{j=0}^ {i_1-1}A_{-1-j} \prod_{j=0}^{i_2-1} A_{n-1-j}\right) \e\left((\alpha_0+l_1L_{-1-i_1})(\alpha_0 + l_1L_{n-1-i_2})\right)\\
&= \sum_{i_1 = 0}^\infty \sum_{i_2=0}^ \infty \e\left(\epsilon_0^ 2 \epsilon_n^ 2\prod_{j=0}^ {i_1-1}A_{-1-j} \prod_{j=0}^{i_2-1} A_{n-1-j}\right) (\alpha_0^ 2 + 2\alpha_0l_1 + l_1^ 2f(n+i_1-i_2)).
\end{split}
\end{equation*}
\end{footnotesize}
Consequently, we conclude that 

\begin{equation*}
\lim_{t\to\infty} \e(X_0^ 2X_n^ 2X_t^ 2X_{t+n}^ 2) = \e(X_0^2 X_n^ 2)^ 2
\end{equation*}
proving the claim.
\end{proof}
\end{lemma}

\begin{rem}
The assumptions of Lemma \ref{lemma:consistency2} cohere with the assumptions of Lemma \ref{lemma:consistency1}. Moreover, the assumptions made related to convergence of covariances are very natural. Indeed, we only assume that the (linear) dependencies within the process $L_t$ vanish over time. Examples of $L$ satisfying the required assumptions can be found in Section \ref{examples}.
\end{rem}

\subsection{Estimation of the model parameters}
Set, for $N\geq 1$,

\begin{equation*}
\hat{\mu}_{2,N} = \frac{1}{N}\sum_{t=1}^{N} X^4_t
\end{equation*}
and

\begin{equation*}
g_0(\bo_0) = b_0(\bo_0)^2 - 4a_0(\bo_0)c_0(\bo_0),
\end{equation*}
where $a_0(\bo_0)$, $b_0(\bo_0)$ and $c_0(\bo_0)$ are as in \eqref{quadraticterms3}. In addition, let 
\begin{equation*}
\bohh = [\ha(n+1), \ha(n), \ha(n-1), \ha(1), \ha(0), \hat{\mu}_{2,N}]
\end{equation*}
and $\hat{\bxi}_{0, N} = [\bohh, \hat{\mu}_N]$ for some fixed $n\neq 0$. The following estimators are motivated by \eqref{alphaofgamma2}, \eqref{lofgamma2} and \eqref{alpha0ofgamma2}. 

\begin{defi}\label{def1}
We define estimators $\hat{\alpha}_1$, $\hat{l}_1$ and $\hat{\alpha}_0$ for the model parameters $\alpha_1$, $l_1$ and $\alpha_0$ respectively through

\begin{equation}
\label{haaalpha1}
\hat{\alpha}_1 = \alpha_1(\bohh)=  \frac{-b_0(\bohh) \pm \sqrt{g_0(\bohh)}}{2a_0(\bohh)},
\end{equation}

\begin{footnotesize}
\begin{equation}
\label{haal1}
\hat{l}_1=l_1(\bohh) = \sqrt{\frac{1}{s(0)}\left(\alpha_1(\bohh)^2 \ha(0) - 2\alpha_1(\bohh) \ha(1)+ \ha(0) - \frac{\hat{\mu}_{2,N} Var(\epsilon_0^2)}{\e(\epsilon_0^4)} \right)}
\end{equation}
\end{footnotesize}
and

\begin{equation}
\label{haaalpha0}
\hat{\alpha}_0 = \alpha_0(\hat{\bxi}_{0, N}) = \hat{\mu}_N(1-\alpha_1(\bohh)) - l_1(\bohh),
\end{equation}
where $n \neq 0$.

\end{defi}

\begin{theorem}
\label{theo:consistency}
Assume that $a_0(\bo_0) \neq 0$ and $g_0(\bo_0) > 0$. Let the assumptions of Lemma \ref{lemma:consistency2} prevail. Then $\hat{\alpha}_1, \hat{l}_1$ and $\hat{\alpha}_0$ given by \eqref{haaalpha1}, \eqref{haal1} and \eqref{haaalpha0} are consistent.
\end{theorem}
\begin{proof}
Since the assumptions of Lemma \ref{lemma:consistency2} are satisfied, so are the assumptions of Lemma \ref{lemma:consistency1} implying that the autocovariance estimators, the mean and the second moment estimator of $X_t^2$ are consistent. The claim follows from the continuous mapping theorem.
\end{proof}

Let us denote

\begin{equation*}
g(\bo) = b(\bo)^2 - 4a(\bo)^2,
\end{equation*}
where $a(\bo)$ and $b(\bo)$ are as in \eqref{quadraticterms2}. In addition, let 
\begin{equation*}
\boh = [\ha(n_1+1), \ha(n_2+1), \ha(n_1), \ha(n_2), \ha(n_1-1), \ha(n_2-1)]
\end{equation*}
and $\hat{\bxi}_N = [\boh, \hat{\mu}_N]$ for some fixed $n_1,n_2\neq 0$ with $n_1 \neq n_2$. 
The following estimators are motivated by \eqref{alphaofgamma}, \eqref{lofgamma} and \eqref{alpha0ofgamma}. 

\begin{defi}\label{def2}
We define estimators $\hat{\alpha}_1$, $\hat{l}_1$ and $\hat{\alpha}_0$ for the model parameters $\alpha_1$, $l_1$ and $\alpha_0$ respectively through

\begin{equation}
\label{haalpha1}
\hat{\alpha}_1 = \alpha_1(\boh)=  \frac{-b(\boh) \pm \sqrt{g(\boh)}}{2a(\boh)},
\end{equation}

\begin{footnotesize}
\begin{equation}
\label{hal1}
\hat{l}_1=l_1(\boh) = \sqrt{\frac{\alpha_1^2(\boh)\ha(n_2) - \alpha_1(\boh)(\ha(n_2+1) + \ha(n_2-1)) + \ha(n_2)}{s(n_2)}}
\end{equation}
\end{footnotesize}
and

\begin{equation}
\label{haalpha0}
\hat{\alpha}_0 = \alpha_0(\hat{\bxi}_N) = \hat{\mu}_N(1-\alpha_1(\boh)) - l_1(\boh),
\end{equation}
where $n_1, n_2\neq 0$ and $n_1\neq n_2$.

\end{defi}

\begin{theorem}
\label{theo:consistency2}
Assume that $s(n_2) \neq 0, a(\bo) \neq 0$ and $g(\bo) > 0$. Let the assumptions of Lemma \ref{lemma:consistency2} prevail. Then $\hat{\alpha}_1, \hat{l}_1$ and $\hat{\alpha}_0$ given by \eqref{haalpha1}, \eqref{hal1} and \eqref{haalpha0} are consistent.
\end{theorem}
\begin{proof}
The proof is basically the same as with Theorem \ref{theo:consistency}.
\end{proof}

\begin{rem}
\begin{itemize}
\item Statements of Theorems \ref{theo:consistency} and \ref{theo:consistency2} hold true also when $g_0(\bo_0) = 0$ and $g(\bo) = 0$, but in these cases the estimators do not necessarily become real valued as the sample size grows. In comparison, in \cite{vouti} the estimators were forced to be real by using indicator functions.

\item The estimators from Definitions \ref{def1} and \ref{def2} are of course related. In practice (see the next section) we use those from Definition \ref{def1} while those from Definition \ref{def2} are needed just in case when we need  more information in order to choose the correct sign for $\hat\alpha _1$, see Remark \ref{rem:sign}.

\item Note that here we implicitly assumed that the correct sign can be chosen in $\hat\alpha_1$. However, this is not a restriction as discussed.
\end{itemize}
\end{rem}

\subsection{Examples}
\label{examples}
We will present several examples of stationary processes for which  our main result stated in Theorem \ref{theo:consistency} apply. Our examples are constructed as 
$$L_{t}:= \left( X_{t+1}- X_{t}\right) ^{2}, \mbox{ for every } t\in \mathbb{Z} $$
where  $(X_{t})_{t\in \mathbb{R}}$  is a stochastic process with stationary increments. We discuss below the case when $X$ is a continuous Gaussian process (the fractional Brownian motion), a continuous non-Gaussian process (the Rosenblatt process), or a jump process (the compensated Poisson process). 

\subsubsection{The fractional Brownian motion}

Let $X_{t}: = B ^{H}_{t}$ for every $t\in \mathbb{R} $ where $(B ^{H}_{t}) _{t\in \mathbb{R}} $ is a two-sided fractional Brownian motion with Hurst parameter $H\in (0,1)$. Recall that $ B^{H} $ is a centered Gaussian process with covariance
\begin{equation*}
\mathbb{E} (B_{t}B_{s})=\frac{1}{2} (\vert t\vert ^{2H}+ \vert s\vert^{2H} -\vert t-s\vert ^{2H} ), \hskip0.3cm s,t \in \mathbb{R}.
\end{equation*}
Let us verify that the conditions from Lemma \ref{lemma:consistency2} and Theorem \ref{theo:consistency} are satisfied by $L_{t}= (B ^{H}_{t+1}- B ^{H}_{t}) ^{2}$. First, notice that (see Lemma 2 in \cite{BTT}) that for $t\geq 1$  
\begin{equation*}
Cov( L_{0}, L_{t})=\mathbb{E}\left( (B^{H}_{1} ) ^{2} (B ^{H}_{t+1} -B ^{H}_{t} ) ^{2}  \right) -1= 2(r_{H}(t))^2 
\end{equation*} 
with
\begin{equation}
\label{rh}
r_{H}(t)= \frac{1}{2} \left[ (t+1) ^{2H}+(t-1) ^{2H}-2t ^{2H}\right]\to _{t\to \infty}0
\end{equation}
since $r_{H}(t) $ behaves as $t^{2H-2}$ for $t$ large.

Let us now turn to the third-order condition, i.e. $Cov (L_{0}L_{n}, L_{t})= \mathbb{E} (L_{0} L_{n} L_{t} )-\mathbb{E} (L_{0} L_{n}  ) \to 0$ as $t\to \infty$. We can suppose $n\geq 1$ is fixed and $t>n$. 

For any three centered Gaussian random variables $X_{1}, X_{2}, X_{3}$  with unit variance we have $\mathbb{E} ( X_{1} ^{2} X_{2} ^{2}) = 1+2(\mathbb{E}(X_{1}X_{2}) ) ^{2} $ and 
\begin{eqnarray*}
\mathbb{E} ( X_{1} ^{2} X_{2} ^{2}X_{3} ^{2})&=& 2 \left( (\mathbb{E} (X_{1} X_{2}))^{2} +  (\mathbb{E} (X_{1} X_{3}))^{2} +(\mathbb{E} (X_{2} X_{3}))^{2} \right)\\
&+& 4 \mathbb{E} (X_{1} X_{2}))\mathbb{E} (X_{1} X_{3}))\mathbb{E} (X_{2} X_{3}))+1\\
&=&\mathbb{E} ( X_{1} ^{2} X_{2} ^{2}) + 2 \left(   (\mathbb{E} (X_{1} X_{3}))^{2} +(\mathbb{E} (X_{2} X_{3}))^{2} \right) \\
&+& 4 \mathbb{E} (X_{1} X_{2}))\mathbb{E} (X_{1} X_{3}))\mathbb{E} (X_{2} X_{3})).
\end{eqnarray*}
By applying this formula to $X_{1}= B ^{H}_{1}, X_{2}= B^H_{n+1}- B ^{H}_{n}, X_{3}= B ^{H}_{t+1} -B ^{H}_{t}$, we find
 \begin{equation*}
Cov (L_{0}L_{n}, L_{t})=2 r_{H}(t) ^{2} + 2r_{H}(t-n) ^{2}+4r_{H}(n)r_{H}(t) r_{H}(t-n)  
\end{equation*}
where $r_{H}$ is given by (\ref{rh}). By (\ref{rh}), the above expression converges to zero as $t\to \infty$.

Similarly for the fourth-order condition, the formulas are more complex but we can verify by standard calculations that, for every $n_{1}, n_{2}\geq 1$ and for every $t> \max (n_{1}, n_{2})$, the quantity 
$$ \mathbb{E}(L_{0} L_{n_{1}}L_{t}L_{t+n_{2}}) - \mathbb{E}(L_{0} L_{n_{1}}) \mathbb{E} (L_{t}L_{t+n_{2}}) $$
can be expressed as a polynomial (without term of degree zero) in $r_{H}(t), r_{H}(t-n_{1}), r_{H}(t+n_{2}), r_{H}(t+n_{2} -n_{1}) $ with coefficients depending on $n_{1}, n_{2}$. The conclusion is obtained by (\ref{rh}).

\subsubsection{The compensated Poisson process}

Let $(N_{t})_{t\in \mathbb{R} }$ be a Poisson process with intensity $\lambda =1$. Recall that $N$ is a cadlag adapted stochastic process, with independent increments, such that for every $s<t$, the random variable $N_{t}-N_{s}$ follows a Poisson distribution with parameter $t-s$. Define the compensated Poisson process $ (\tilde{N}_{t}) _{t\in \mathbb{R} } $ by $\tilde{N}_{t}=N_{t}-t$ for every $t\in \mathbb{R}$ and let $L_{t}= (N_{t+1}-N_{t})^ {2}$.  Clearly $\mathbb{E}L_{t}= 1$ for every $t$ and, by the independence of the increments of $\tilde{N}$, we have that for $t$ large enough 
$$Cov(L_{0}, L_{t})= Cov( L_{0}L_{n}, L_{t}) = Cov (L_{0}L_{n_{1}}, L_{t}L_{t+n_{2}})=0, $$ so the conditions in Theorem \ref{theo:consistency} are fulfilled.

\subsubsection{The Rosenblatt process} 
The (one-sided) Rosenblatt process $(Z^{H}_{t}) _{t\geq 0} $ is a self-similar stochastic process with stationary increments  and long memory  in the second Wiener chaos, i.e. it can be expressed as a multiple stochastic integral of order two with respect to the Wiener process. The Hurst parameter $H$ belongs to $(\frac{1}{2}, 1)$ and it characterizes the main properties of the process.  Its representation is 
$$ Z ^{H}_{t}=  \int_{\mathbb{R} }\int_{\mathbb{R}} f_{H}(y_{1}, y_{2}) dW(y_{1})dW(y_{2}) $$
 where $(W(y))_{y\in \mathbb{R}}$ is Wiener process and $f_{H} $ is  deterministic function such that\\ $\int_{\mathbb{R}} \int_{\mathbb{R}} f_{H}(y_{1}, y_{2}) ^{2} dy_{1}dy_{2} <\infty$. See e.g. \cite{T} for a more complete exposition on the Rosenblatt process. The two-sided Rosenblatt process has been introduced in \cite{Cou}. In particular, it has the same covariance as the fractional Brownian motion, so $\mathbb{E}(L_{t})= \mathbb{E}(Z^{H}_{t+1} -Z^{H}_{t}) ^{2}=1$ for every $t$. The use of the Rosenblatt process can be motivated by the presence of the long-memory in the emprical data for liquidity in financial markets, see \cite{Tsuji}.

The computation of the quantities $Cov(L_{0}, L_{t}), Cov (L_{0}L_{n}, L_{t})$ and $ Cov (L_{0} L_{n_{1}}, L_{t} L_{t+n_{2}})$  requires rather technical tools from stochastic analysis including properties of multiple integrals and product formula which we prefer to avoid here. We only mention that the term $Cov(L_{0}, L_{t})$ can be written as $P( r_{H}(t), r_{H,1}(t) ) $ where $P$ is a polynomial without term of degree zero, $r_{H}$ is given by (\ref{rh}), while 
$$r_{H,1}(t)= \int_{0} ^{1} \int_{0}^{1} \int_{t}^{t+1} \int_{t} ^{t+1} du_{1}du_{2}du_{3}du_{4} \vert u_{1}-u_{2}\vert ^{H-1}\vert u_{2}-u_{3} \vert ^{H-1}\vert u_{3}-u_{4}\vert ^{H-1}\vert u_{4}-u_{1}\vert ^{H-1}.$$
Note that
$$r_{H,1}(t) = \int_{[0,1] ^{4}} du_{1}du_{2}du_{3}du_{4} \vert u_{1}-u_{2}\vert ^{H-1}\vert u_{2}-u_{3}+t \vert ^{H-1}\vert u_{3}-u_{4}\vert ^{H-1}\vert u_{4}-u_{1}+t\vert ^{H-1}.  $$ 
Since $\vert u_{1}-u_{2}\vert ^{H-1}\vert u_{2}-u_{3}+t \vert ^{H-1}\vert u_{3}-u_{4}\vert ^{H-1}\vert u_{4}-u_{1}+t\vert ^{H-1}$ converges to zero as $t\to \infty$ for every $u_{i}$ and since  this  integrand is bounded for $t$ large by $\vert u_{1}-u_{2}\vert ^{H-1}\vert u_{2}-u_{3} \vert ^{H-1}\vert u_{3}-u_{4}\vert ^{H-1}\vert u_{4}-u_{1}\vert ^{H-1}$, which is integrable over $[0,1] ^{4}$, we obtain, via the dominated convergence theorem, that $Cov(L_{0}, L_{t})\to _{t\to \infty} 0.$
Similarly, the quantities $Cov (L_{0}L_{n}, L_{t})$ and $ Cov (L_{0} L_{n_{1}}, L_{t} L_{t+n_{2}})$ can be also expressed as polynomials (without constant terms) of $r_{H}, r_{H, k}$, $k=1,2,3, 4$ where $$r_{H,k}(t) = \int_{A_{1}\times ...A_{2k}}  du_{1}...du_{2k}  \vert u_{1}- u_{2}\vert ^{H-1}...\vert u_{2k-1}-u_{2k}\vert ^{H-1} \vert u_{2k}-u_{1}\vert ^{H-1}, $$ where at least one set $A_{i}$ is $(t, t+1)$. Thus we may apply a similar argument as above.

\section{Simulations}
This section provides some visual illustrations of convergence of the estimators \eqref{haaalpha1}, \eqref{haal1} and \eqref{haaalpha0} with respect to different liquidities $(L_t)_{t\in\mathbb{Z}}$. \\
The general setting throughout the simulations is the following. The IID process $(\epsilon_t)_{t\in\mathbb{Z}}$ is assumed to be a sequence of standard normals. In this case the restriction given by Lemma \ref{lemma:consistency2} reads $\alpha_1 < \frac{1}{105^{\frac{1}{4}}} \approx 0.31$. The lag used is $n=1$ and the true values of the model parameters are $\alpha_0=1$, $\alpha_1 = 0.1$ and $l_1 = 0.5$. The used sample sizes are $N=100, N=1000, N=10000$ and $N=100000$. The initial $X_0^2$ is set to a value $1.7$. After the processes $L_t$ with $t=0, 1,...N-2$ and $\epsilon_t$ with $t=1,2,...N-1$ are simulated, the initial is used to generate $\sigma_1^2$ using \eqref{arch}. Together with $\epsilon_1$ this gives $X_1^2$, after which \eqref{arch} yields the sample $\{X_0^2, X_1^2,..., X^2_{N-1}\}$.\\
In the first three subsections simulation results of the generalized ARCH process with liquidity given by $L_t = (B^H_{t+1} - B^H_{t})^2$ are presented. The used Hurst indices are $H=\frac{1}{3}, H=\frac{2}{3}$ and $H=\frac{4}{5}$. In the fourth subsection the liquidity process is given by $L_t = (\tilde{N}_{t+1} - \tilde{N}_t)^2$, where $N_t$ is a compensated Poisson process with $\lambda = 1$. 

In all subsections the sample size $N$ is varied, and each setting is repeated $1000$ times to provide histograms of the estimates. Our simulations show that the behaviour of the limit distributions is close to Gaussian one, as $N$ increases. We also note that, since the estimators involve square roots, they may produce complex valued estimates. However, asymptotically the estimates become real. Throughout the simulations the complex valued estimates have been simply removed, although the percentage of complex values is computed in each setting. Finally, some illustrative tables are given in Appendix \ref{sec:A_tables}.

\subsection{Fractional Brownian motion with $H=\frac{1}{3}$.}
Histograms of the estimates of the model parameters corresponding to $L_t = (B^H_{t+1} - B^H_{t})^2$ with $H = \frac{1}{3}$ are provided in Figures \ref{fig:H=1per3-N=100}, \ref{fig:H=1per3-N=1000}, \ref{fig:H=1per3-N=10000} and \ref{fig:H=1per3-N=100000}. The used sample sizes were $N=100, N=1000, N=10000$ and $N= 100000$. The sample sizes $N=100$ and $N=1000$ resulted complex valued estimates in $44.2\%$ and $3.5\%$ of the simulations respectively, whereas with the larger sample sizes all the estimates were real.

\begin{figure}[H]
\centering
  \begin{subfigure}[t]{0.32\textwidth}
   \includegraphics[width=\textwidth]{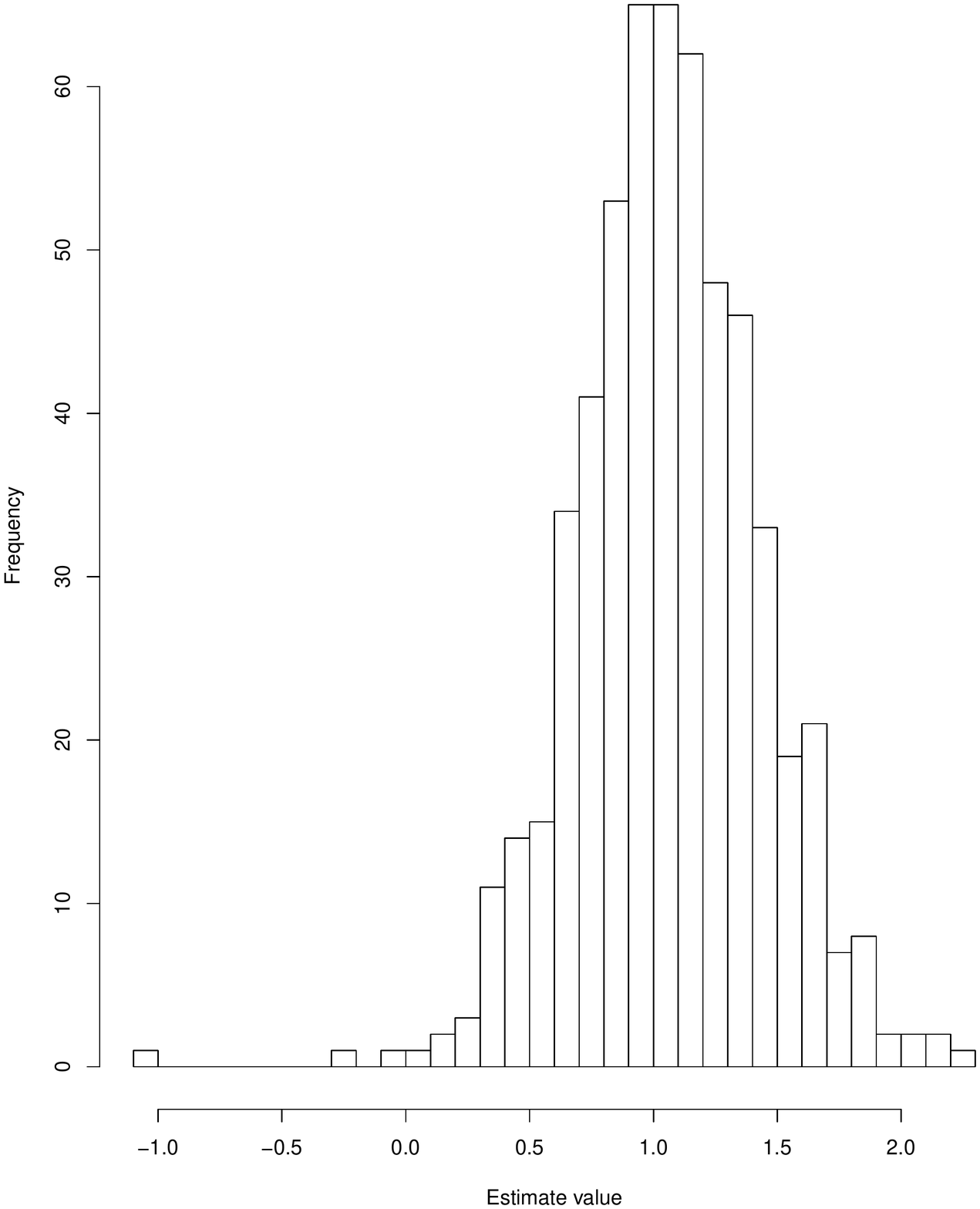}
    \caption{Estimates of $\alpha_0$.}
    \label{fig:H=1per3-N=100-alpha0}
  \end{subfigure}
 \begin{subfigure}[t]{0.32\textwidth}
    \includegraphics[width=\textwidth]{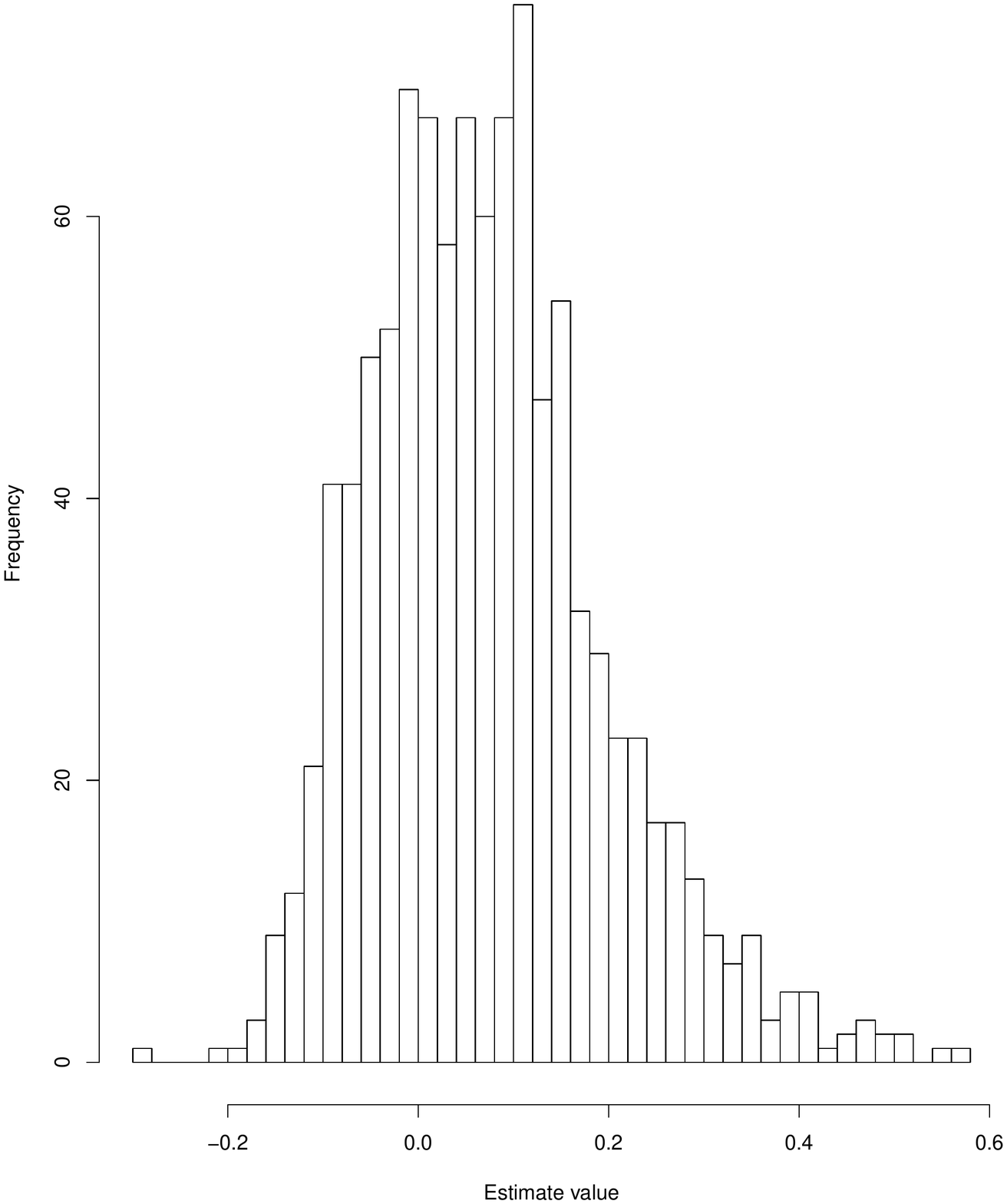}
    \caption{Estimates of $\alpha_1$.}
   \label{fig:H=1per3-N=100-alpha1}
 \end{subfigure}
 \begin{subfigure}[t]{0.32\textwidth}
    \includegraphics[width=\textwidth]{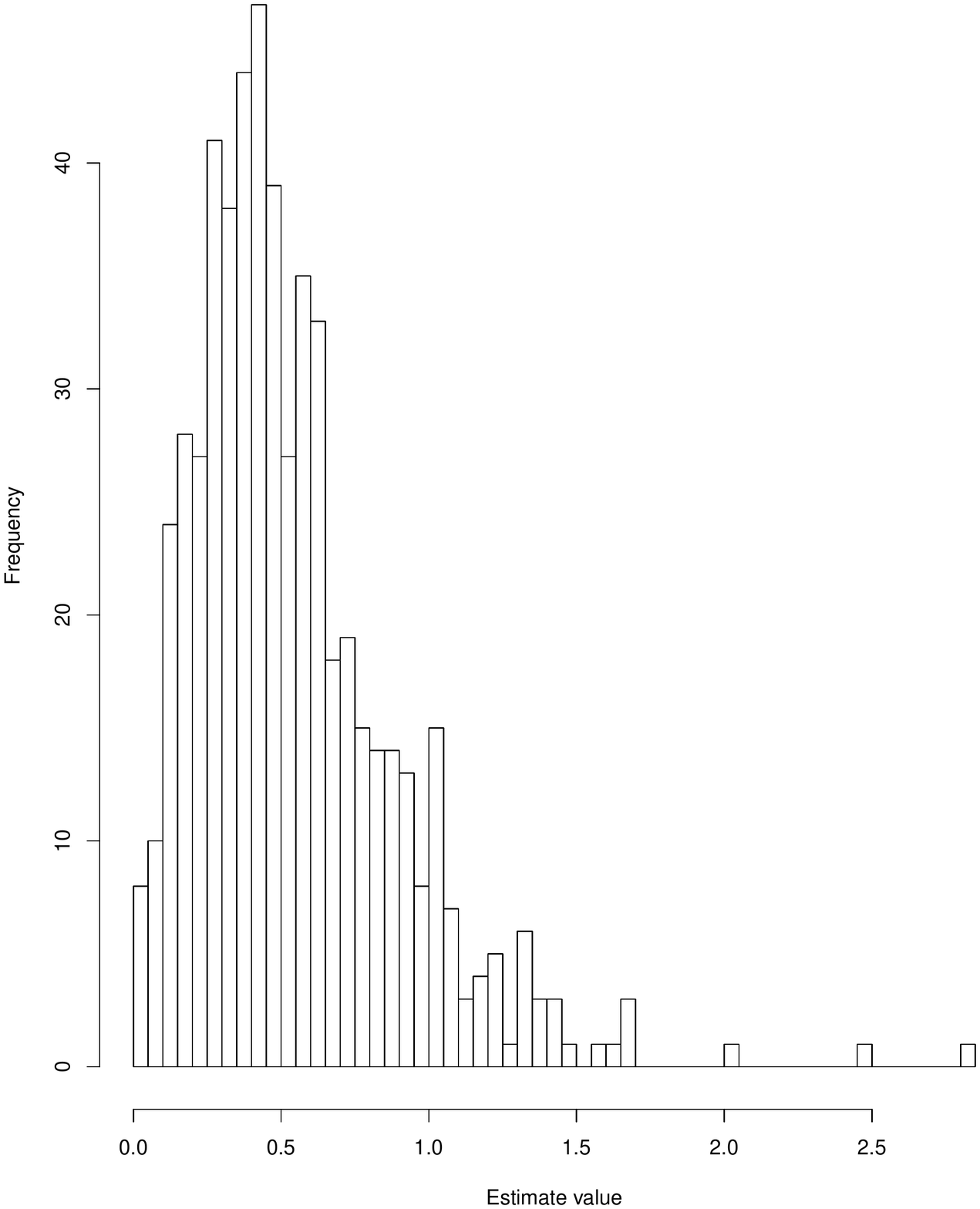}
    \caption{Estimates of $l_1$.}
   \label{fig:H=1per3-N=100-l1}
 \end{subfigure}
\caption{Fractional Brownian motion liquidity with $H=\frac{1}{3}$ and $N=100$.}
\label{fig:H=1per3-N=100}
\end{figure}

\begin{figure}[H]
\centering
  \begin{subfigure}[t]{0.32\textwidth}
   \includegraphics[width=\textwidth]{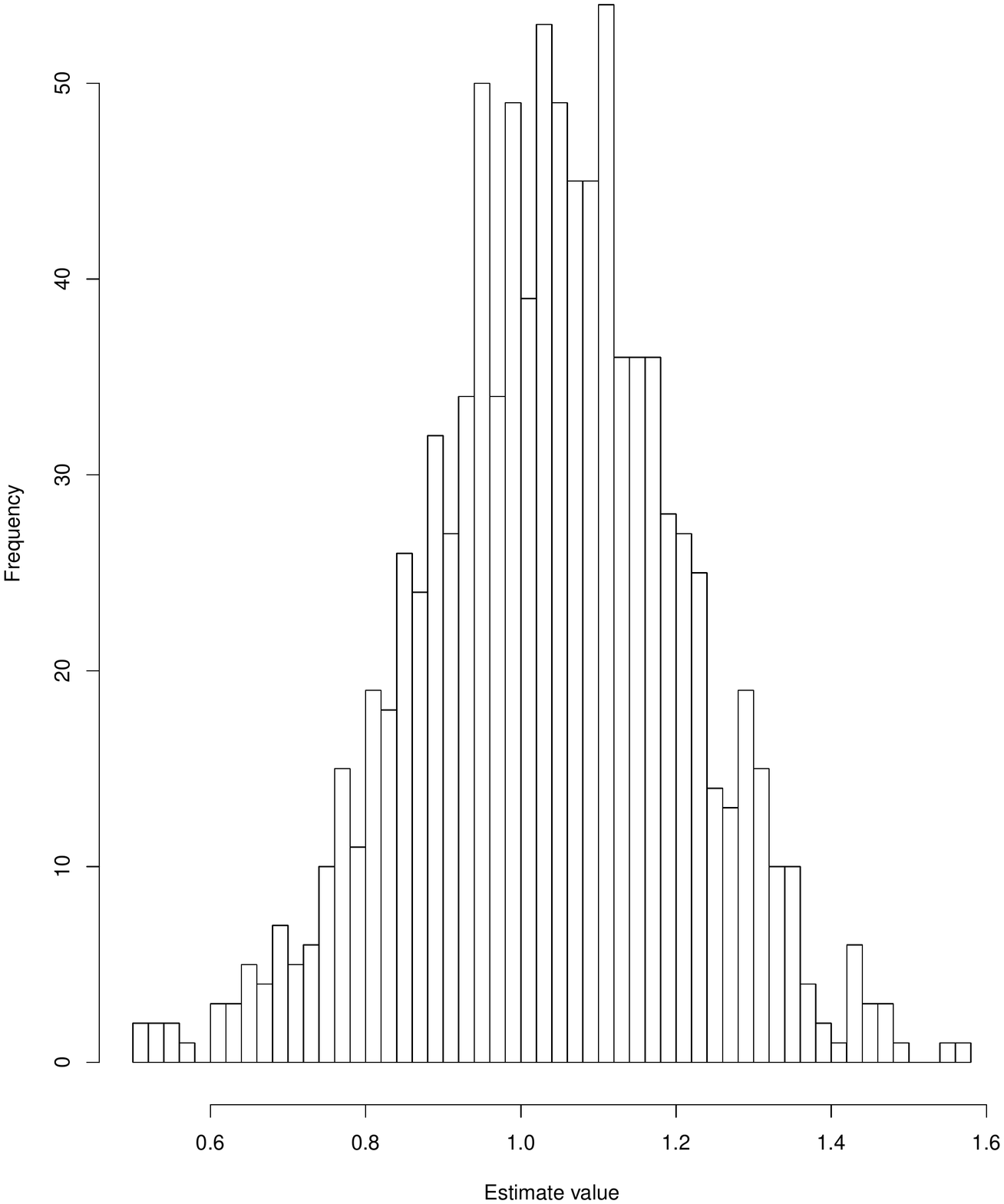}
    \caption{Estimates of $\alpha_0$.}
    \label{fig:H=1per3-N=1000-alpha0}
  \end{subfigure}
 \begin{subfigure}[t]{0.32\textwidth}
    \includegraphics[width=\textwidth]{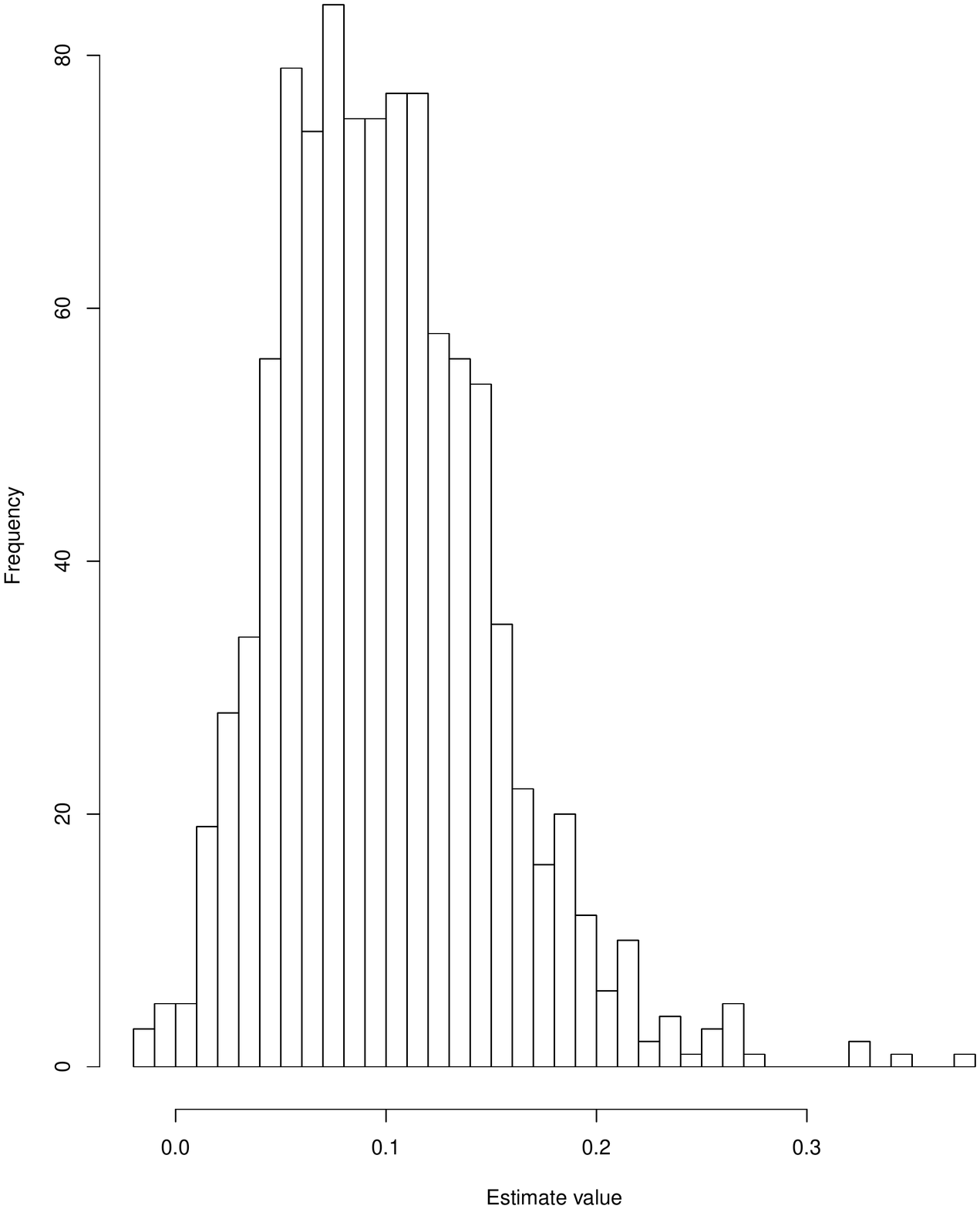}
    \caption{Estimates of $\alpha_1$.}
   \label{fig:H=1per3-N=1000-alpha1}
 \end{subfigure}
 \begin{subfigure}[t]{0.32\textwidth}
    \includegraphics[width=\textwidth]{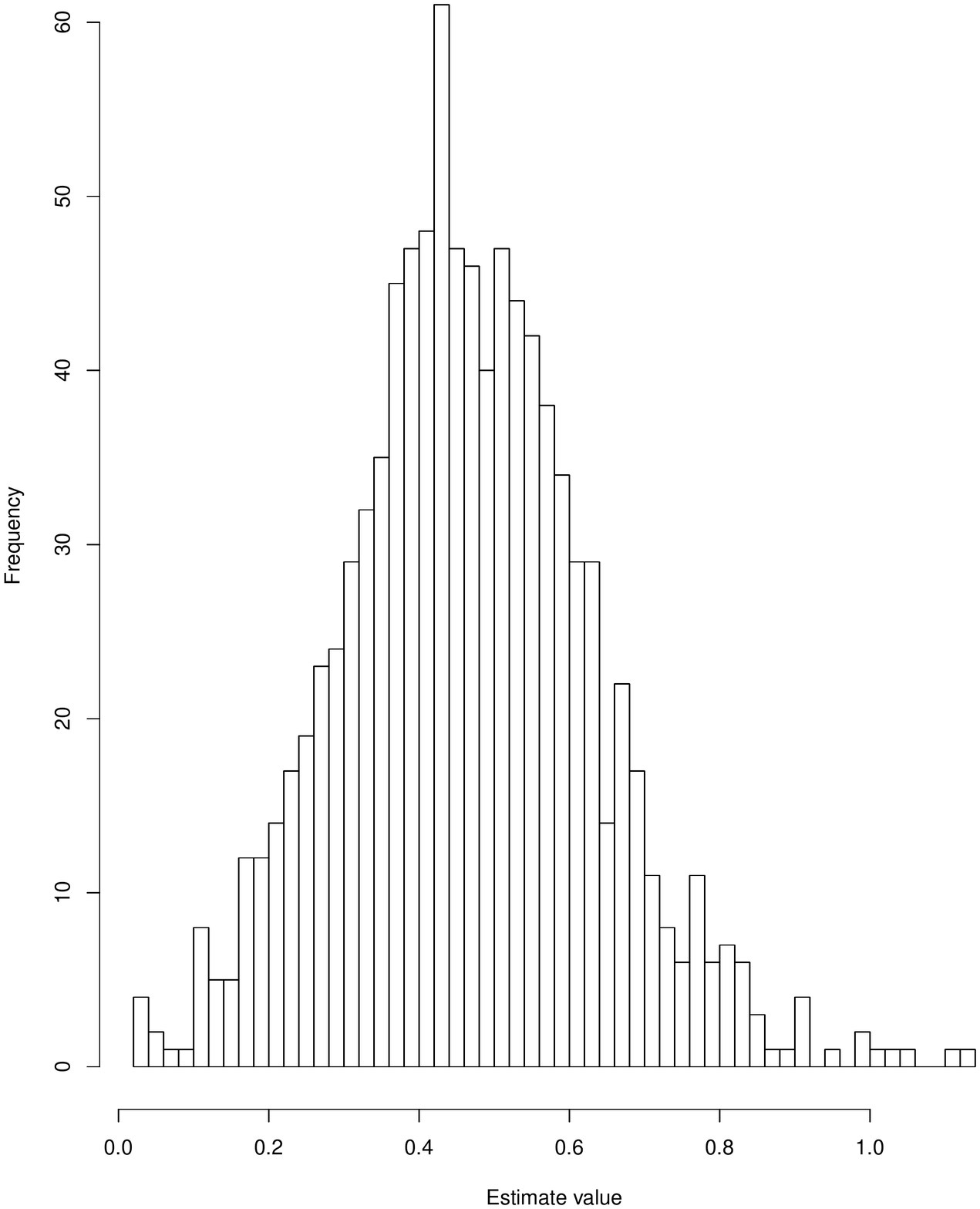}
    \caption{Estimates of $l_1$.}
   \label{fig:H=1per3-N=1000-l1}
 \end{subfigure}
\caption{Fractional Brownian motion liquidity with $H=\frac{1}{3}$ and $N=1000$.}
\label{fig:H=1per3-N=1000}
\end{figure}

\begin{figure}[H]
\centering
  \begin{subfigure}[t]{0.32\textwidth}
   \includegraphics[width=\textwidth]{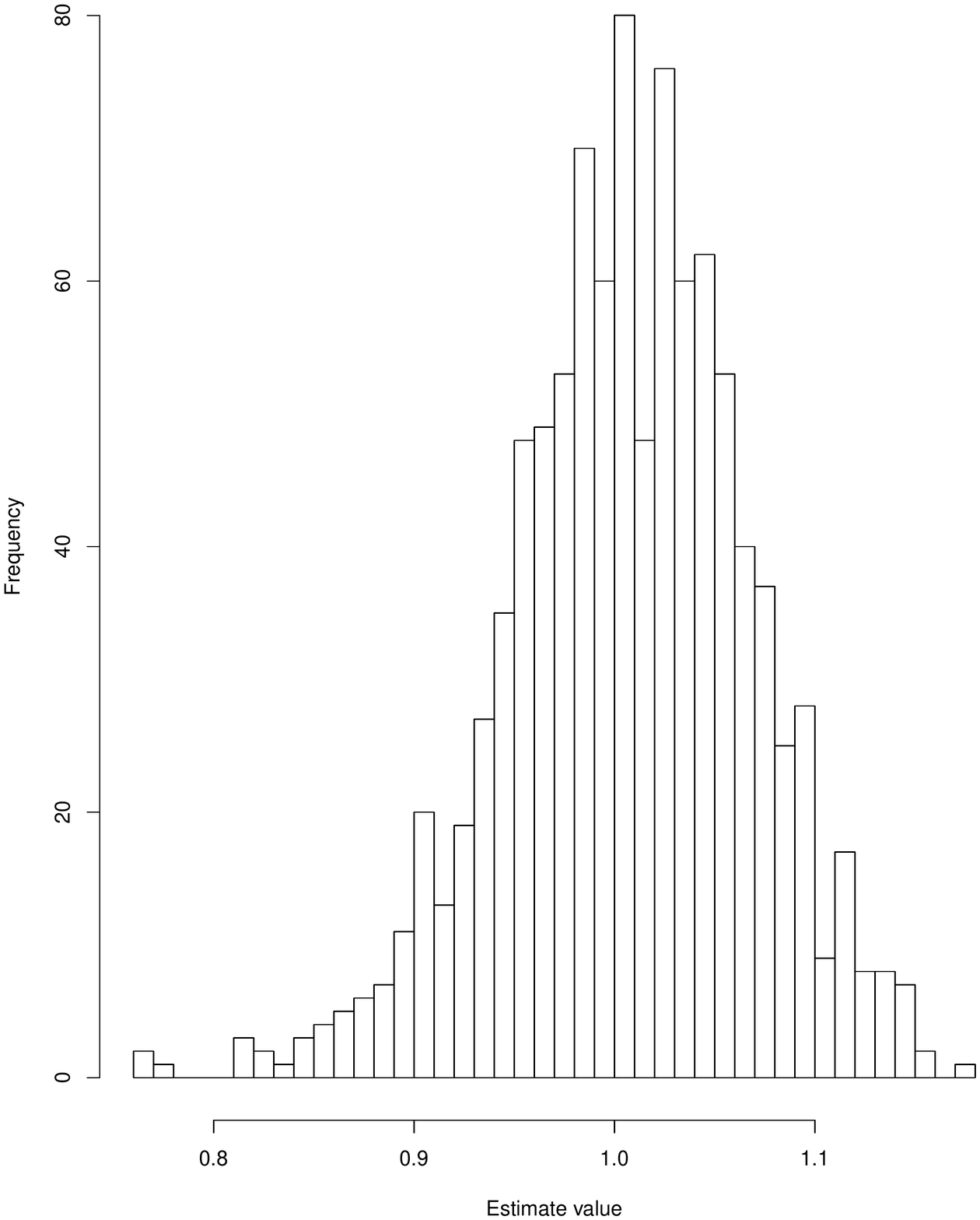}
    \caption{Estimates of $\alpha_0$.}
    \label{fig:H=1per3-N=10000-alpha0}
  \end{subfigure}
 \begin{subfigure}[t]{0.32\textwidth}
    \includegraphics[width=\textwidth]{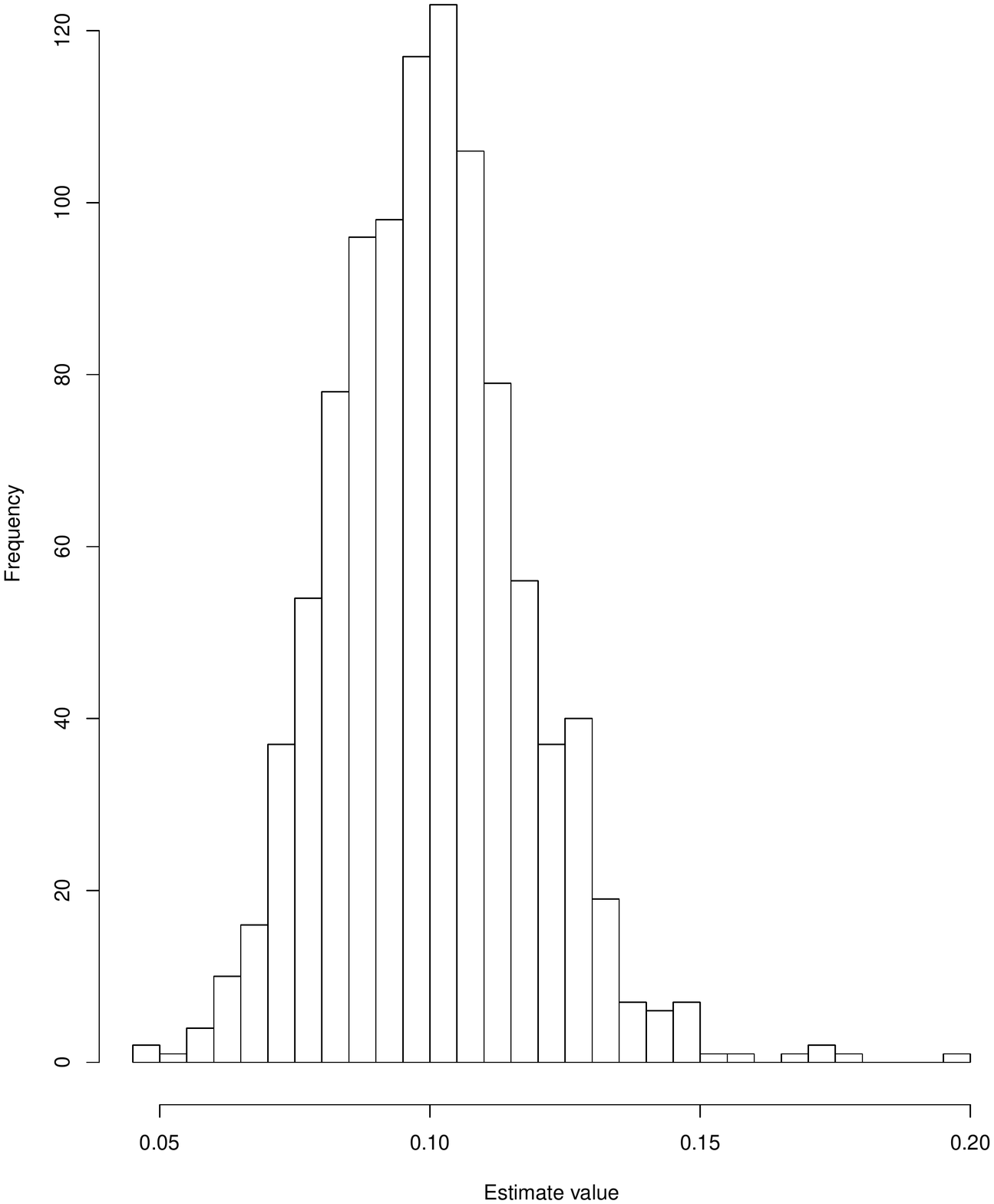}
    \caption{Estimates of $\alpha_1$.}
   \label{fig:H=1per3-N=10000-alpha1}
 \end{subfigure}
 \begin{subfigure}[t]{0.32\textwidth}
    \includegraphics[width=\textwidth]{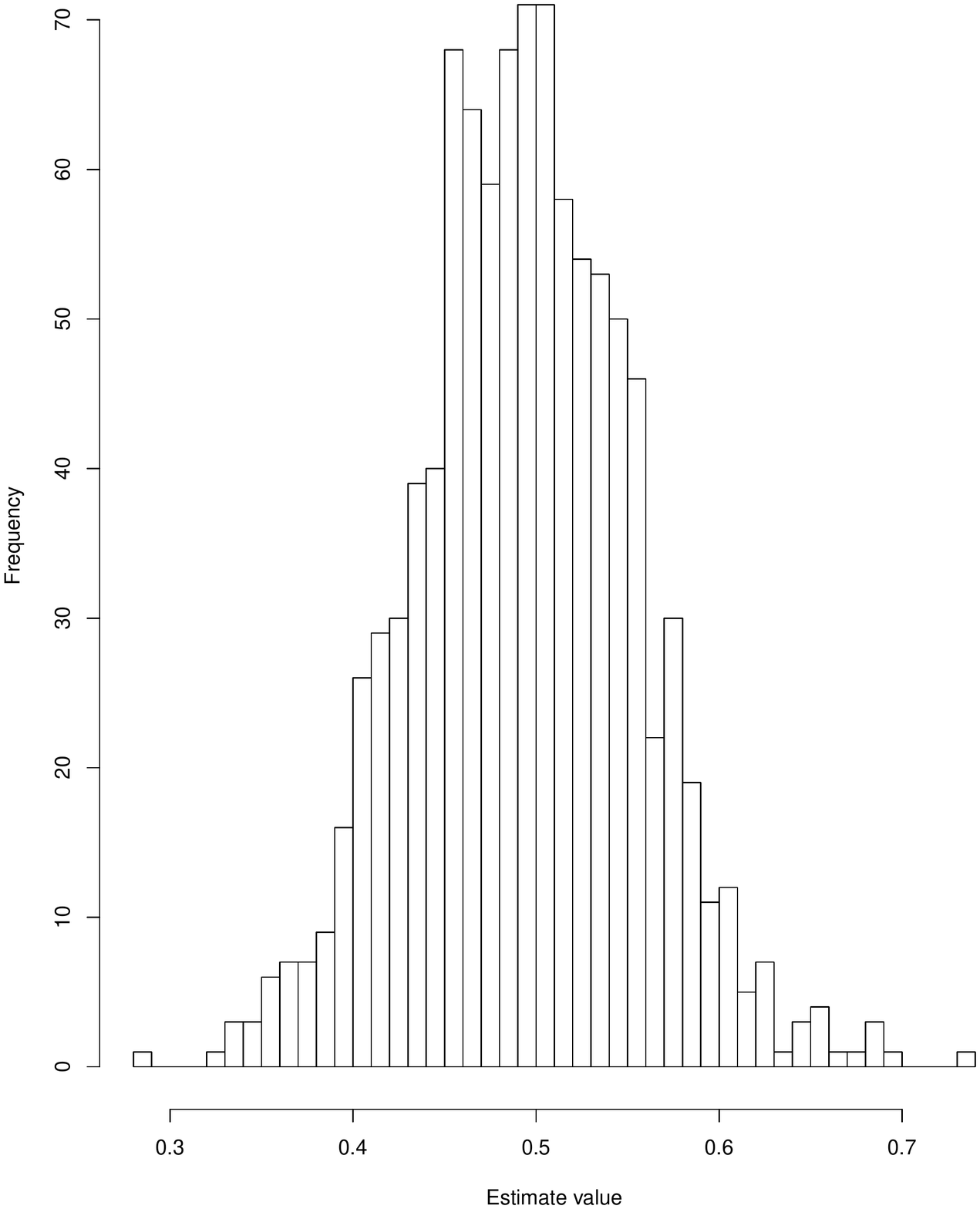}
    \caption{Estimates of $l_1$.}
   \label{fig:H=1per3-N=10000-l1}
 \end{subfigure}
\caption{Fractional Brownian motion liquidity with $H=\frac{1}{3}$ and $N=10000$.}
\label{fig:H=1per3-N=10000}
\end{figure}

\begin{figure}[H]
\centering
  \begin{subfigure}[t]{0.32\textwidth}
   \includegraphics[width=\textwidth]{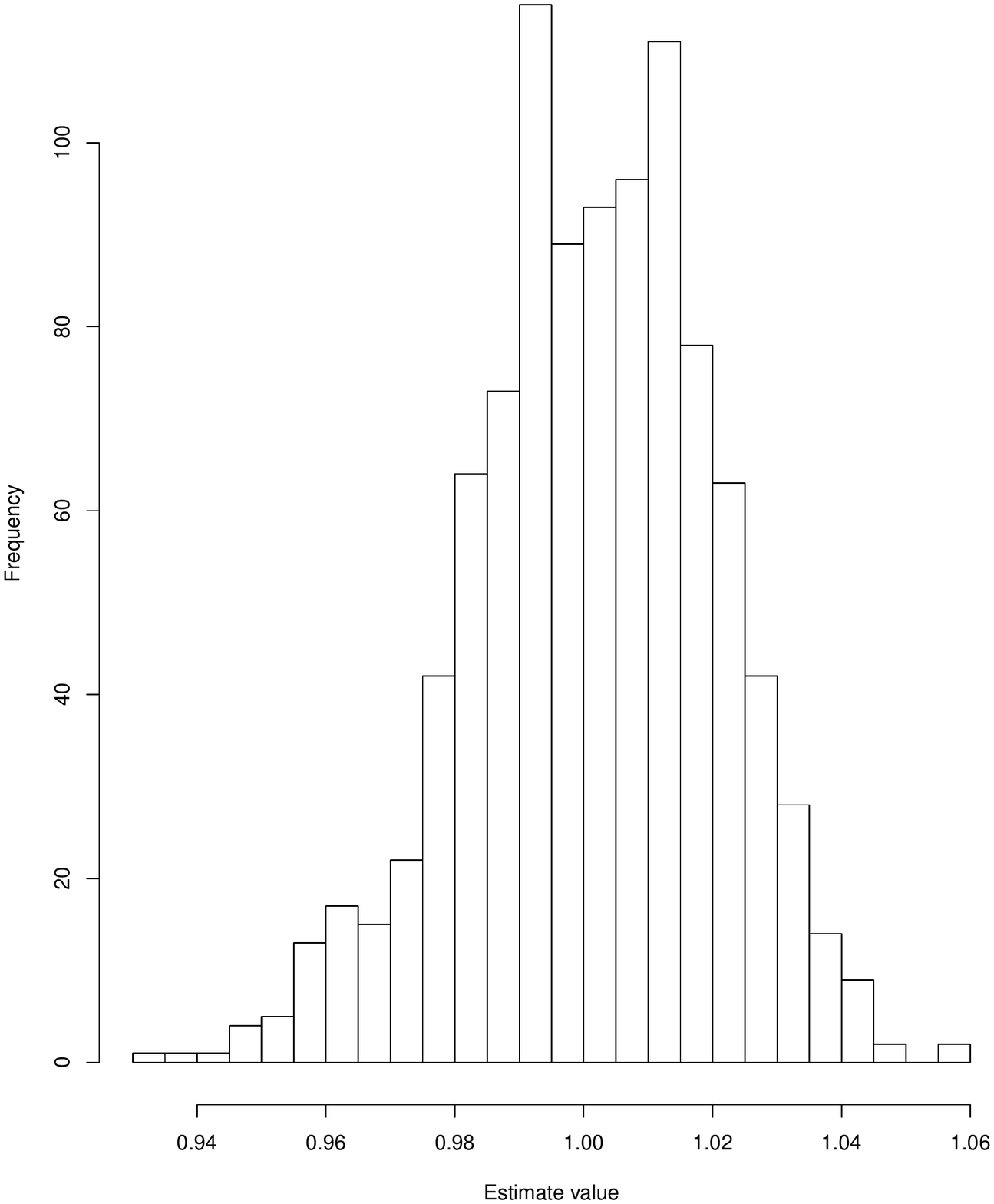}
    \caption{Estimates of $\alpha_0$.}
    \label{fig:H=1per3-N=100000-alpha0}
  \end{subfigure}
 \begin{subfigure}[t]{0.32\textwidth}
    \includegraphics[width=\textwidth]{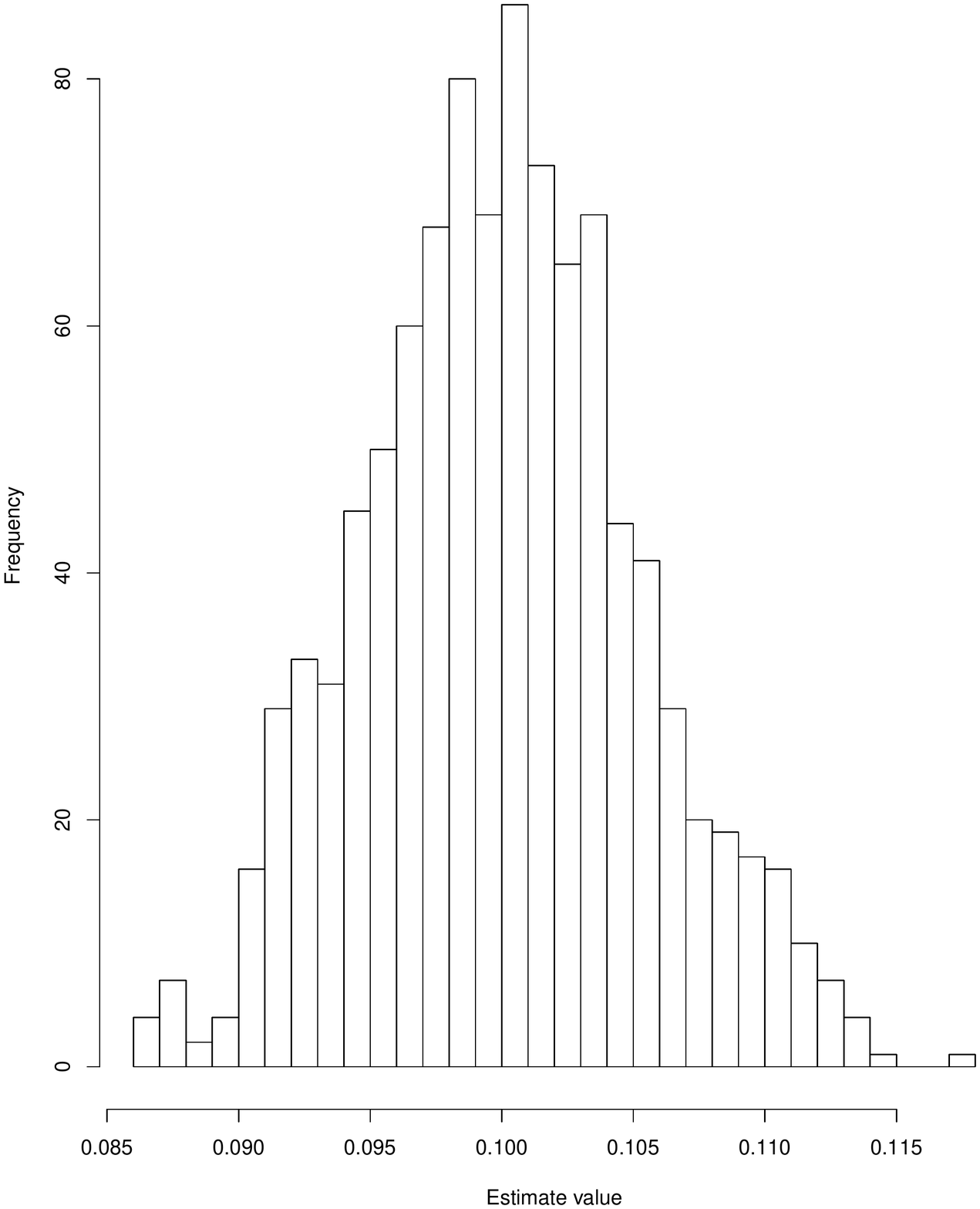}
    \caption{Estimates of $\alpha_1$.}
   \label{fig:H=1per3-N=100000-alpha1}
 \end{subfigure}
 \begin{subfigure}[t]{0.32\textwidth}
    \includegraphics[width=\textwidth]{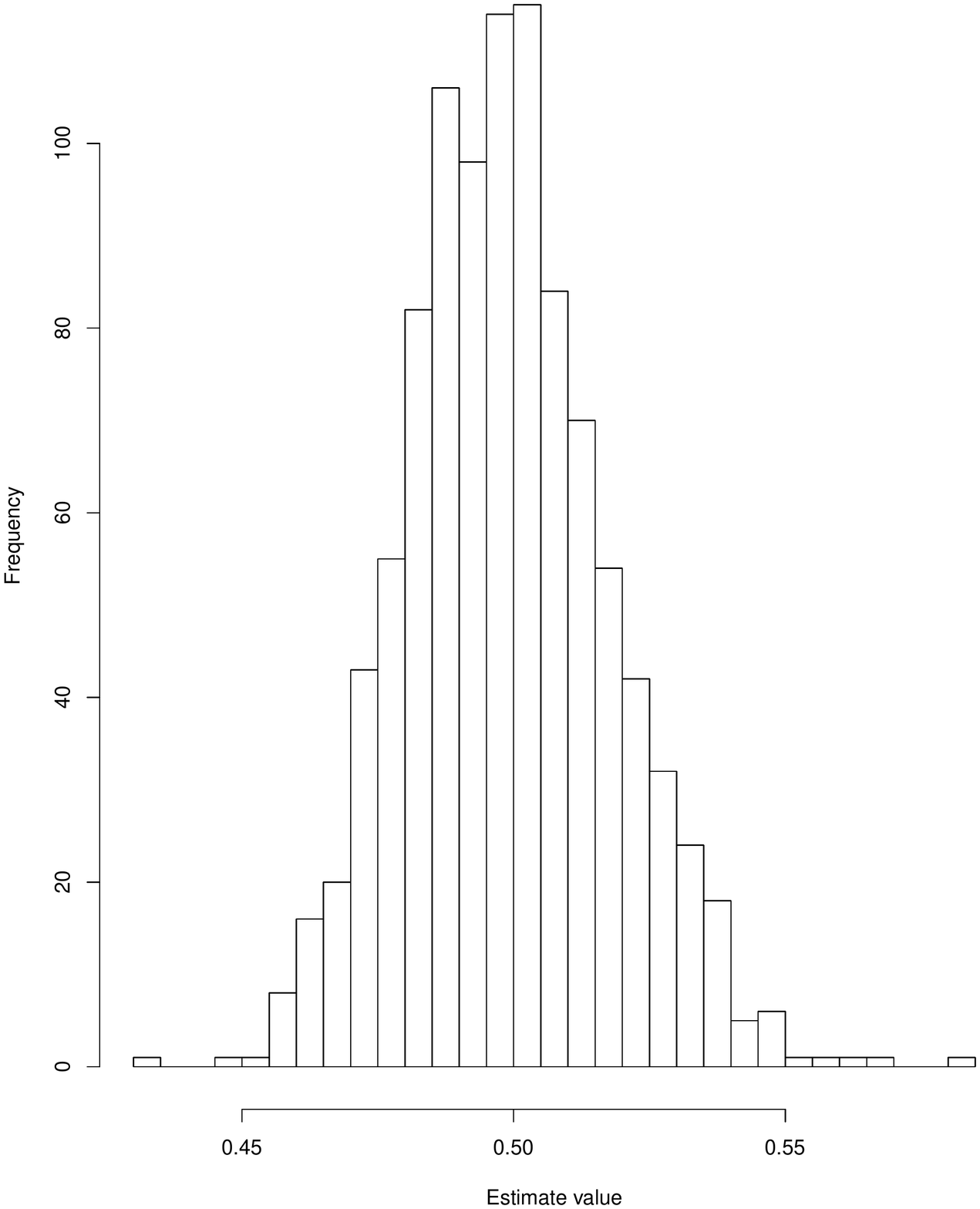}
    \caption{Estimates of $l_1$.}
   \label{fig:H=1per3-N=100000-l1}
 \end{subfigure}
\caption{Fractional Brownian motion liquidity with $H=\frac{1}{3}$ and $N=100000$.}
\label{fig:H=1per3-N=100000}
\end{figure}

\subsection{Fractional Brownian motion with $H=\frac{2}{3}$.}
Histograms of the estimates of the model parameters corresponding to $L_t = (B^H_{t+1} - B^H_{t})^2$ with $H = \frac{2}{3}$ are provided in Figures \ref{fig:H=2per3-N=100}, \ref{fig:H=2per3-N=1000}, \ref{fig:H=2per3-N=10000} and \ref{fig:H=2per3-N=100000}. The used sample sizes were $N=100, N=1000, N=10000$ and $N= 100000$. The sample sizes $N=100$ and $N=1000$ resulted complex valued estimates in $45.5\%$ and $2.9\%$ of the simulations respectively, whereas with the larger sample sizes all the estimates were real.

\begin{figure}[H]
\centering
  \begin{subfigure}[t]{0.32\textwidth}
   \includegraphics[width=\textwidth]{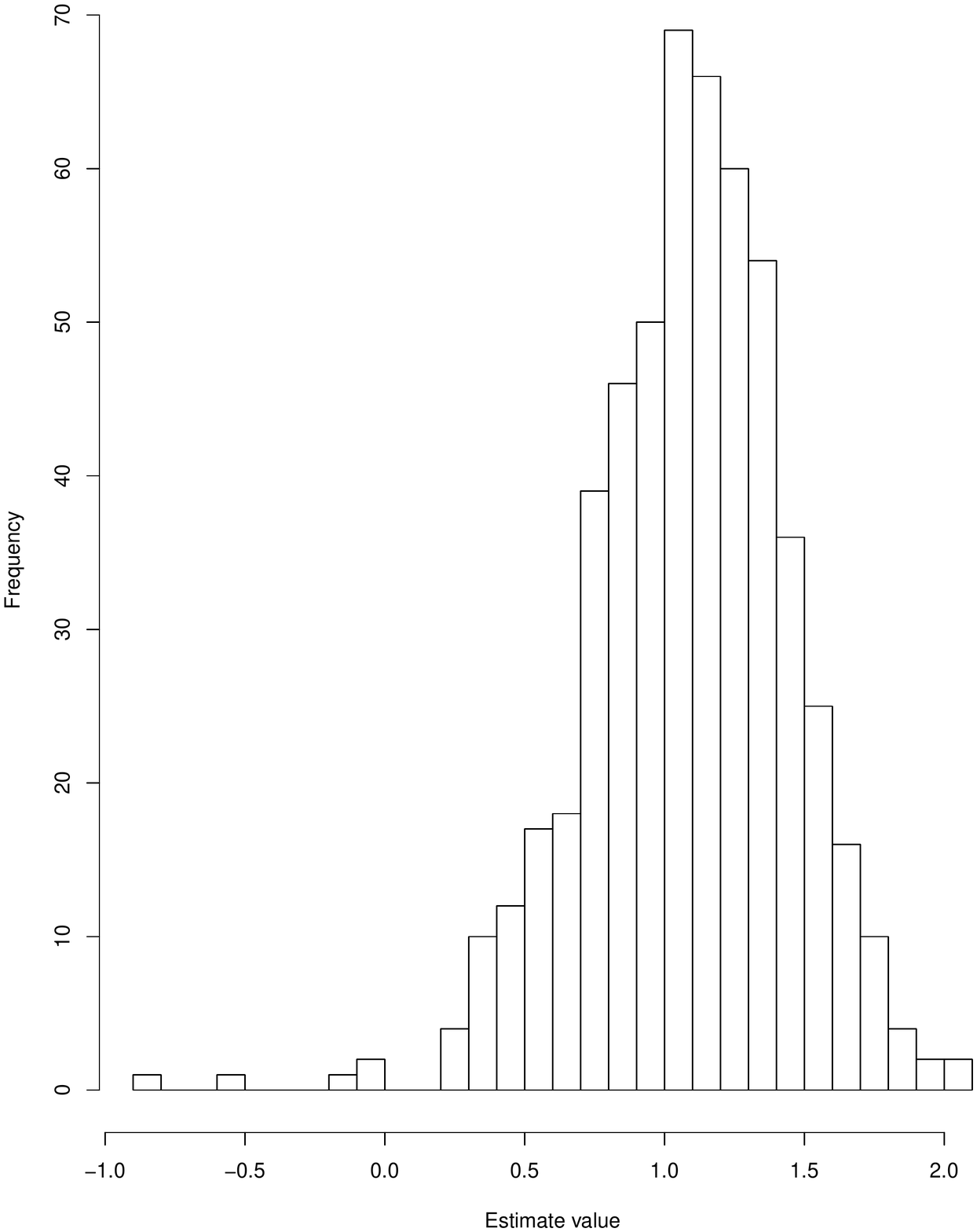}
    \caption{Estimates of $\alpha_0$.}
    \label{fig:H=2per3-N=100-alpha0}
  \end{subfigure}
 \begin{subfigure}[t]{0.32\textwidth}
    \includegraphics[width=\textwidth]{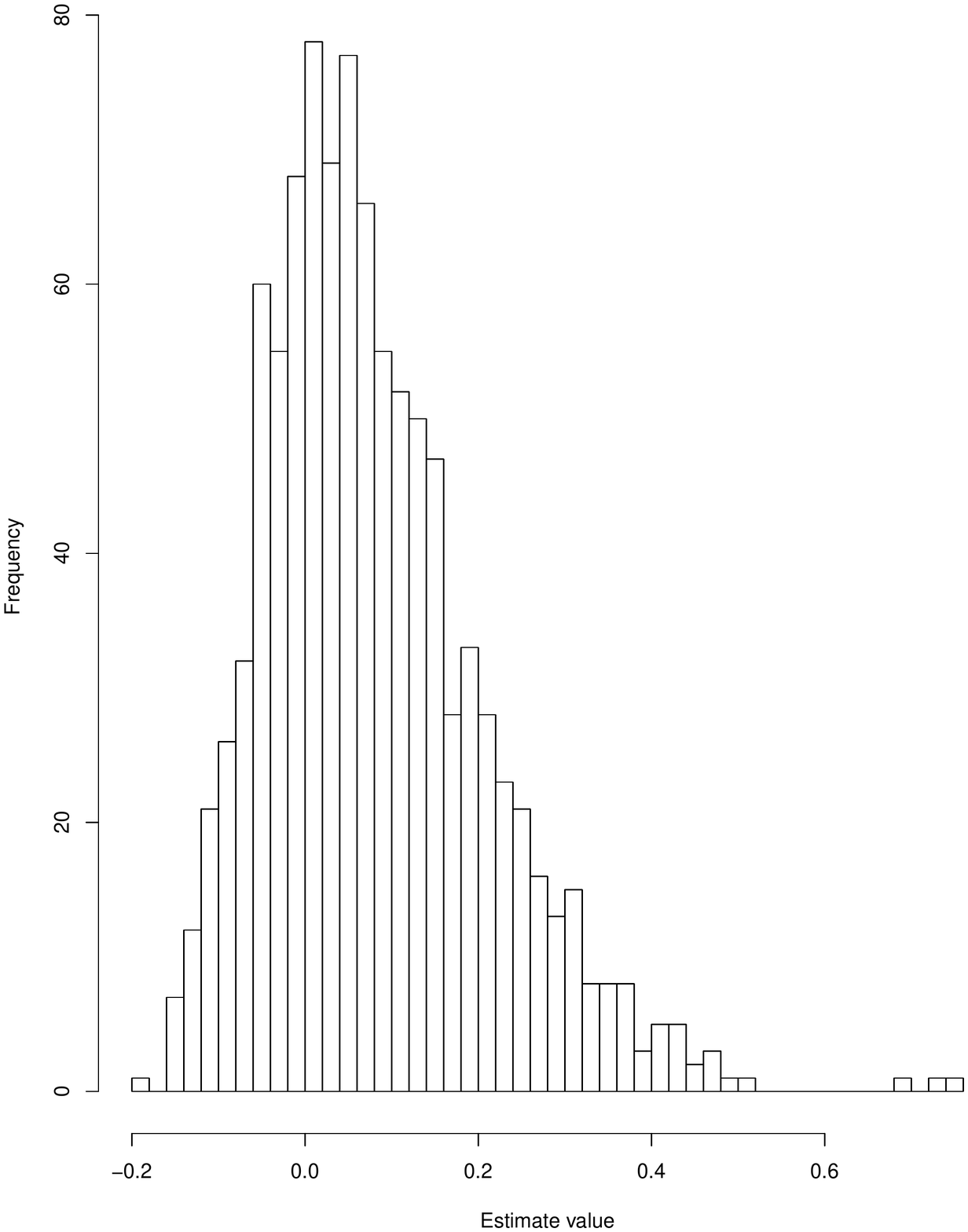}
    \caption{Estimates of $\alpha_1$.}
   \label{fig:H=2per3-N=100-alpha1}
 \end{subfigure}
 \begin{subfigure}[t]{0.32\textwidth}
    \includegraphics[width=\textwidth]{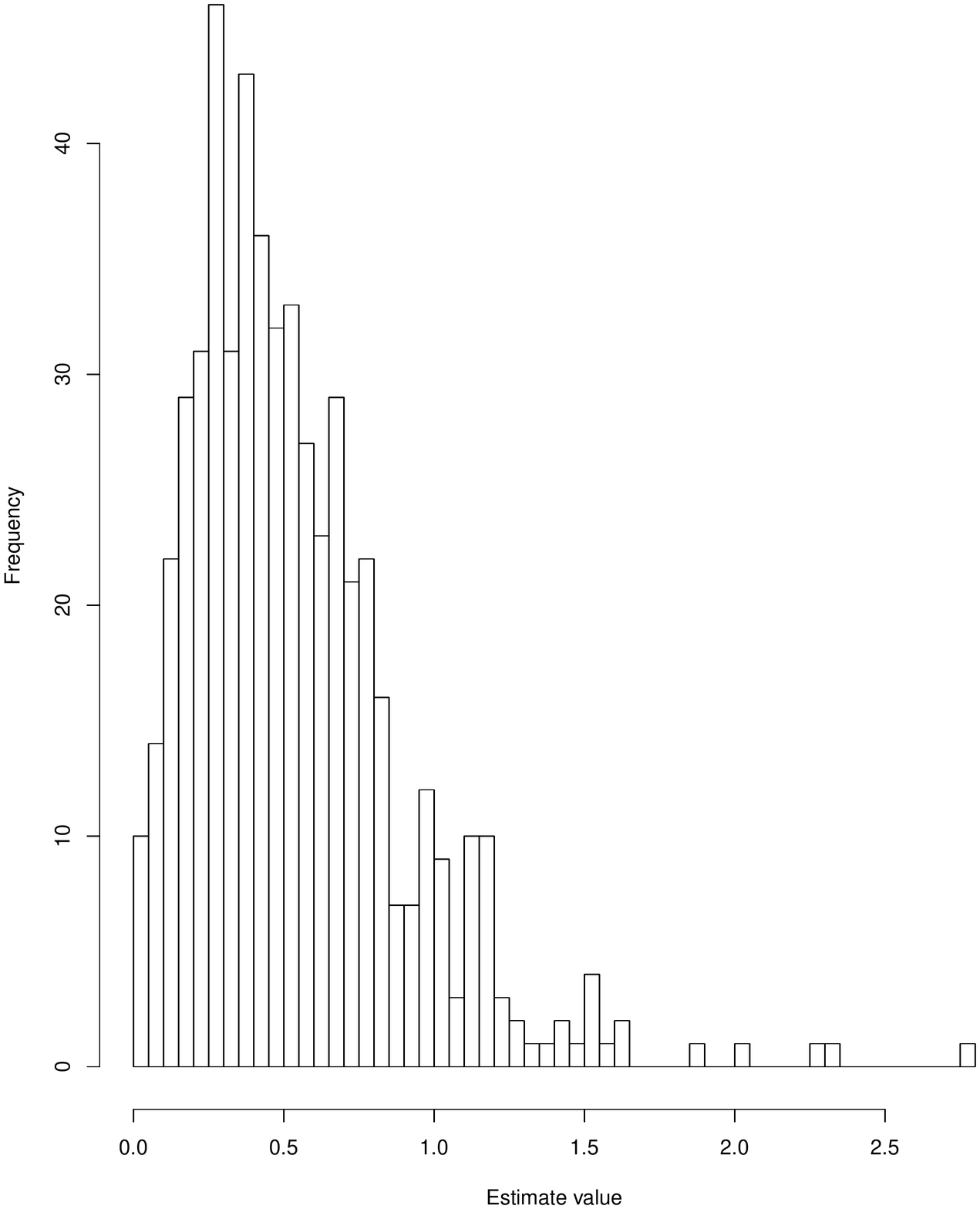}
    \caption{Estimates of $l_1$.}
   \label{fig:H=2per3-N=100-l1}
 \end{subfigure}
\caption{Fractional Brownian motion liquidity with $H=\frac{2}{3}$ and $N=100$.}
\label{fig:H=2per3-N=100}
\end{figure}

\begin{figure}[H]
\centering
  \begin{subfigure}[t]{0.32\textwidth}
   \includegraphics[width=\textwidth]{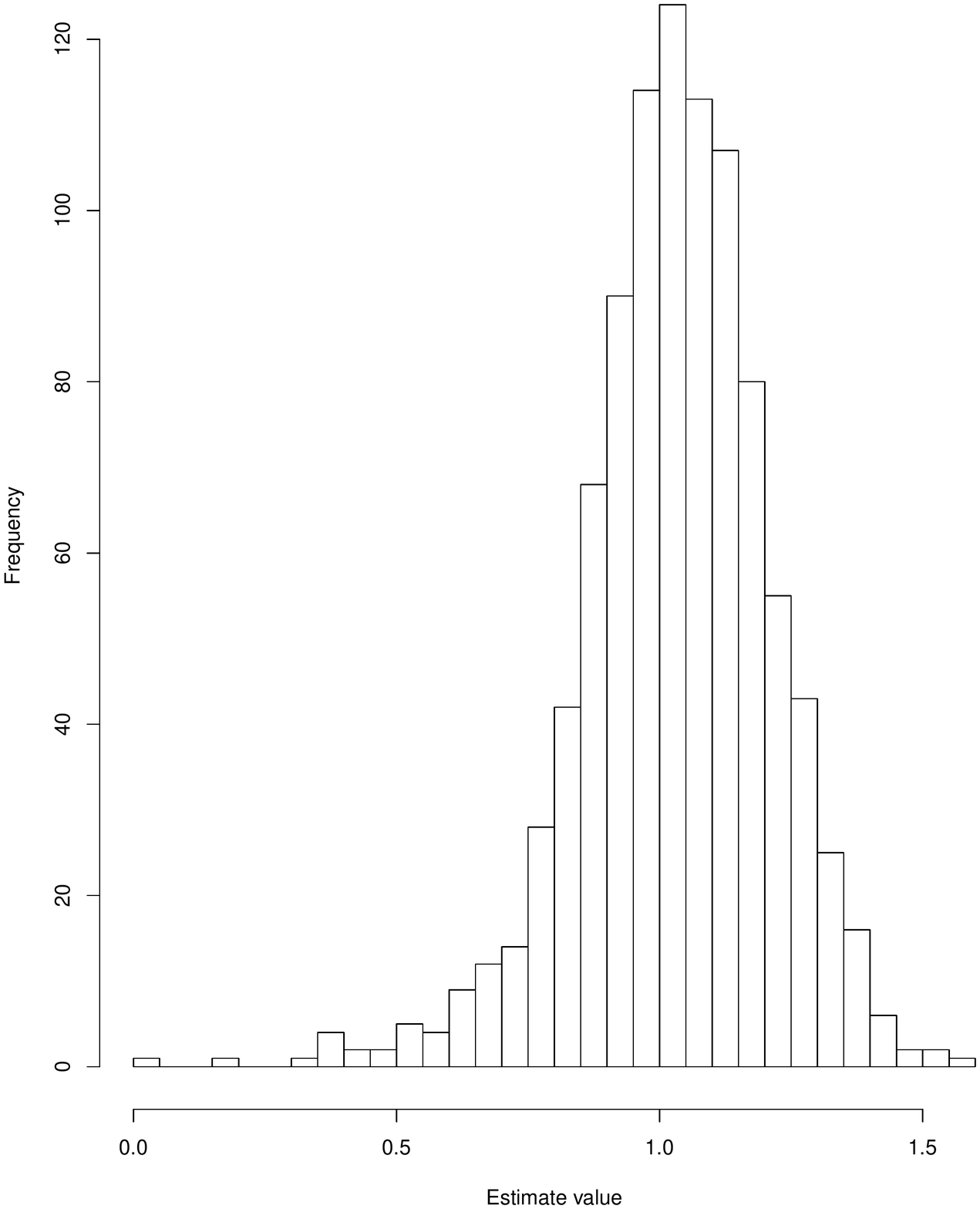}
    \caption{Estimates of $\alpha_0$.}
    \label{fig:H=2per3-N=1000-alpha0}
  \end{subfigure}
 \begin{subfigure}[t]{0.32\textwidth}
    \includegraphics[width=\textwidth]{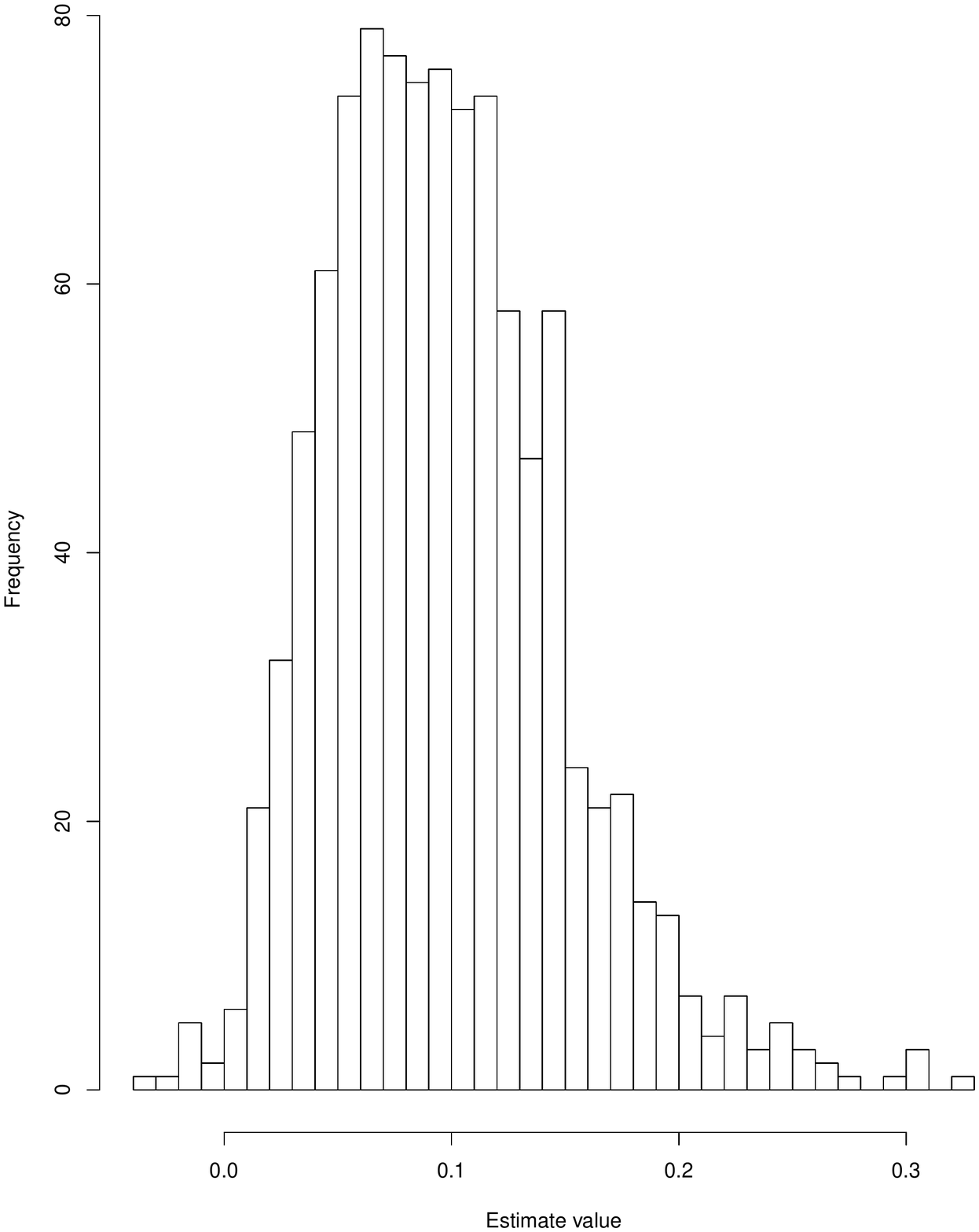}
    \caption{Estimates of $\alpha_1$.}
   \label{fig:H=2per3-N=1000-alpha1}
 \end{subfigure}
 \begin{subfigure}[t]{0.32\textwidth}
    \includegraphics[width=\textwidth]{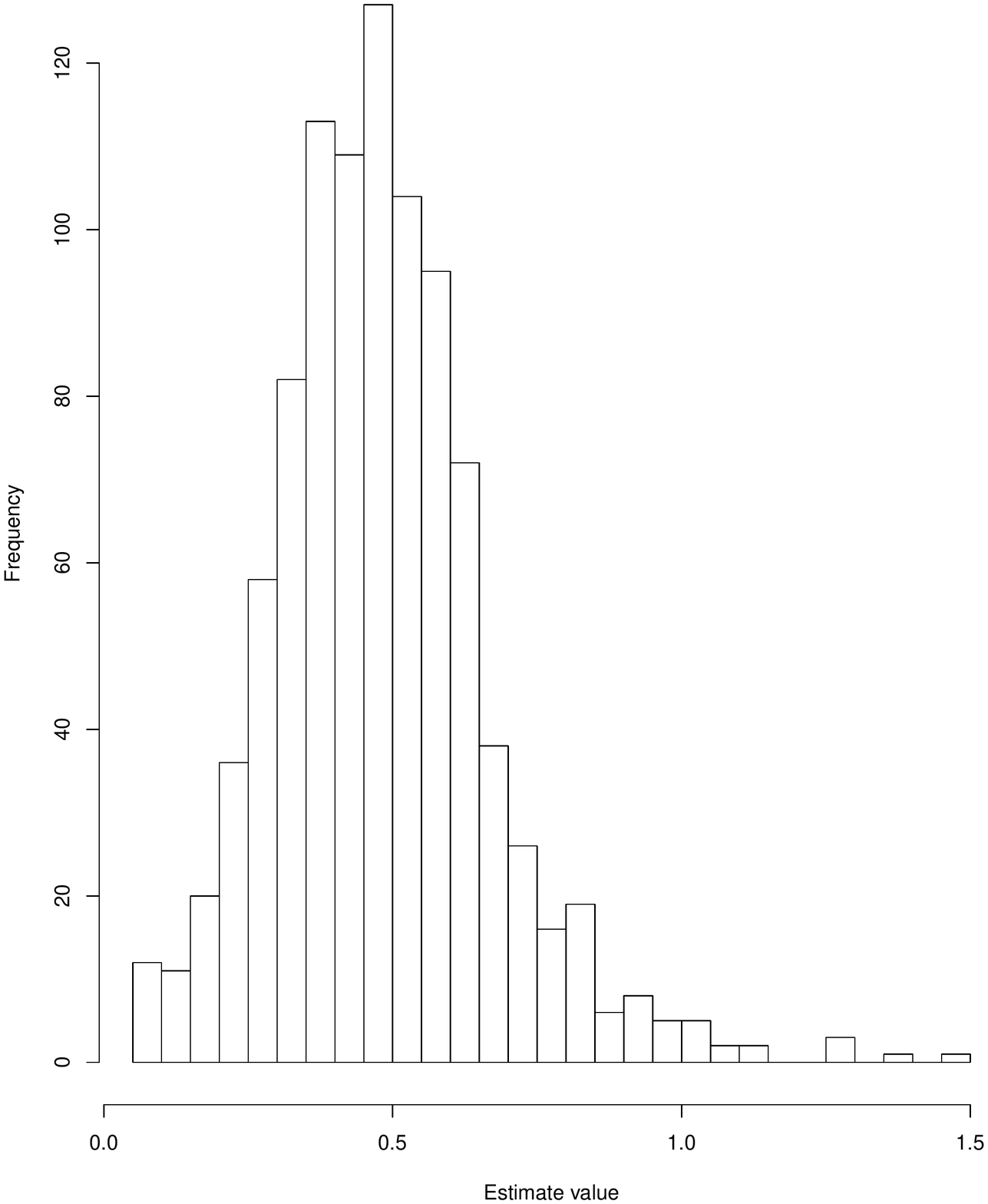}
    \caption{Estimates of $l_1$.}
   \label{fig:H=2per3-N=1000-l1}
 \end{subfigure}
\caption{Fractional Brownian motion liquidity with $H=\frac{2}{3}$ and $N=1000$.}
\label{fig:H=2per3-N=1000}
\end{figure}

\begin{figure}[H]
\centering
  \begin{subfigure}[t]{0.32\textwidth}
   \includegraphics[width=\textwidth]{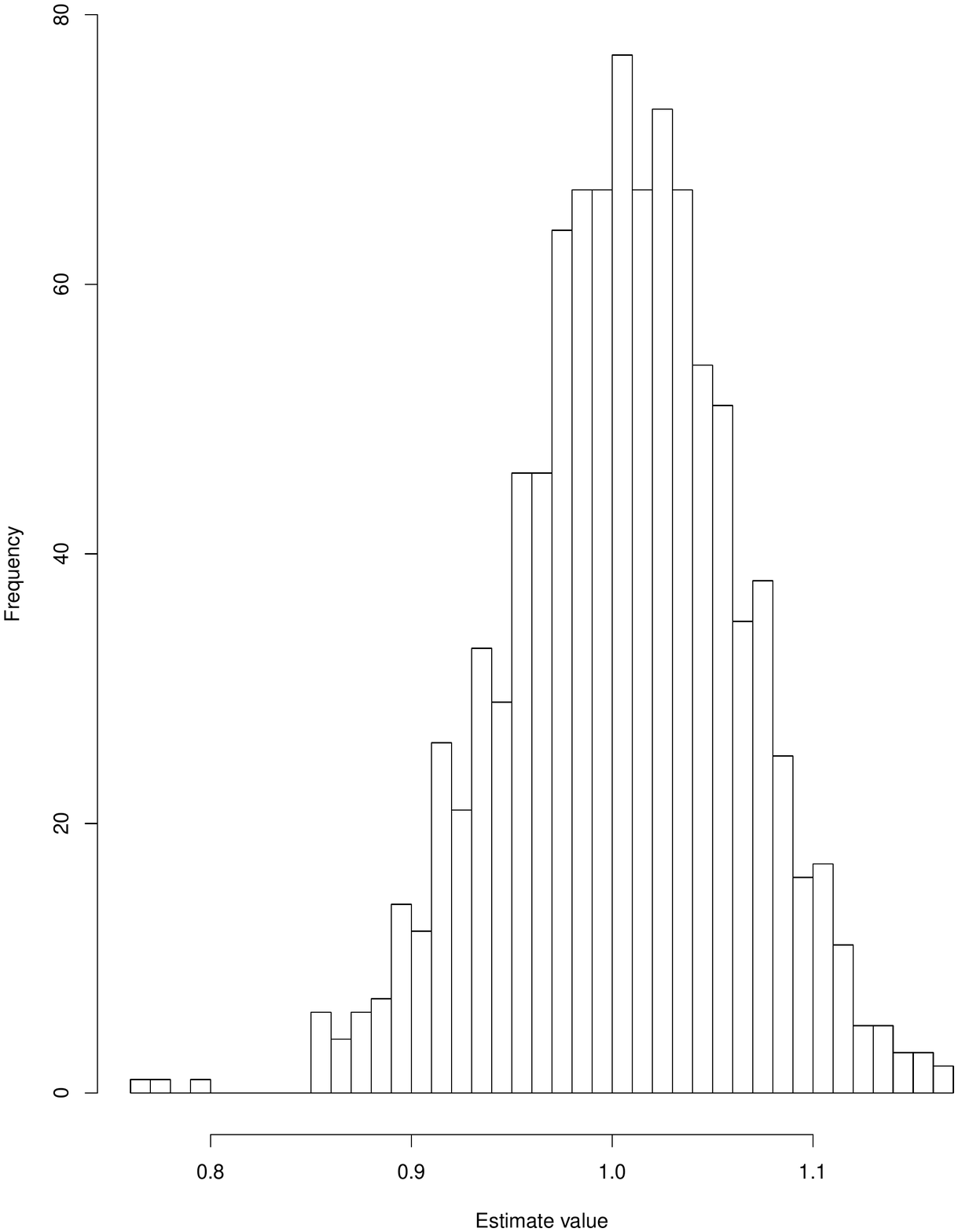}
    \caption{Estimates of $\alpha_0$.}
    \label{fig:H=2per3-N=10000-alpha0}
  \end{subfigure}
 \begin{subfigure}[t]{0.32\textwidth}
    \includegraphics[width=\textwidth]{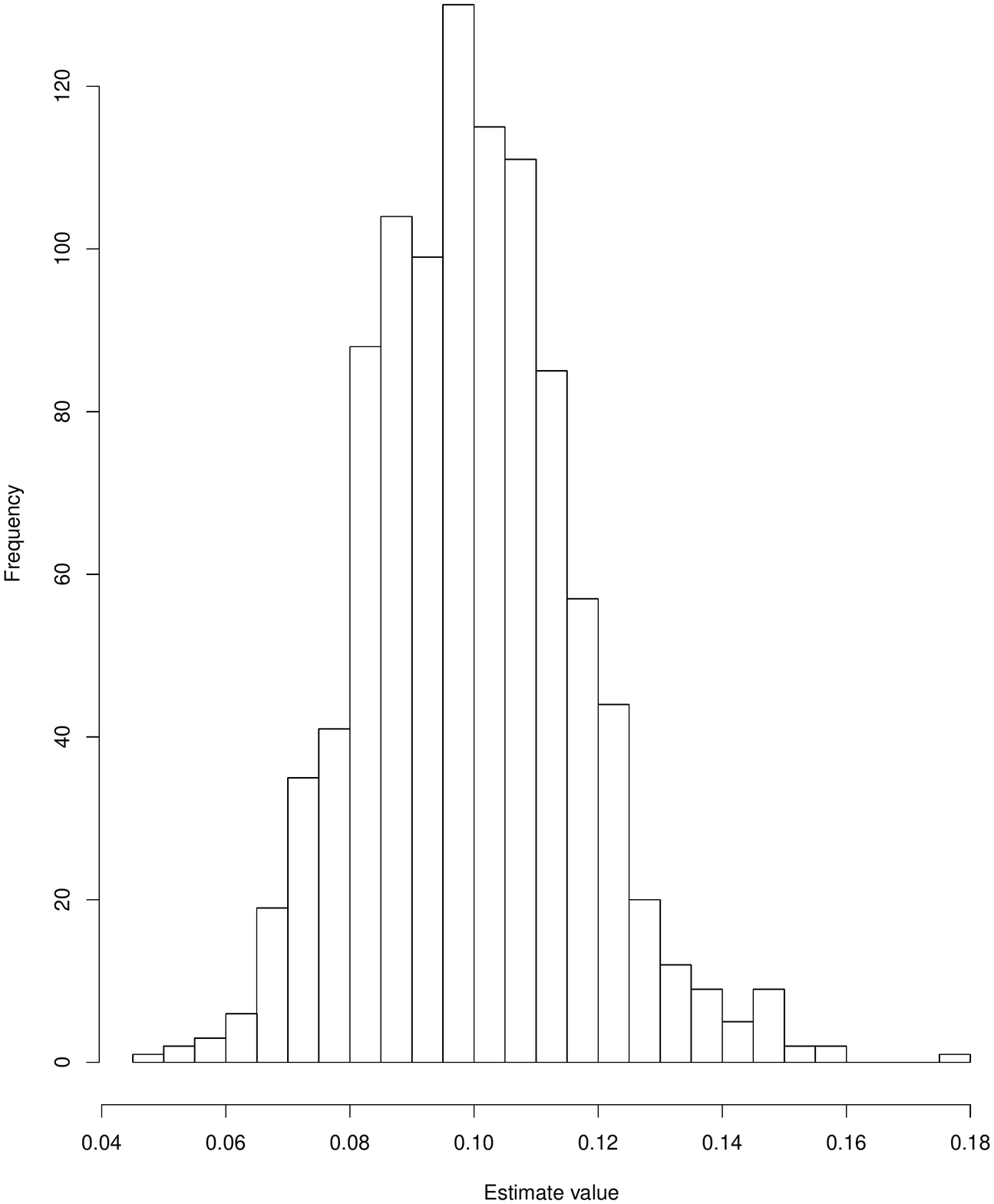}
    \caption{Estimates of $\alpha_1$.}
   \label{fig:H=2per3-N=10000-alpha1}
 \end{subfigure}
 \begin{subfigure}[t]{0.32\textwidth}
    \includegraphics[width=\textwidth]{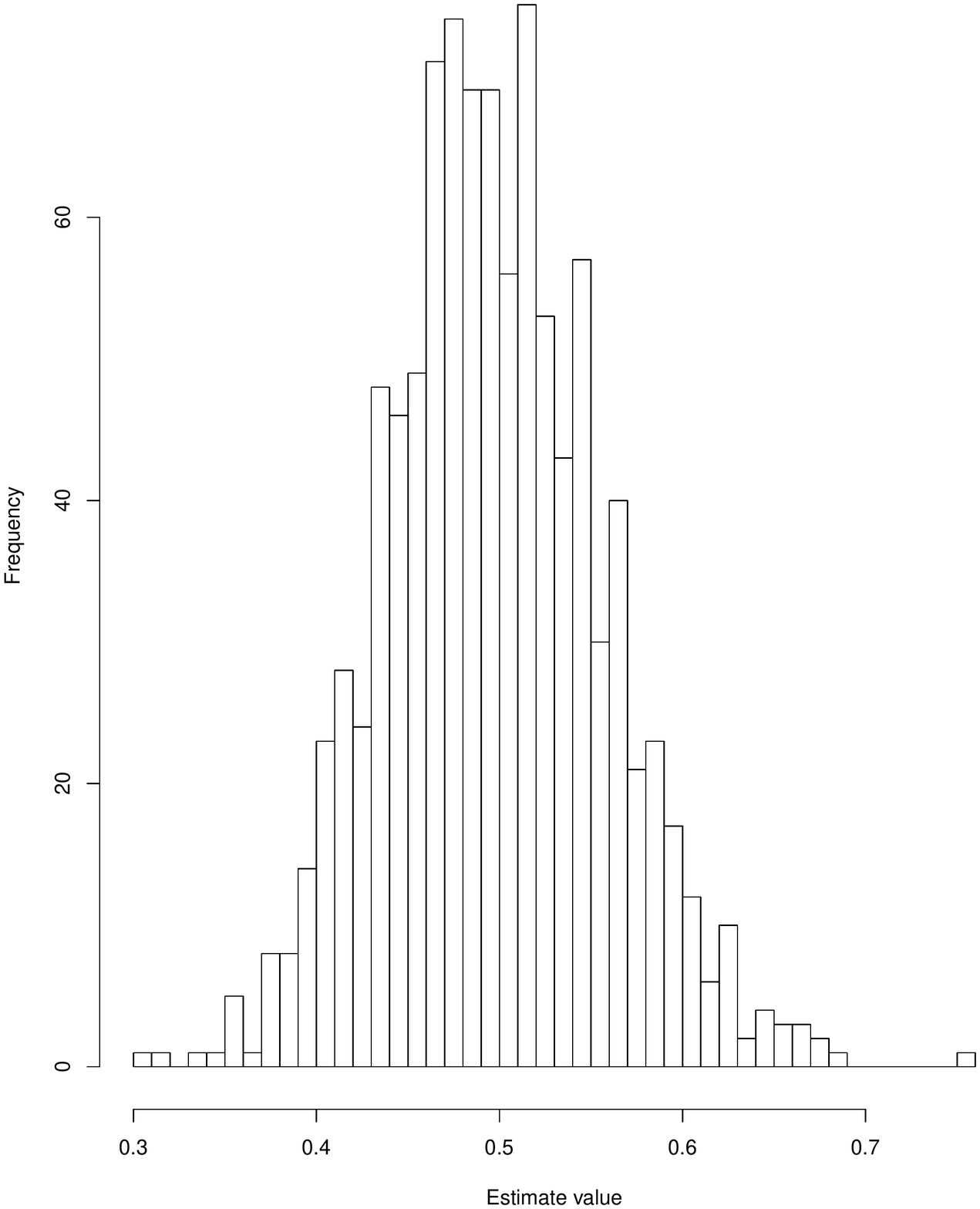}
    \caption{Estimates of $l_1$.}
   \label{fig:H=2per3-N=10000-l1}
 \end{subfigure}
\caption{Fractional Brownian motion liquidity with $H=\frac{2}{3}$ and $N=10000$.}
\label{fig:H=2per3-N=10000}
\end{figure}

\begin{figure}[H]
\centering
  \begin{subfigure}[t]{0.32\textwidth}
   \includegraphics[width=\textwidth]{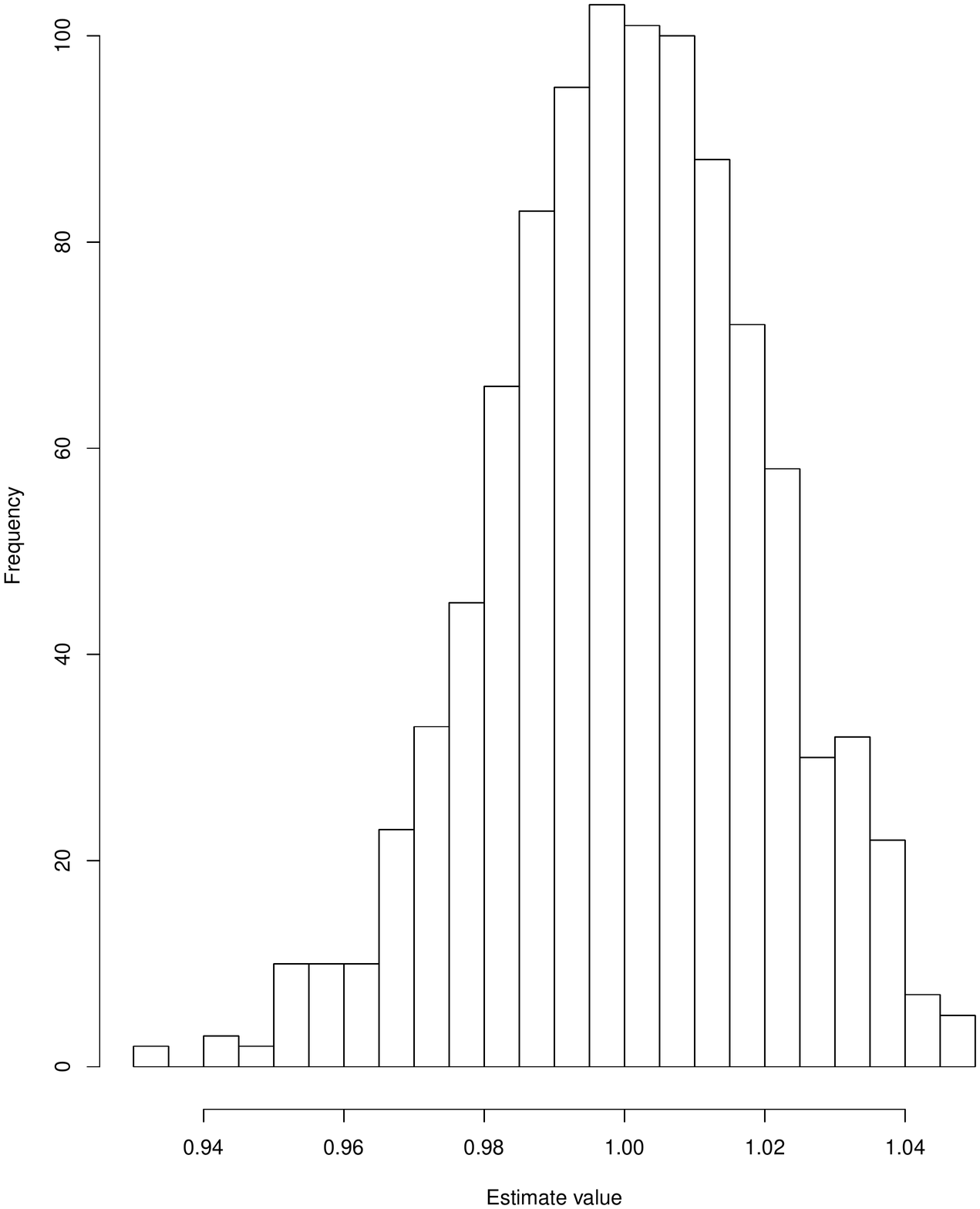}
    \caption{Estimates of $\alpha_0$.}
    \label{fig:H=2per3-N=100000-alpha0}
  \end{subfigure}
 \begin{subfigure}[t]{0.32\textwidth}
    \includegraphics[width=\textwidth]{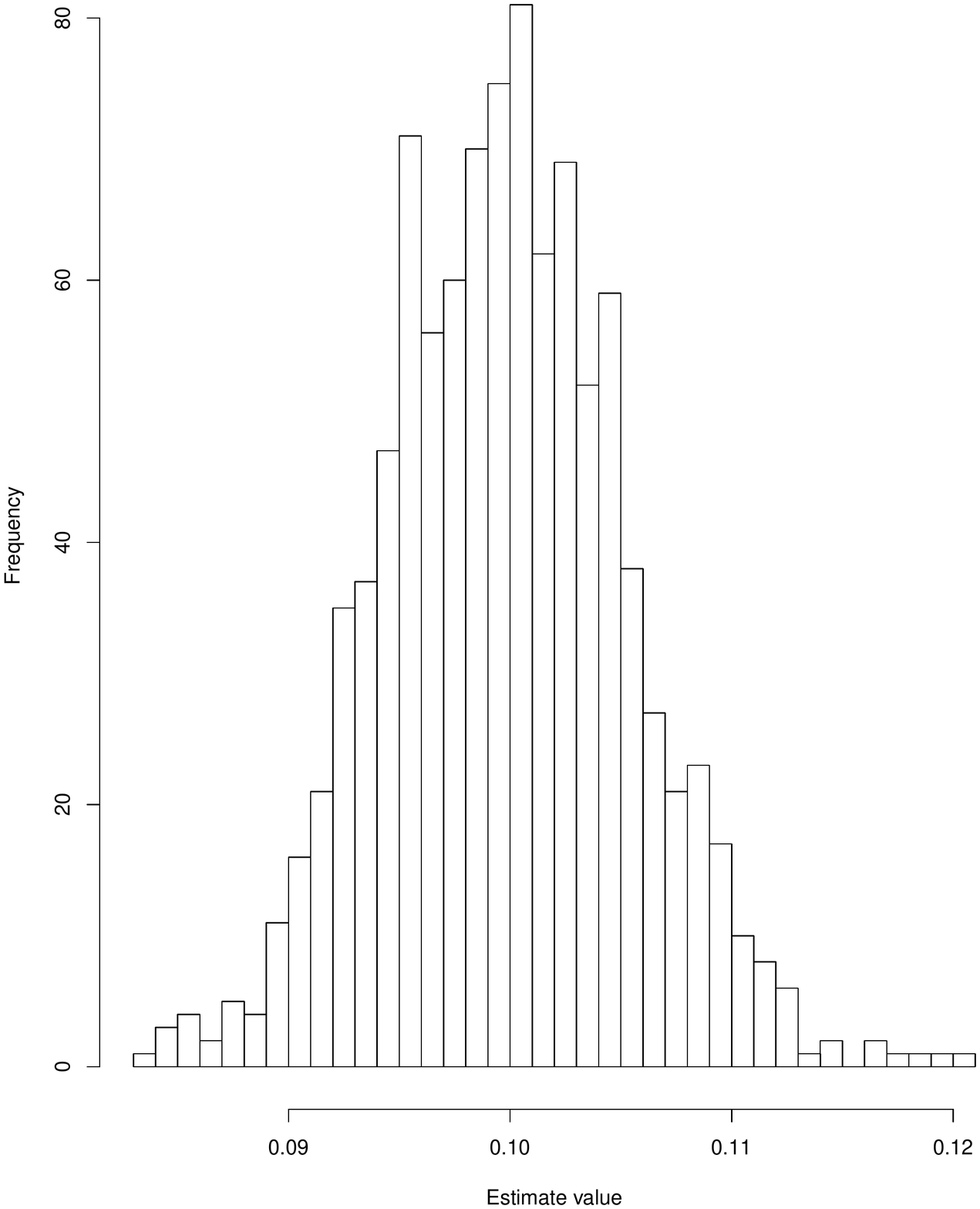}
    \caption{Estimates of $\alpha_1$.}
   \label{fig:H=2per3-N=100000-alpha1}
 \end{subfigure}
 \begin{subfigure}[t]{0.32\textwidth}
    \includegraphics[width=\textwidth]{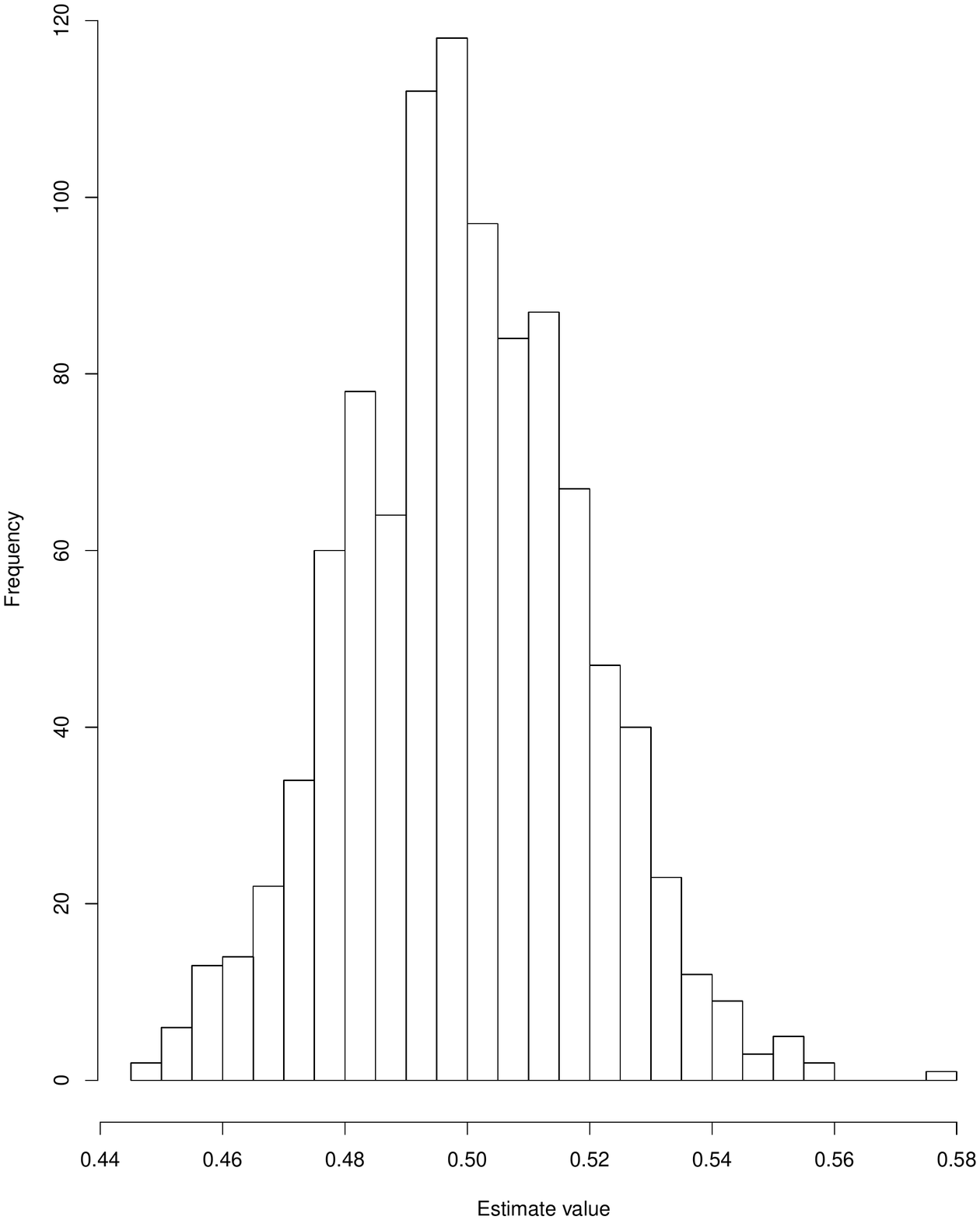}
    \caption{Estimates of $l_1$.}
   \label{fig:H=2per3-N=100000-l1}
 \end{subfigure}
\caption{Fractional Brownian motion liquidity with $H=\frac{2}{3}$ and $N=100000$.}
\label{fig:H=2per3-N=100000}
\end{figure}

\subsection{Fractional Brownian motion with $H=\frac{4}{5}$.}
Histograms of the estimates of the model parameters corresponding to $L_t = (B^H_{t+1} - B^H_{t})^2$ with $H = \frac{4}{5}$ are provided in Figures \ref{fig:H=4per5-N=100}, \ref{fig:H=4per5-N=1000}, \ref{fig:H=4per5-N=10000} and \ref{fig:H=4per5-N=100000}. The used sample sizes were $N=100, N=1000, N=10000$ and $N= 100000$. The sample sizes $N=100$ and $N=1000$ resulted complex valued estimates in $47.9\%$ and $4.3\%$ of the simulations respectively, whereas with the larger sample sizes all the estimates were real.

\begin{figure}[H]
\centering
  \begin{subfigure}[t]{0.32\textwidth}
   \includegraphics[width=\textwidth]{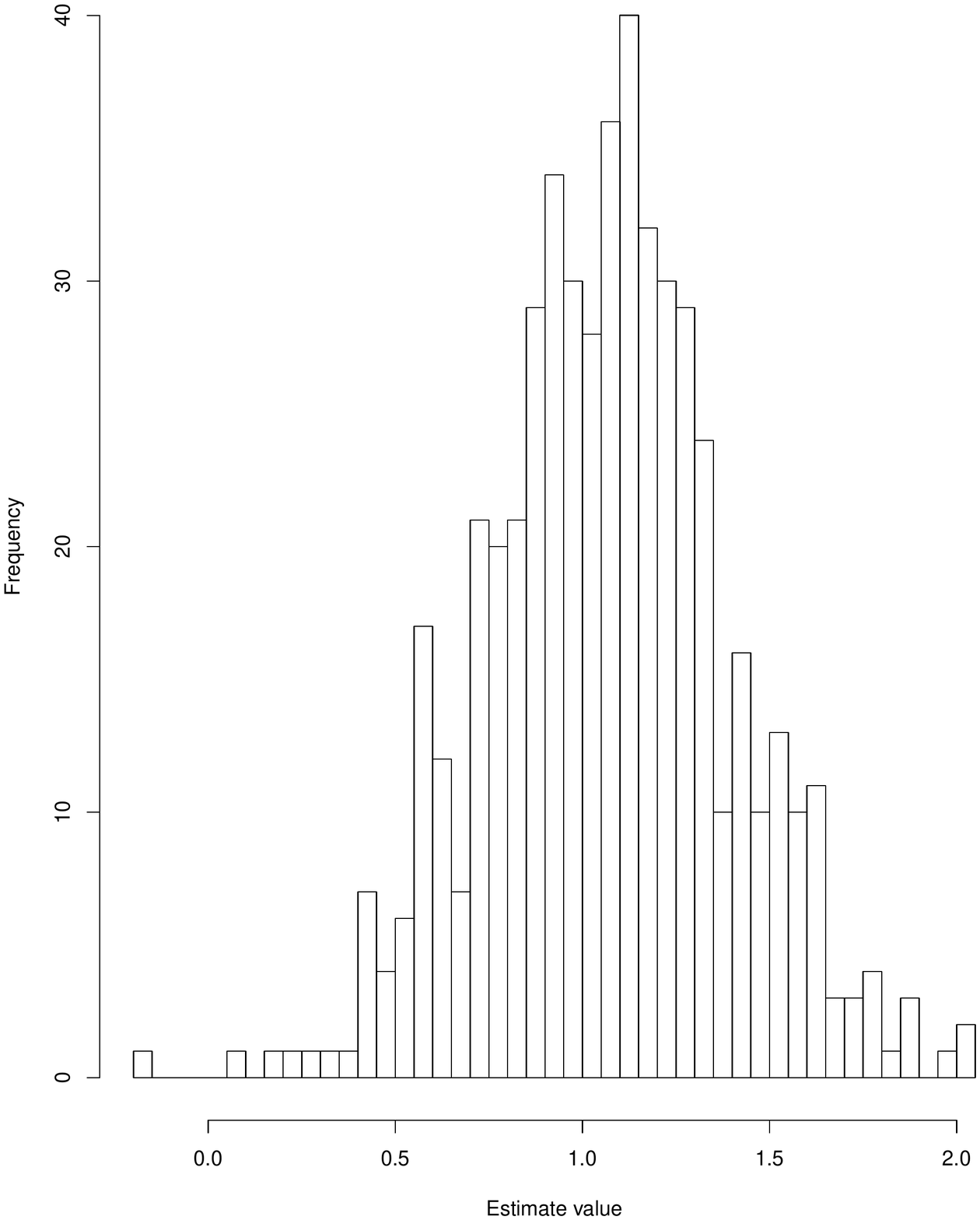}
    \caption{Estimates of $\alpha_0$.}
    \label{fig:H=4per5-N=100-alpha0}
  \end{subfigure}
 \begin{subfigure}[t]{0.32\textwidth}
    \includegraphics[width=\textwidth]{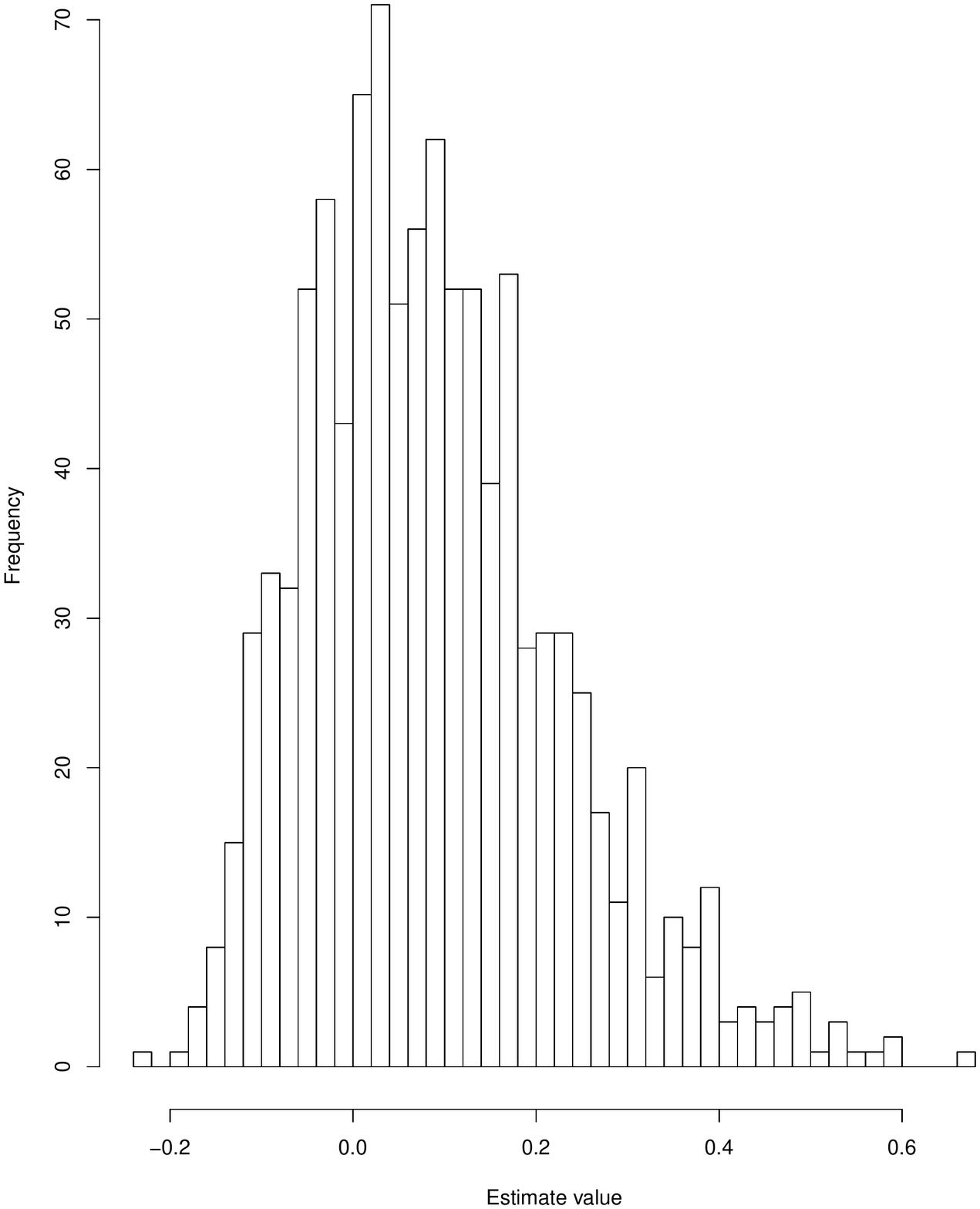}
    \caption{Estimates of $\alpha_1$.}
   \label{fig:H=4per5-N=100-alpha1}
 \end{subfigure}
 \begin{subfigure}[t]{0.32\textwidth}
    \includegraphics[width=\textwidth]{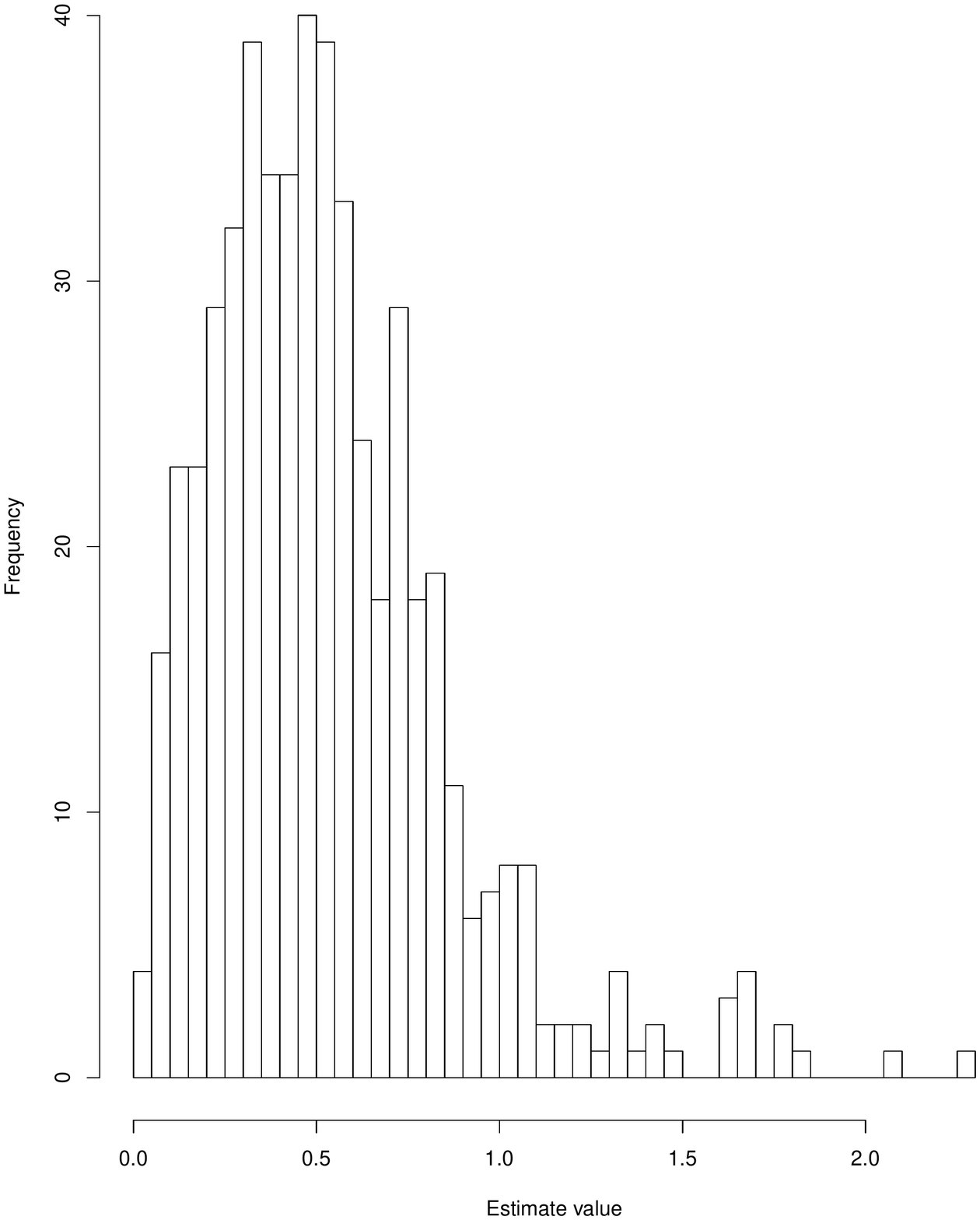}
    \caption{Estimates of $l_1$.}
   \label{fig:H=4per5-N=100-l1}
 \end{subfigure}
\caption{Fractional Brownian motion liquidity with $H=\frac{4}{5}$ and $N=100$.}
\label{fig:H=4per5-N=100}
\end{figure}

\begin{figure}[H]
\centering
  \begin{subfigure}[t]{0.32\textwidth}
   \includegraphics[width=\textwidth]{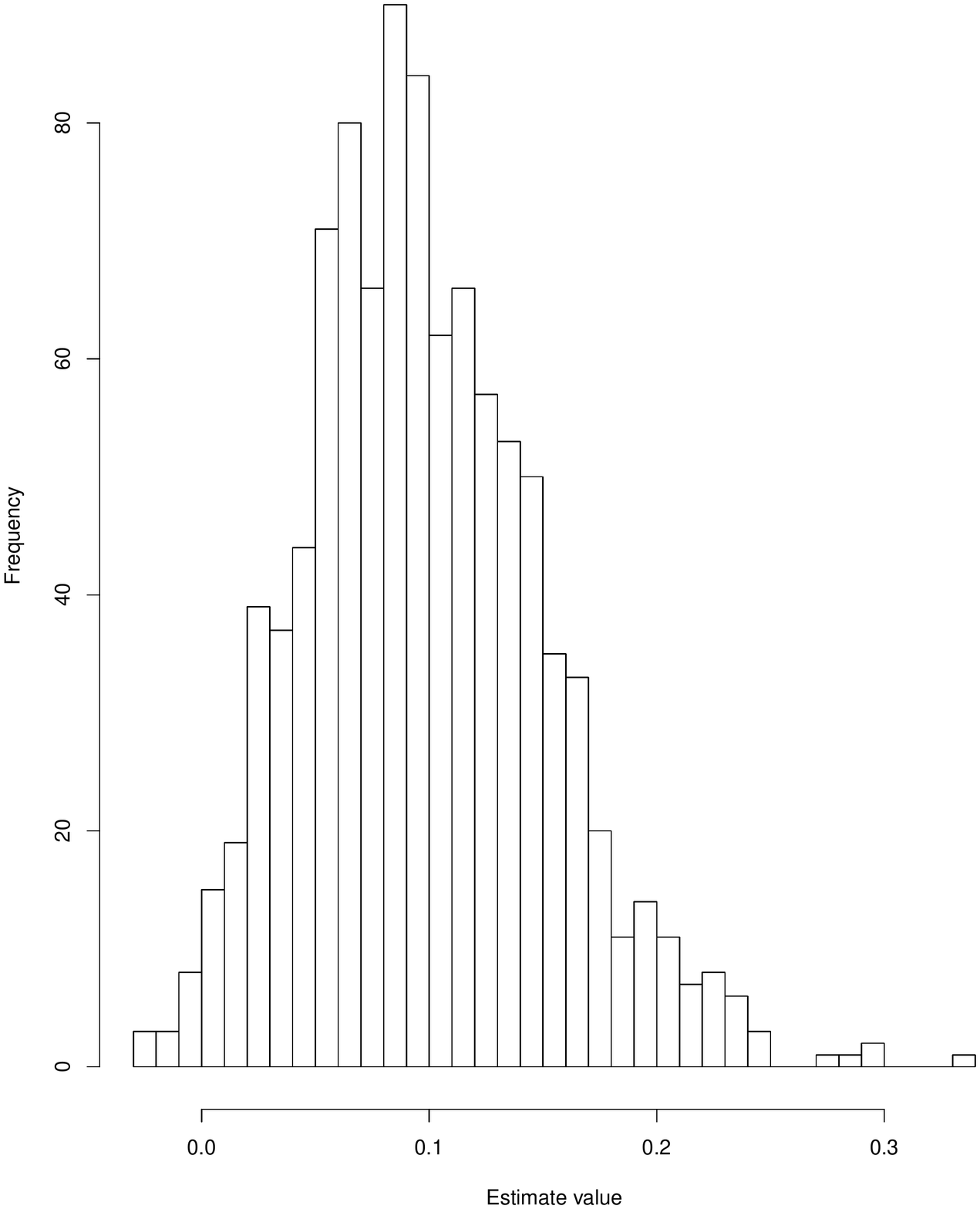}
    \caption{Estimates of $\alpha_0$.}
    \label{fig:H=4per5-N=1000-alpha0}
  \end{subfigure}
 \begin{subfigure}[t]{0.32\textwidth}
    \includegraphics[width=\textwidth]{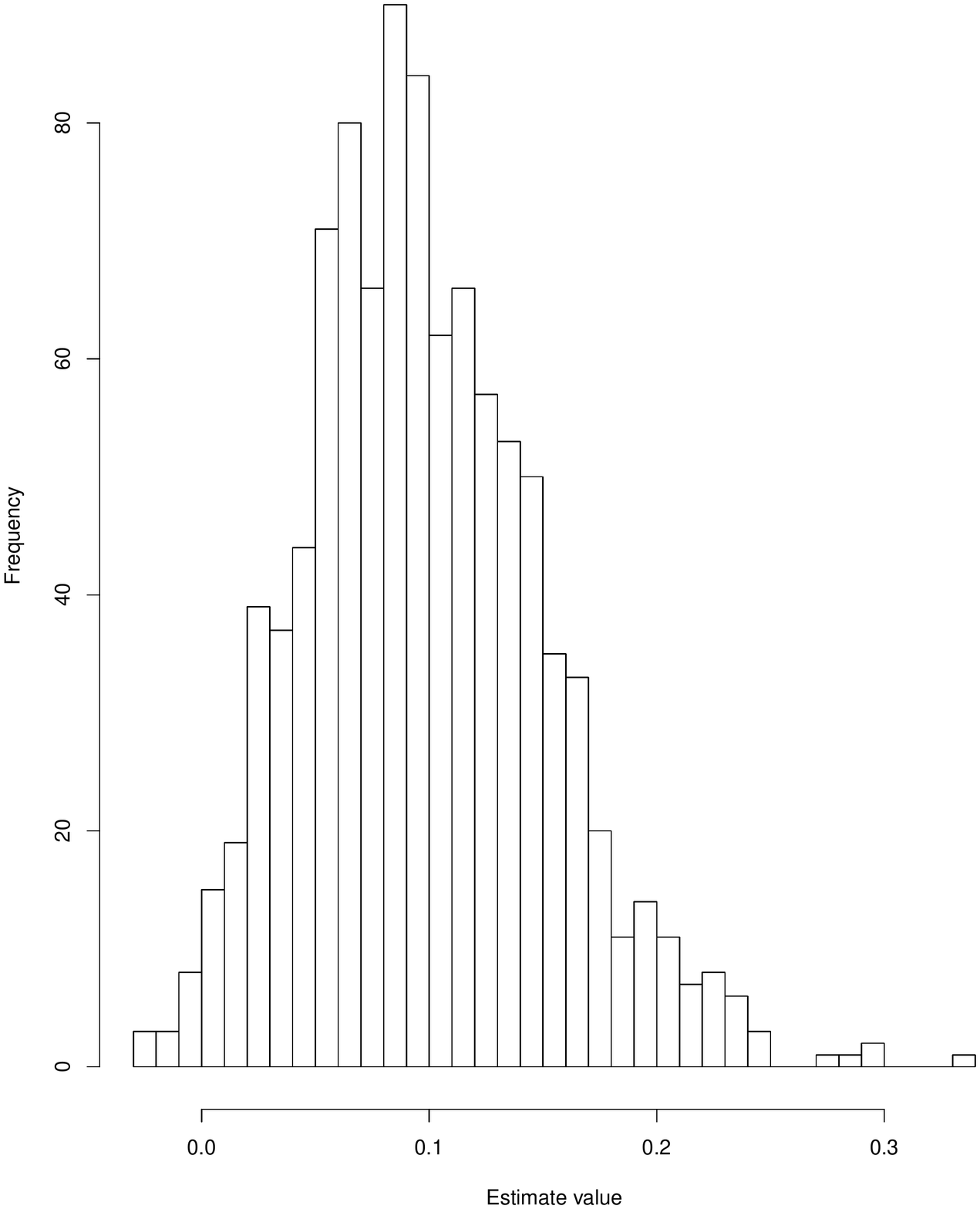}
    \caption{Estimates of $\alpha_1$.}
   \label{fig:H=4per5-N=1000-alpha1}
 \end{subfigure}
 \begin{subfigure}[t]{0.32\textwidth}
    \includegraphics[width=\textwidth]{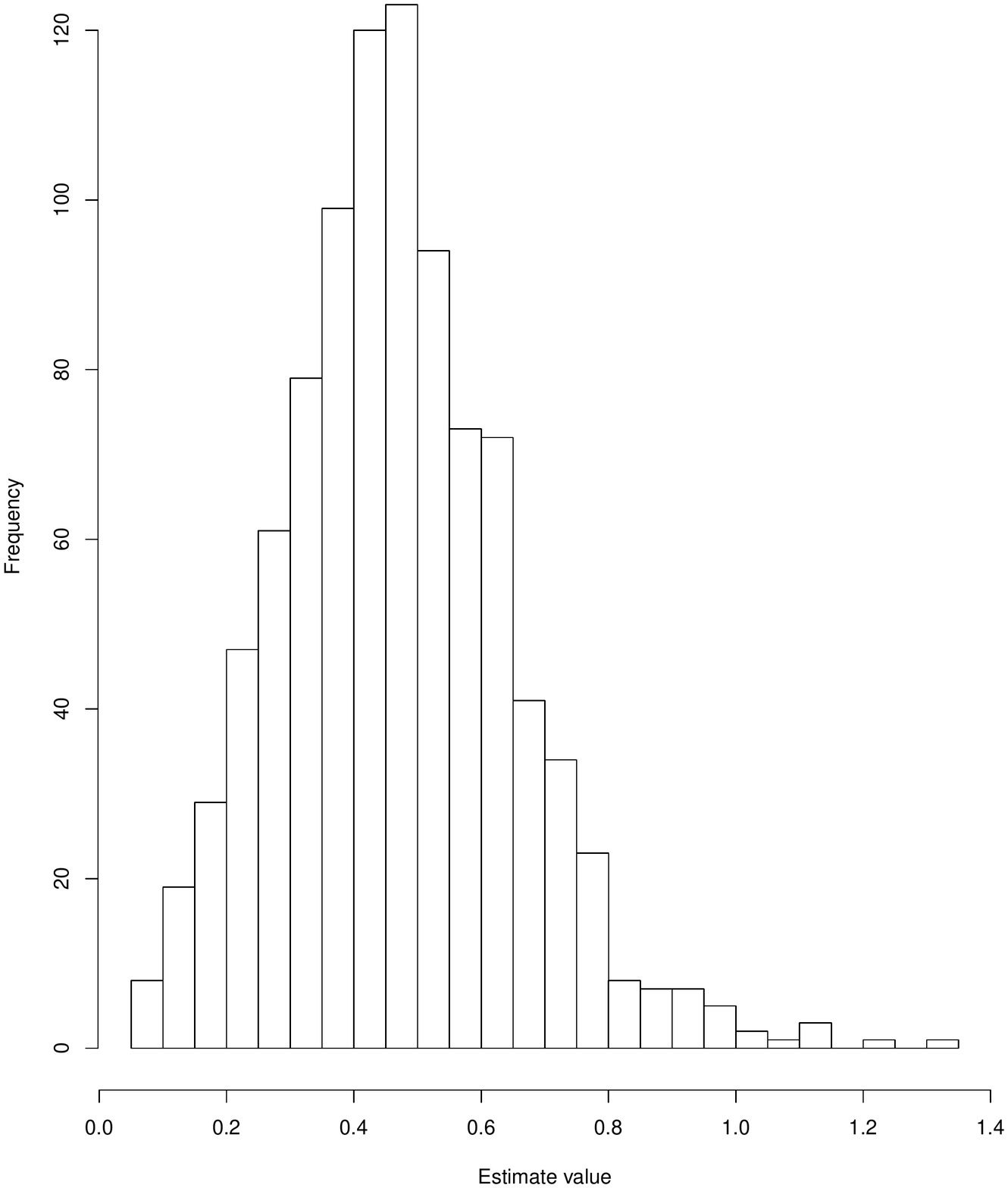}
    \caption{Estimates of $l_1$.}
   \label{fig:H=4per5-N=1000-l1}
 \end{subfigure}
\caption{Fractional Brownian motion liquidity with $H=\frac{4}{5}$ and $N=1000$.}
\label{fig:H=4per5-N=1000}
\end{figure}

\begin{figure}[H]
\centering
  \begin{subfigure}[t]{0.32\textwidth}
   \includegraphics[width=\textwidth]{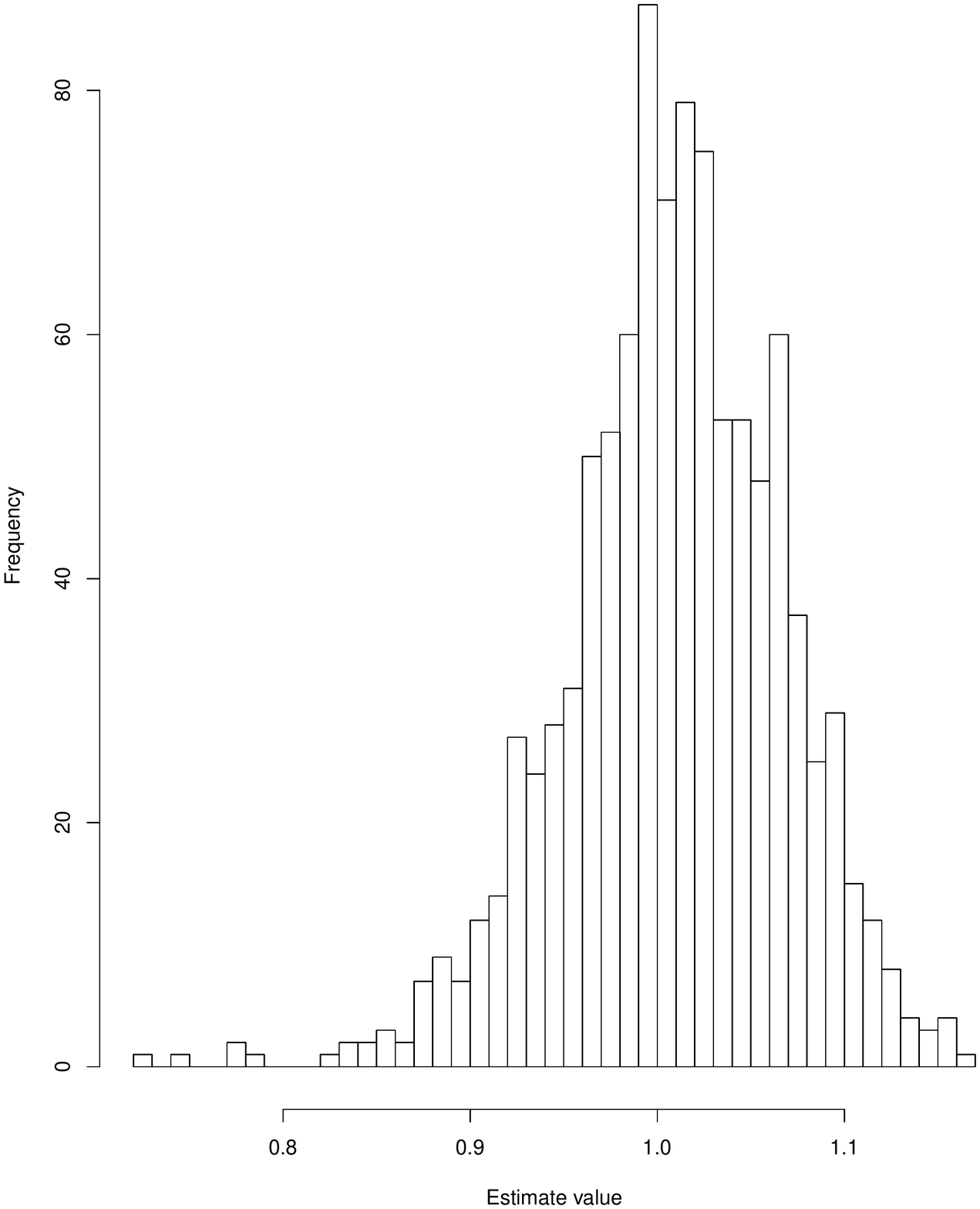}
    \caption{Estimates of $\alpha_0$.}
    \label{fig:H=4per5-N=10000-alpha0}
  \end{subfigure}
 \begin{subfigure}[t]{0.32\textwidth}
    \includegraphics[width=\textwidth]{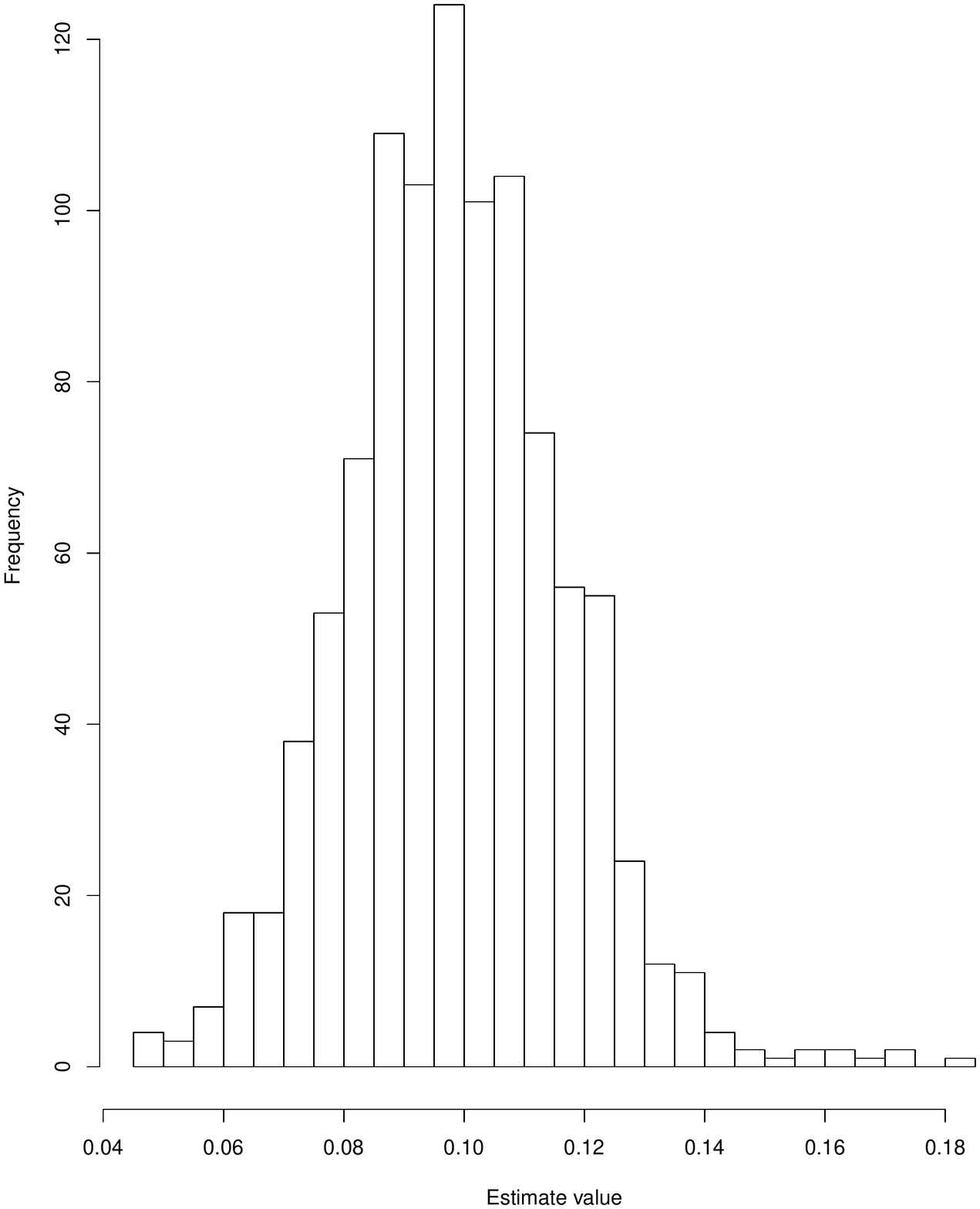}
    \caption{Estimates of $\alpha_1$.}
   \label{fig:H=4per5-N=10000-alpha1}
 \end{subfigure}
 \begin{subfigure}[t]{0.32\textwidth}
    \includegraphics[width=\textwidth]{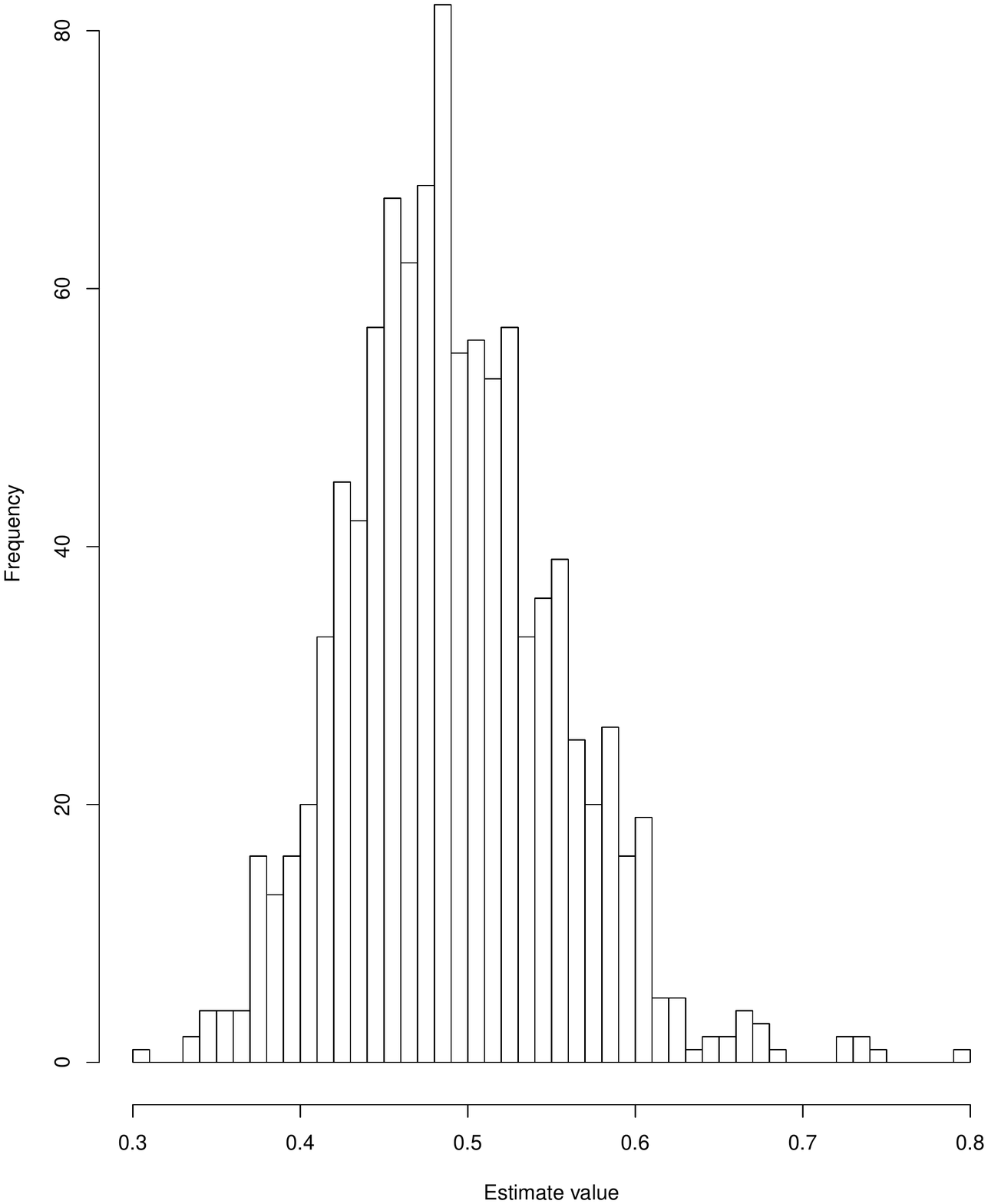}
    \caption{Estimates of $l_1$.}
   \label{fig:H=4per5-N=10000-l1}
 \end{subfigure}
\caption{Fractional Brownian motion liquidity with $H=\frac{4}{5}$ and $N=10000$.}
\label{fig:H=4per5-N=10000}
\end{figure}

\begin{figure}[H]
\centering
  \begin{subfigure}[t]{0.32\textwidth}
   \includegraphics[width=\textwidth]{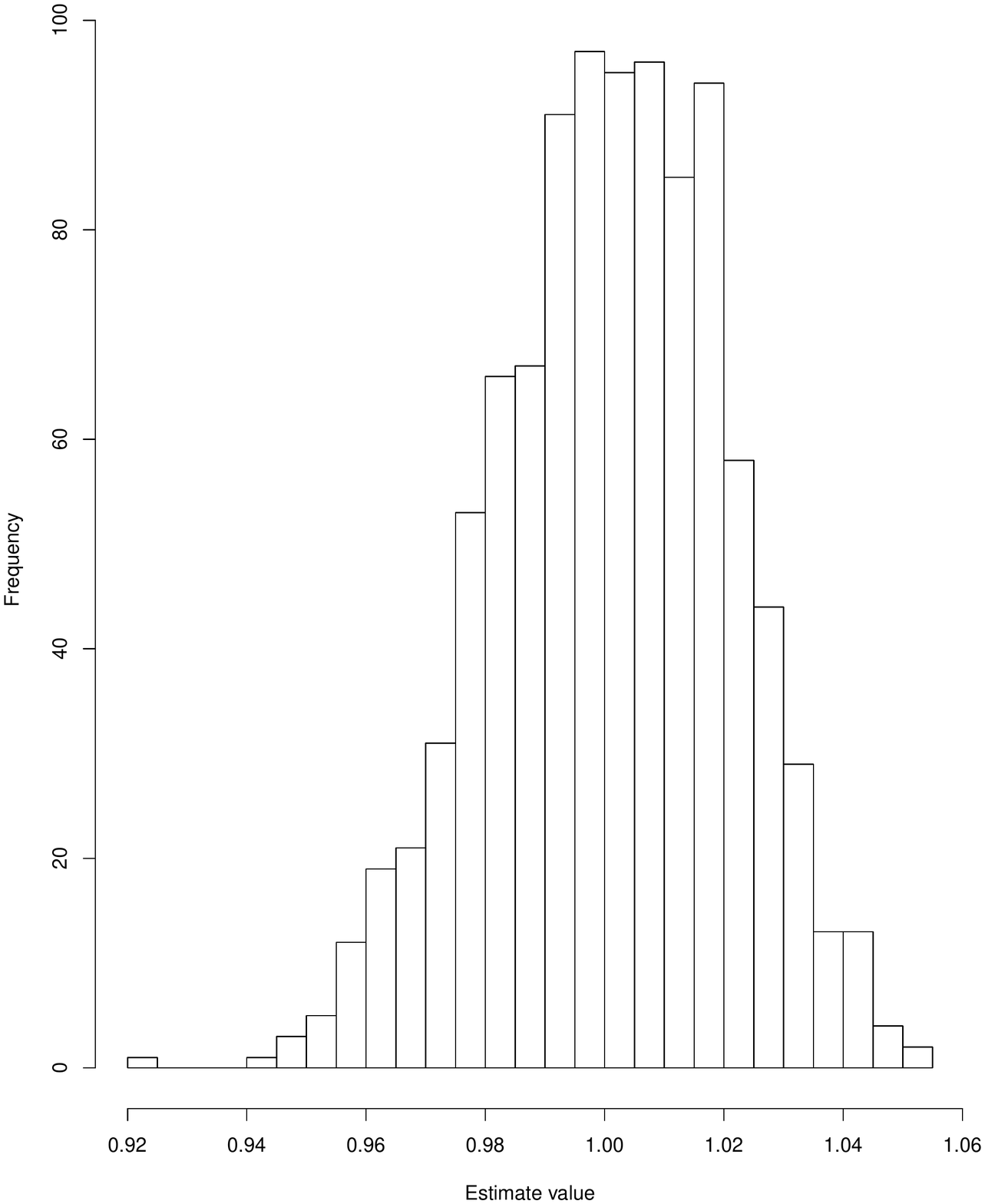}
    \caption{Estimates of $\alpha_0$.}
    \label{fig:H=4per5-N=100000-alpha0}
  \end{subfigure}
 \begin{subfigure}[t]{0.32\textwidth}
    \includegraphics[width=\textwidth]{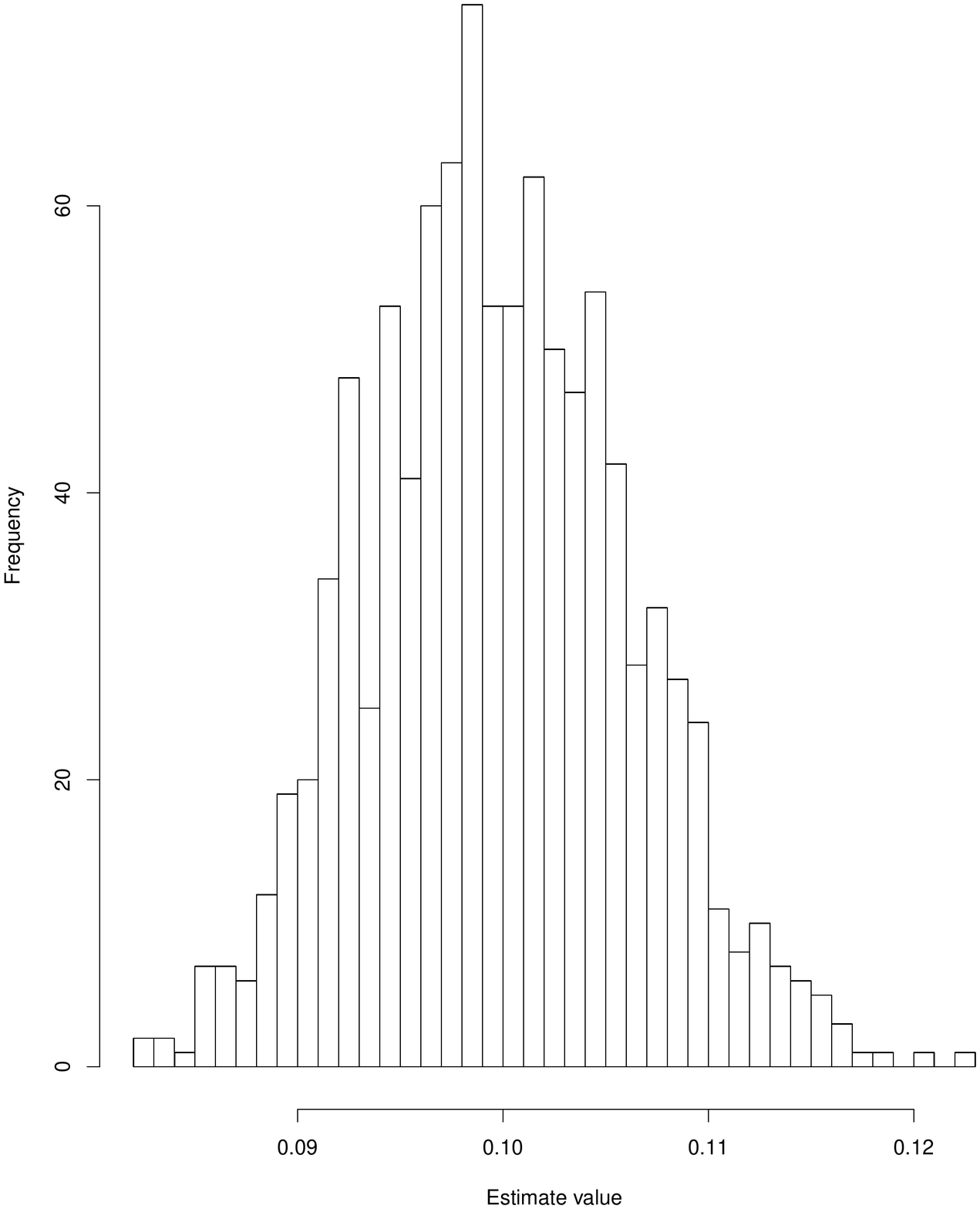}
    \caption{Estimates of $\alpha_1$.}
   \label{fig:H=4per5-N=100000-alpha1}
 \end{subfigure}
 \begin{subfigure}[t]{0.32\textwidth}
    \includegraphics[width=\textwidth]{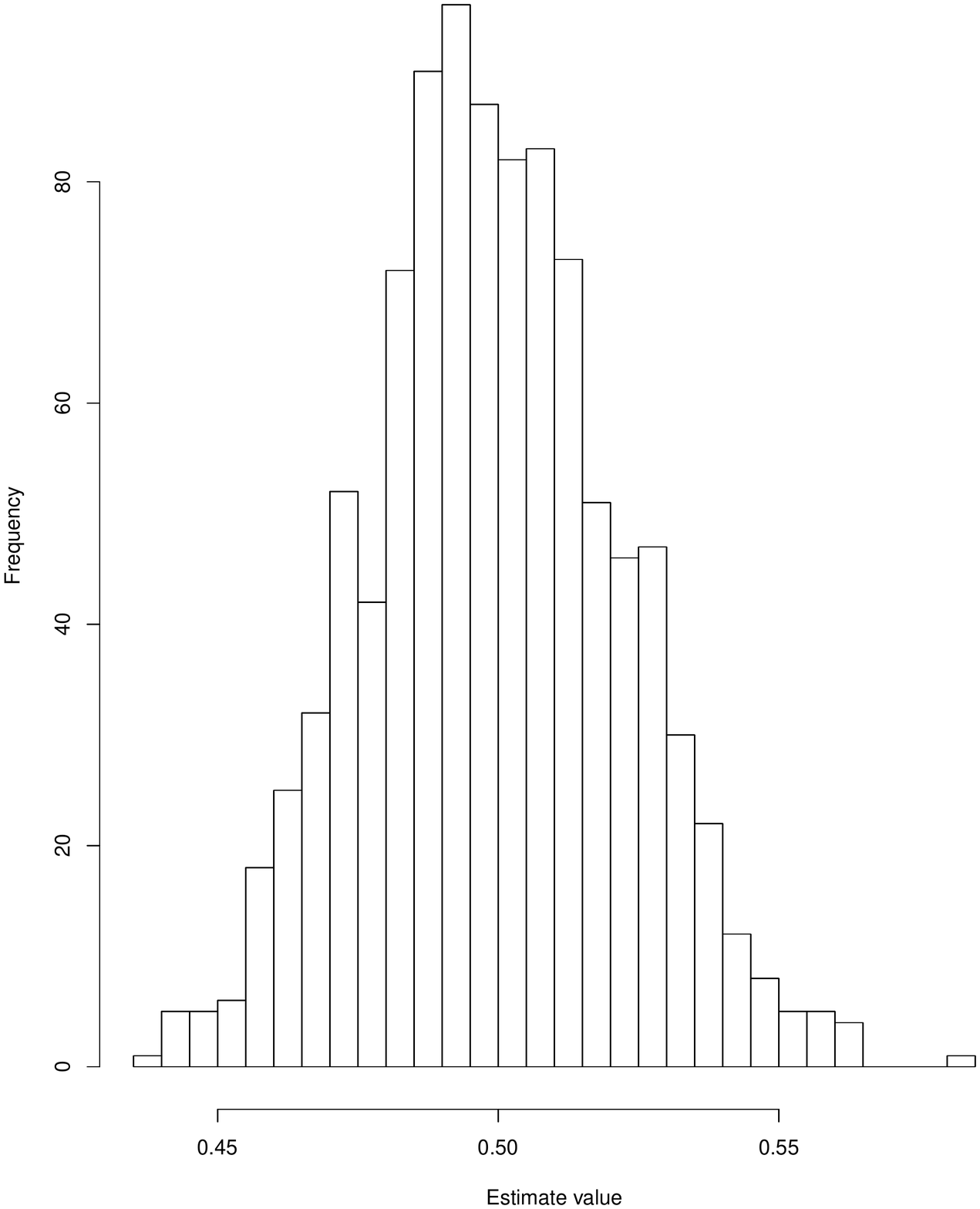}
    \caption{Estimates of $l_1$.}
   \label{fig:H=4per5-N=100000-l1}
 \end{subfigure}
\caption{Fractional Brownian motion liquidity with $H=\frac{4}{5}$ and $N=100000$.}
\label{fig:H=4per5-N=100000}
\end{figure}

\subsection{Compensated Poisson with $\lambda = 1$.} 
Histograms of the estimates of the model parameters corresponding to $L_t = (\tilde{N}_{t+1} - \tilde{N}_{t})^2$ with $\lambda = 1$ are provided in Figures \ref{fig:pois-N=100}, \ref{fig:pois-N=1000}, \ref{fig:pois-N=10000} and \ref{fig:pois-N=100000}. The used sample sizes were $N=100, N=1000, N=10000$ and $N= 100000$. The sample sizes $N=100$ and $N=1000$ resulted complex valued estimates in $45.1\%$ and $3.0\%$ of the simulations respectively, whereas with the larger sample sizes all the estimates were real.

\begin{figure}[H]
\centering
  \begin{subfigure}[t]{0.32\textwidth}
   \includegraphics[width=\textwidth]{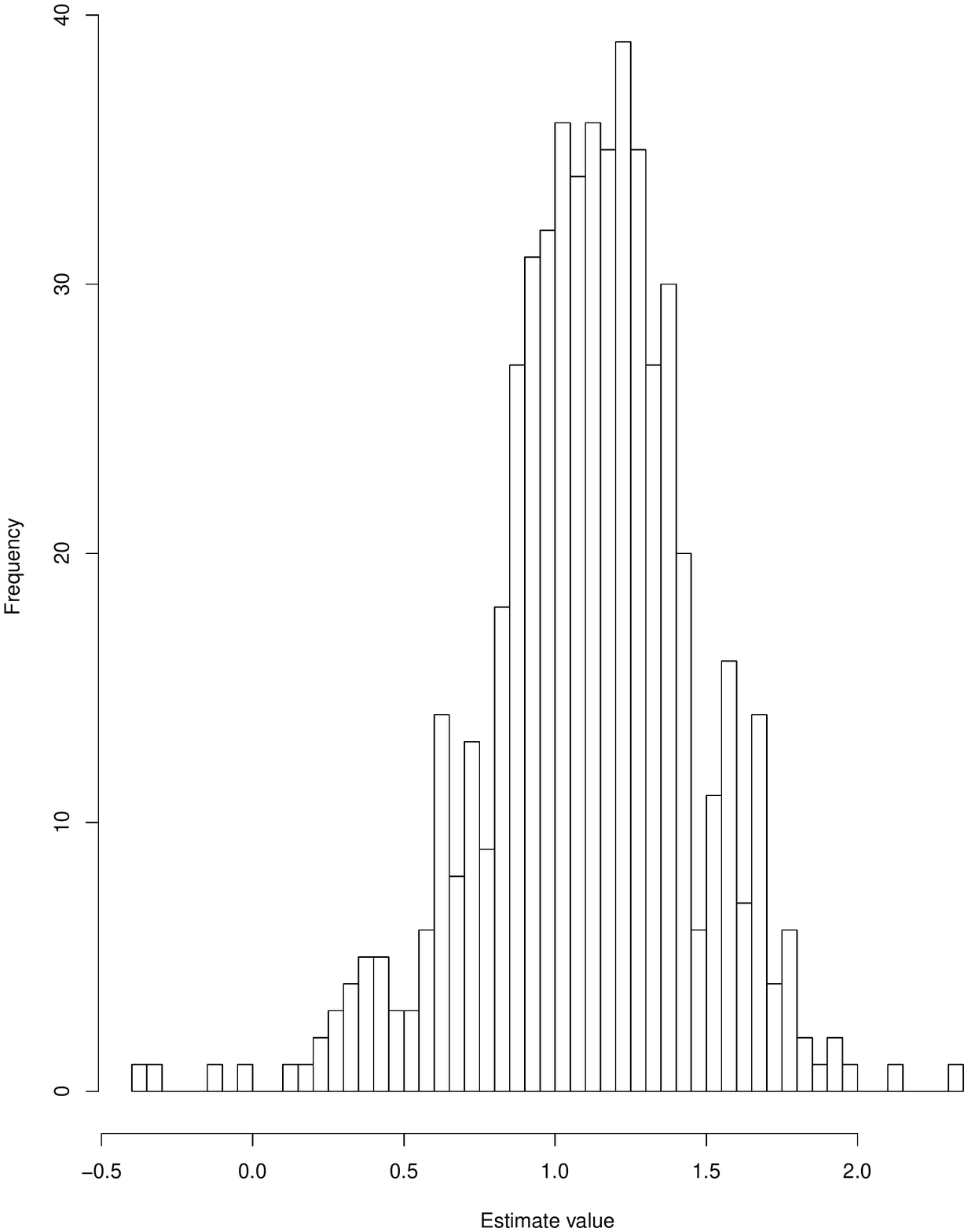}
    \caption{Estimates of $\alpha_0$.}
    \label{fig:pois-N=100-alpha0}
  \end{subfigure}
 \begin{subfigure}[t]{0.32\textwidth}
    \includegraphics[width=\textwidth]{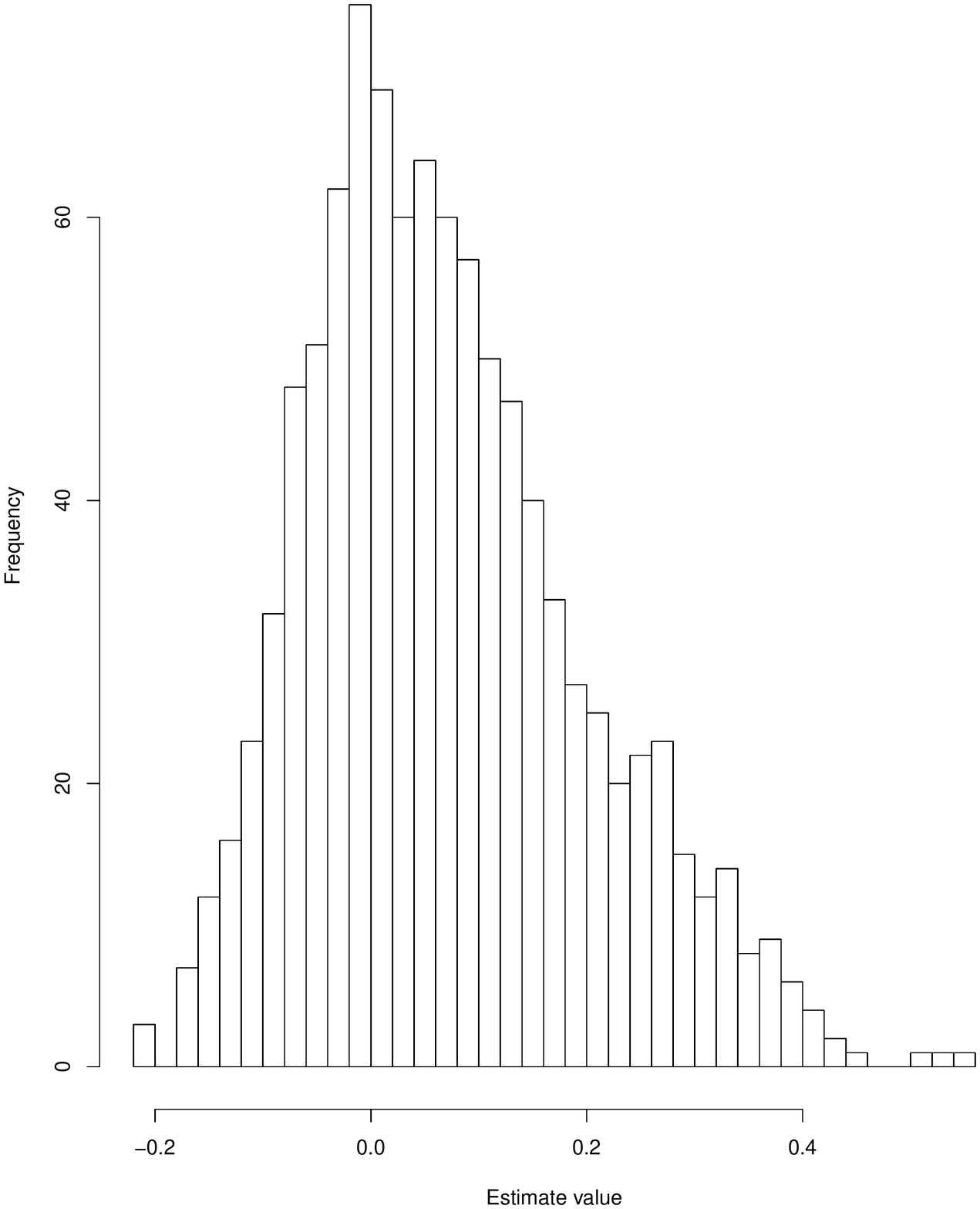}
    \caption{Estimates of $\alpha_1$.}
   \label{fig:pois-N=100-alpha1}
 \end{subfigure}
 \begin{subfigure}[t]{0.32\textwidth}
    \includegraphics[width=\textwidth]{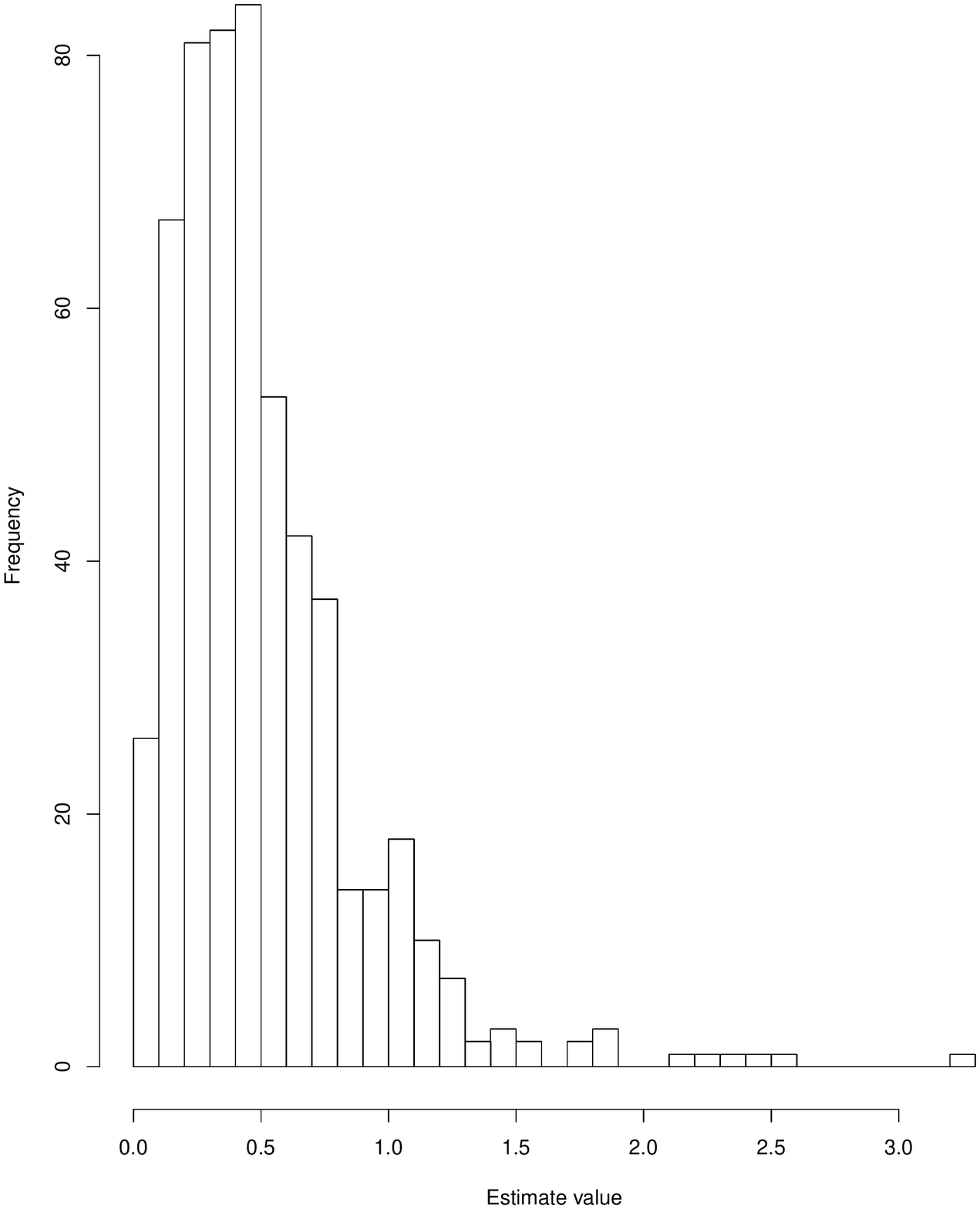}
    \caption{Estimates of $l_1$.}
   \label{fig:pois-N=100-l1}
 \end{subfigure}
\caption{Compensated Poisson liquidity with $\lambda=1$ and $N=100$.}
\label{fig:pois-N=100}
\end{figure}

\begin{figure}[H]
\centering
  \begin{subfigure}[t]{0.32\textwidth}
   \includegraphics[width=\textwidth]{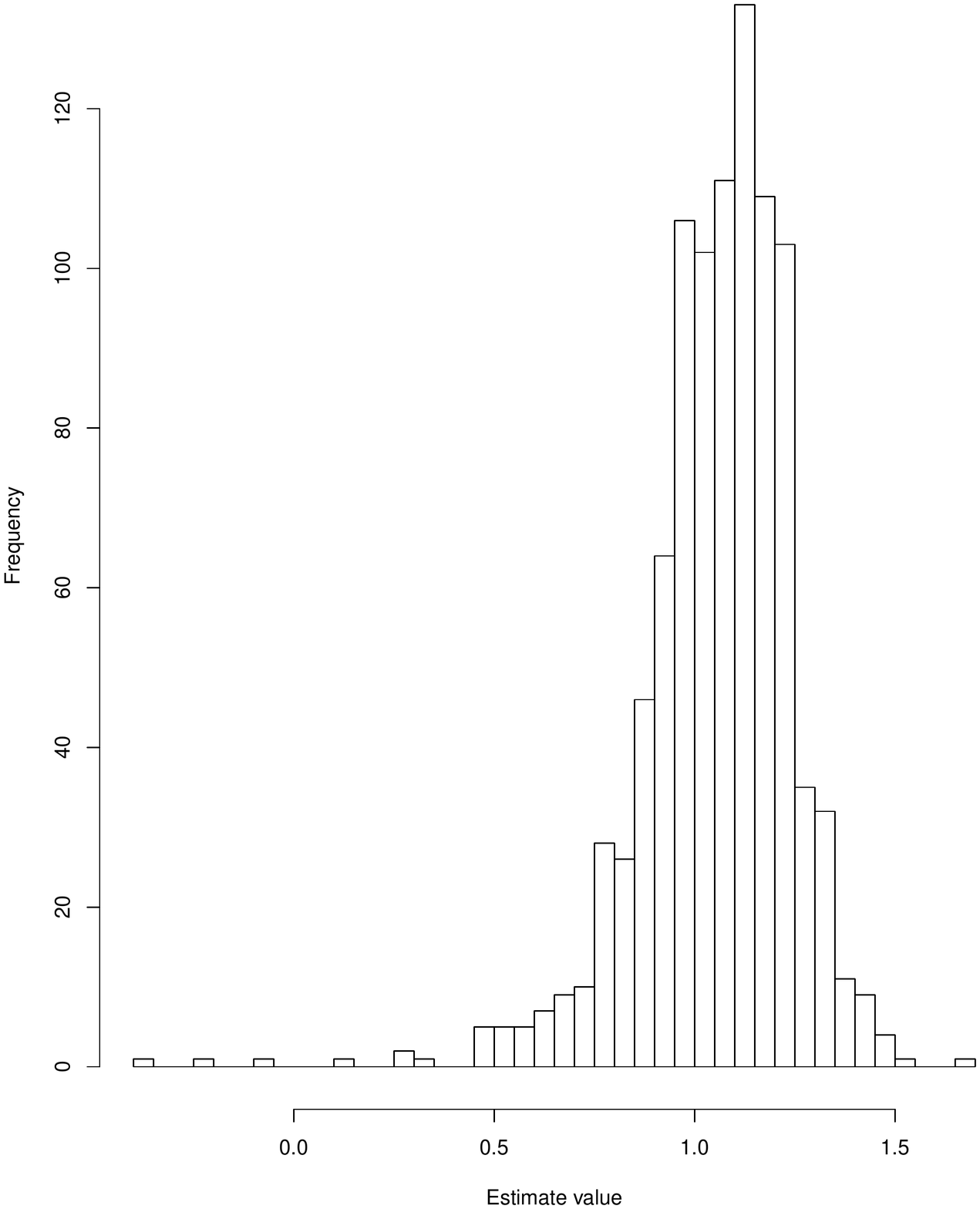}
    \caption{Estimates of $\alpha_0$.}
    \label{fig:pois-N=1000-alpha0}
  \end{subfigure}
 \begin{subfigure}[t]{0.32\textwidth}
    \includegraphics[width=\textwidth]{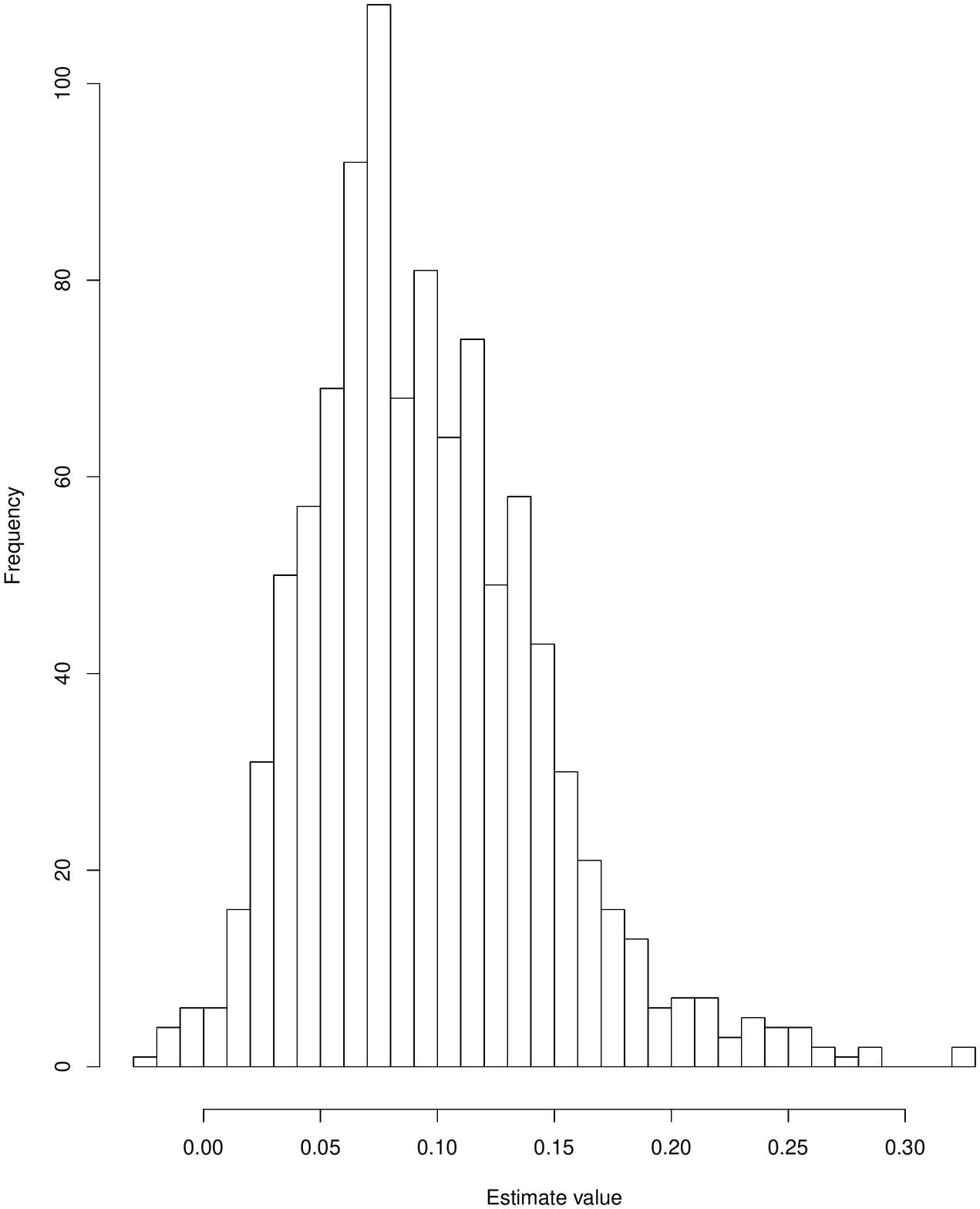}
    \caption{Estimates of $\alpha_1$.}
   \label{fig:pois-N=1000-alpha1}
 \end{subfigure}
 \begin{subfigure}[t]{0.32\textwidth}
    \includegraphics[width=\textwidth]{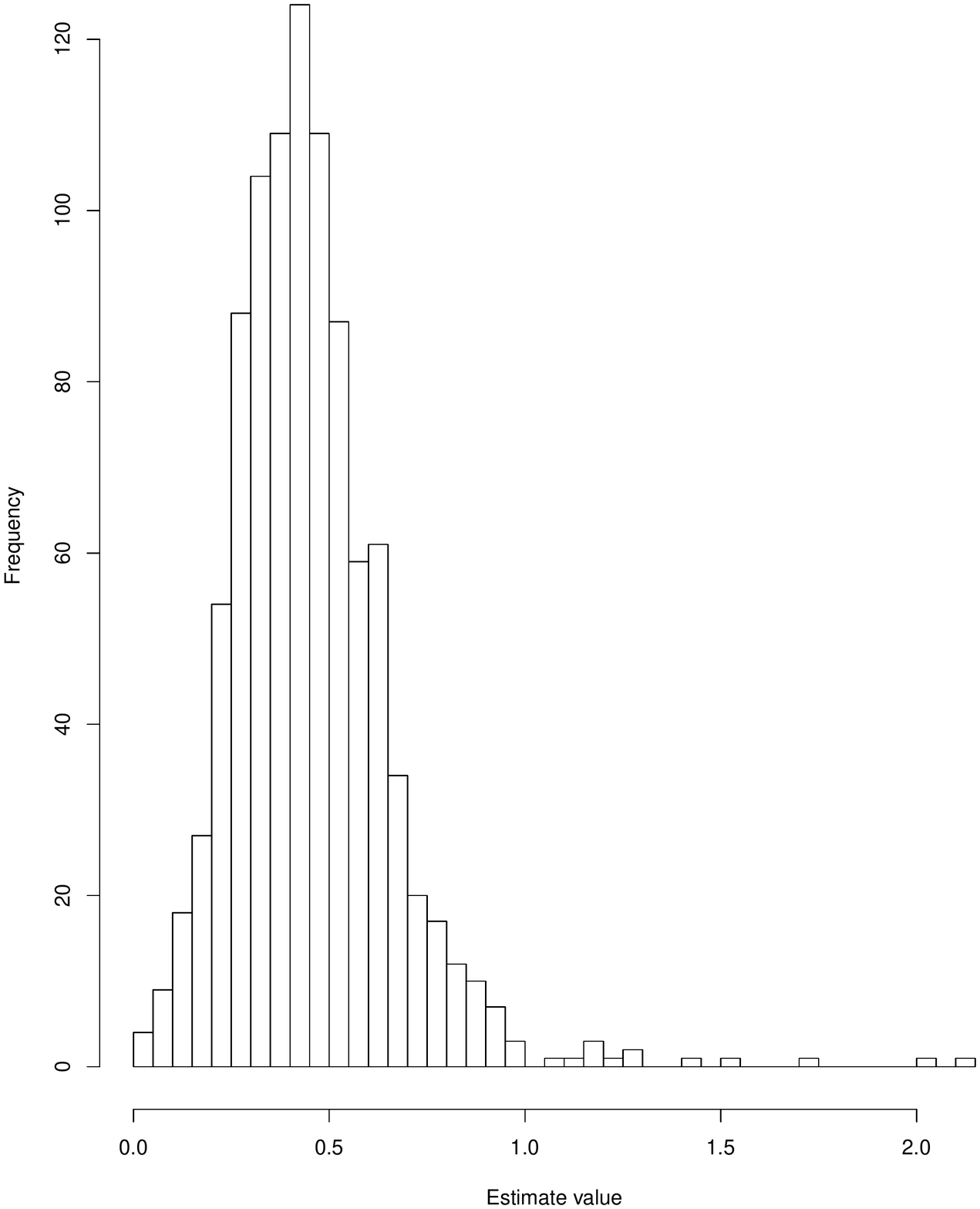}
    \caption{Estimates of $l_1$.}
   \label{fig:pois-N=1000-l1}
 \end{subfigure}
\caption{Compensated Poisson liquidity with $\lambda=1$ and $N=1000$.}
\label{fig:pois-N=1000}
\end{figure}

\begin{figure}[H]
\centering
  \begin{subfigure}[t]{0.32\textwidth}
   \includegraphics[width=\textwidth]{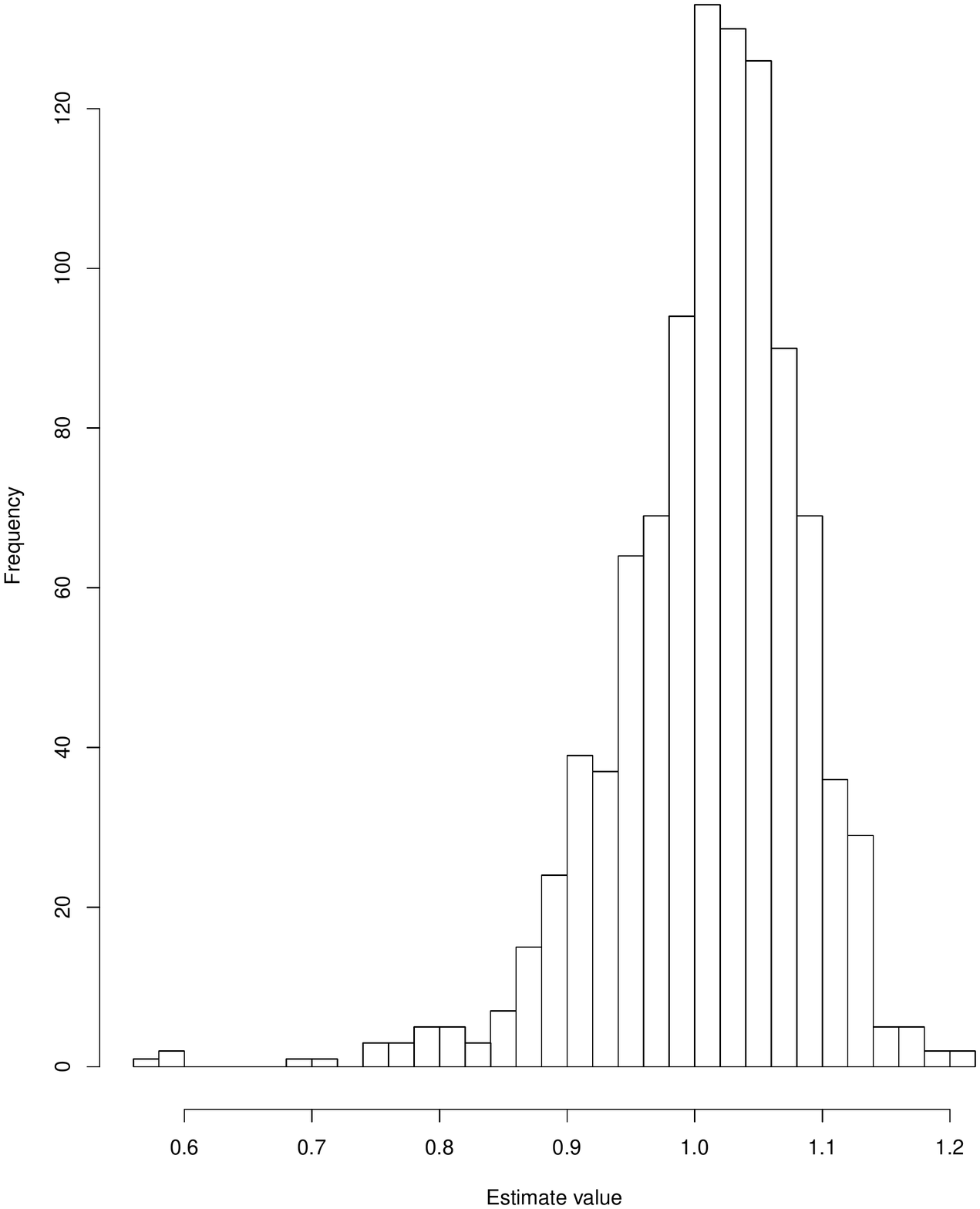}
    \caption{Estimates of $\alpha_0$.}
    \label{fig:pois-N=10000-alpha0}
  \end{subfigure}
 \begin{subfigure}[t]{0.32\textwidth}
    \includegraphics[width=\textwidth]{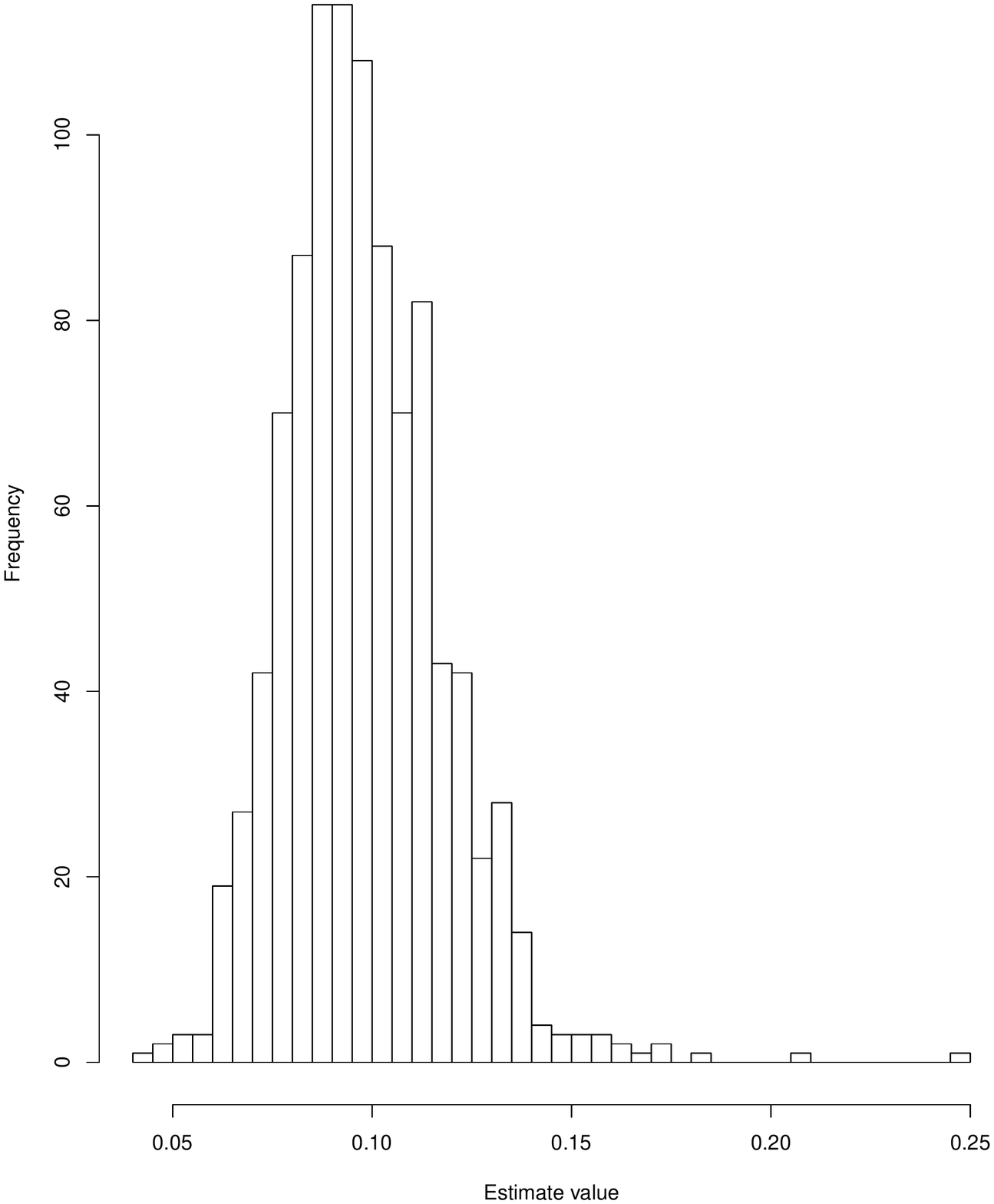}
    \caption{Estimates of $\alpha_1$.}
   \label{fig:pois-N=10000-alpha1}
 \end{subfigure}
 \begin{subfigure}[t]{0.32\textwidth}
    \includegraphics[width=\textwidth]{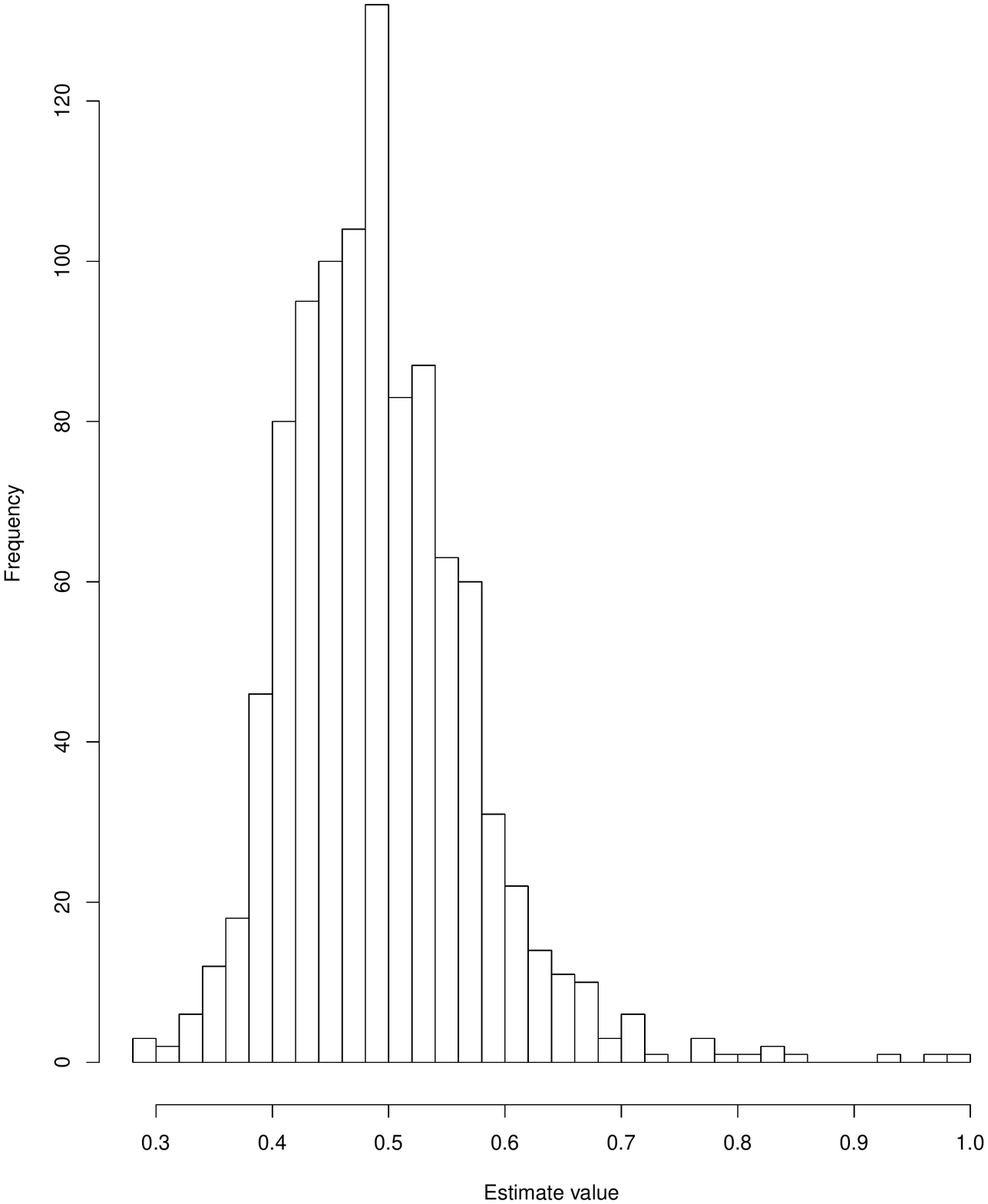}
    \caption{Estimates of $l_1$.}
   \label{fig:pois-N=10000-l1}
 \end{subfigure}
\caption{Compensated Poisson liquidity with $\lambda=1$ and $N=10000$.}
\label{fig:pois-N=10000}
\end{figure}

\begin{figure}[H]
\centering
  \begin{subfigure}[t]{0.32\textwidth}
   \includegraphics[width=\textwidth]{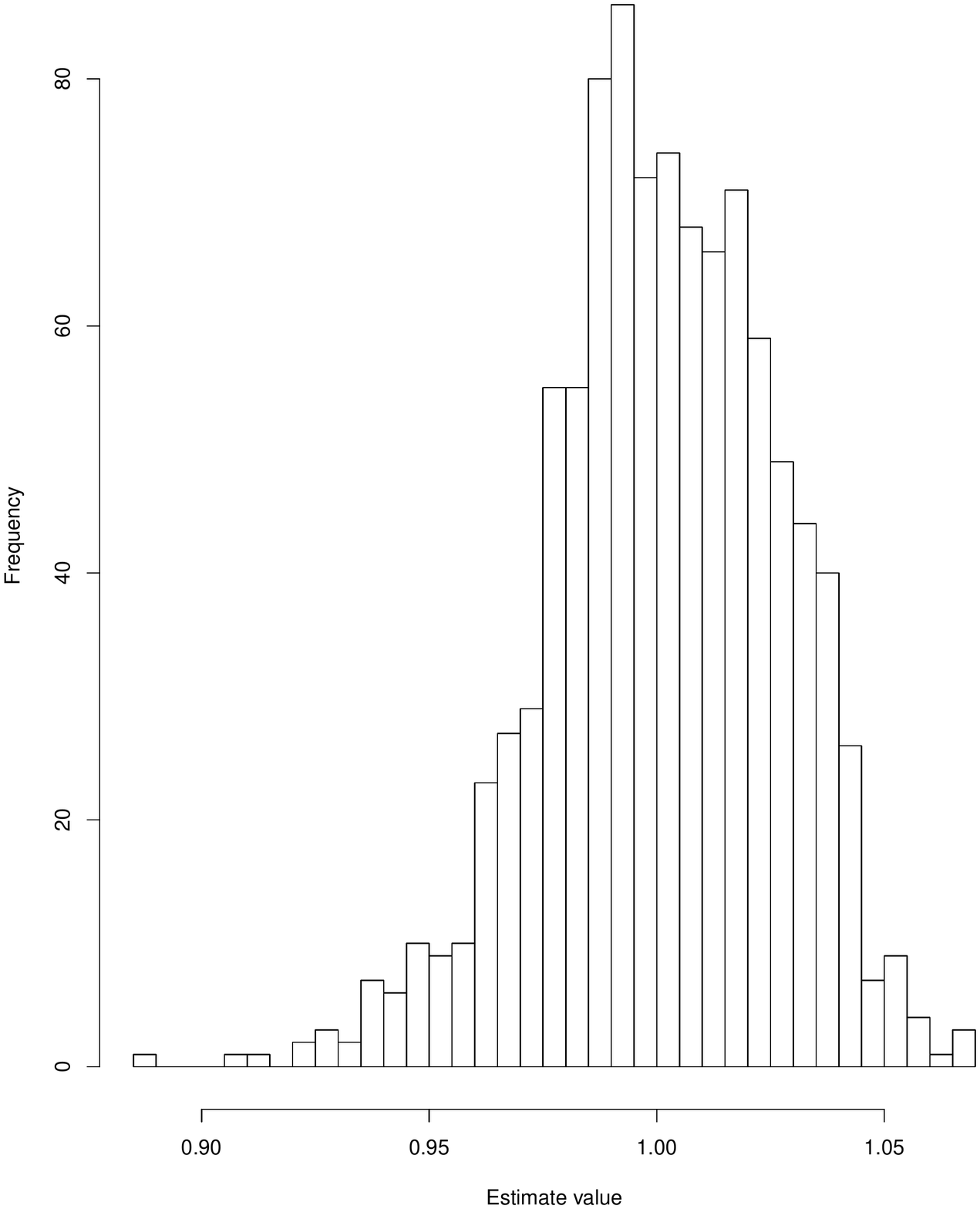}
    \caption{Estimates of $\alpha_0$.}
    \label{fig:pois-N=100000-alpha0}
  \end{subfigure}
 \begin{subfigure}[t]{0.32\textwidth}
    \includegraphics[width=\textwidth]{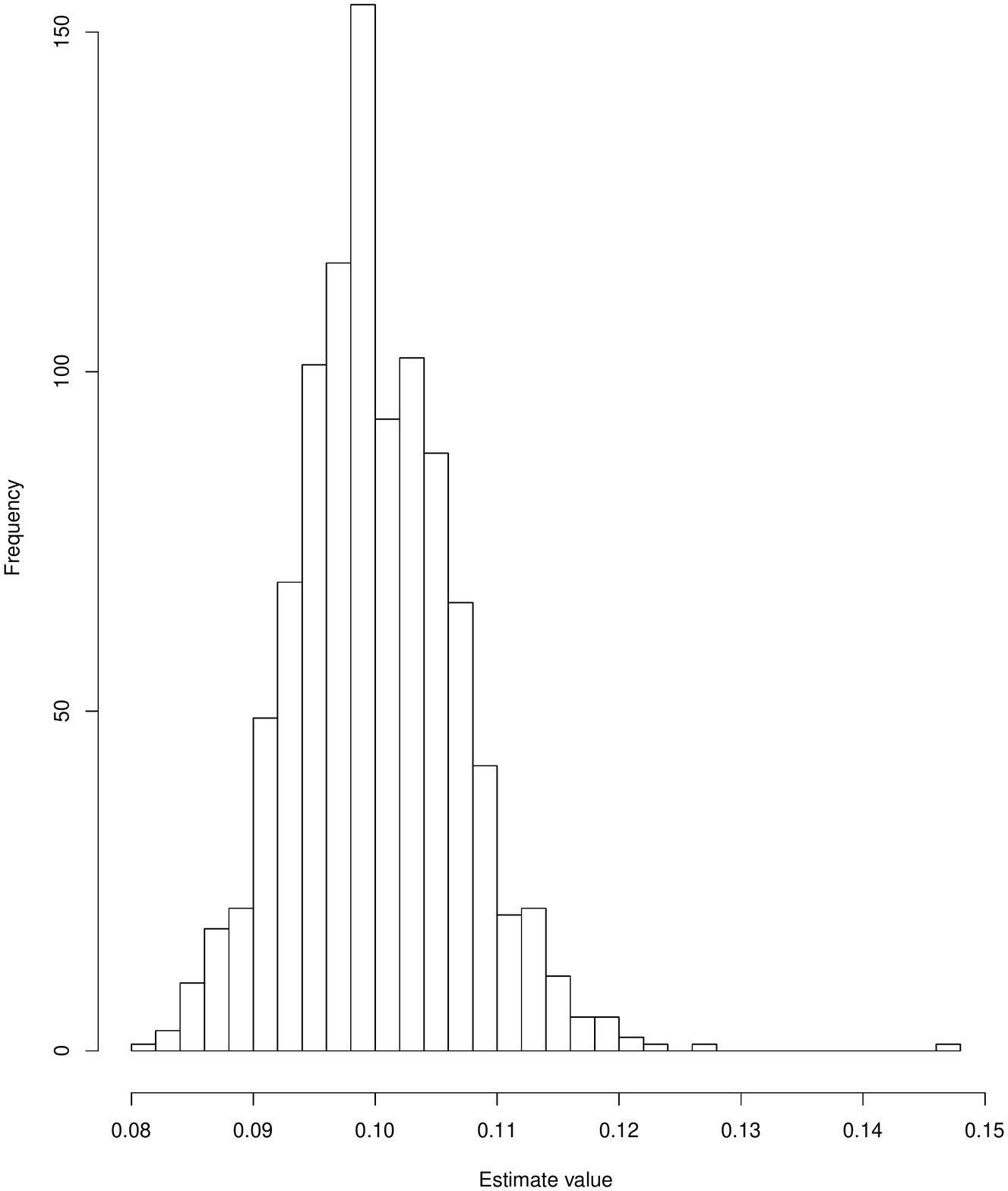}
    \caption{Estimates of $\alpha_1$.}
   \label{fig:pois-N=100000-alpha1}
 \end{subfigure}
 \begin{subfigure}[t]{0.32\textwidth}
    \includegraphics[width=\textwidth]{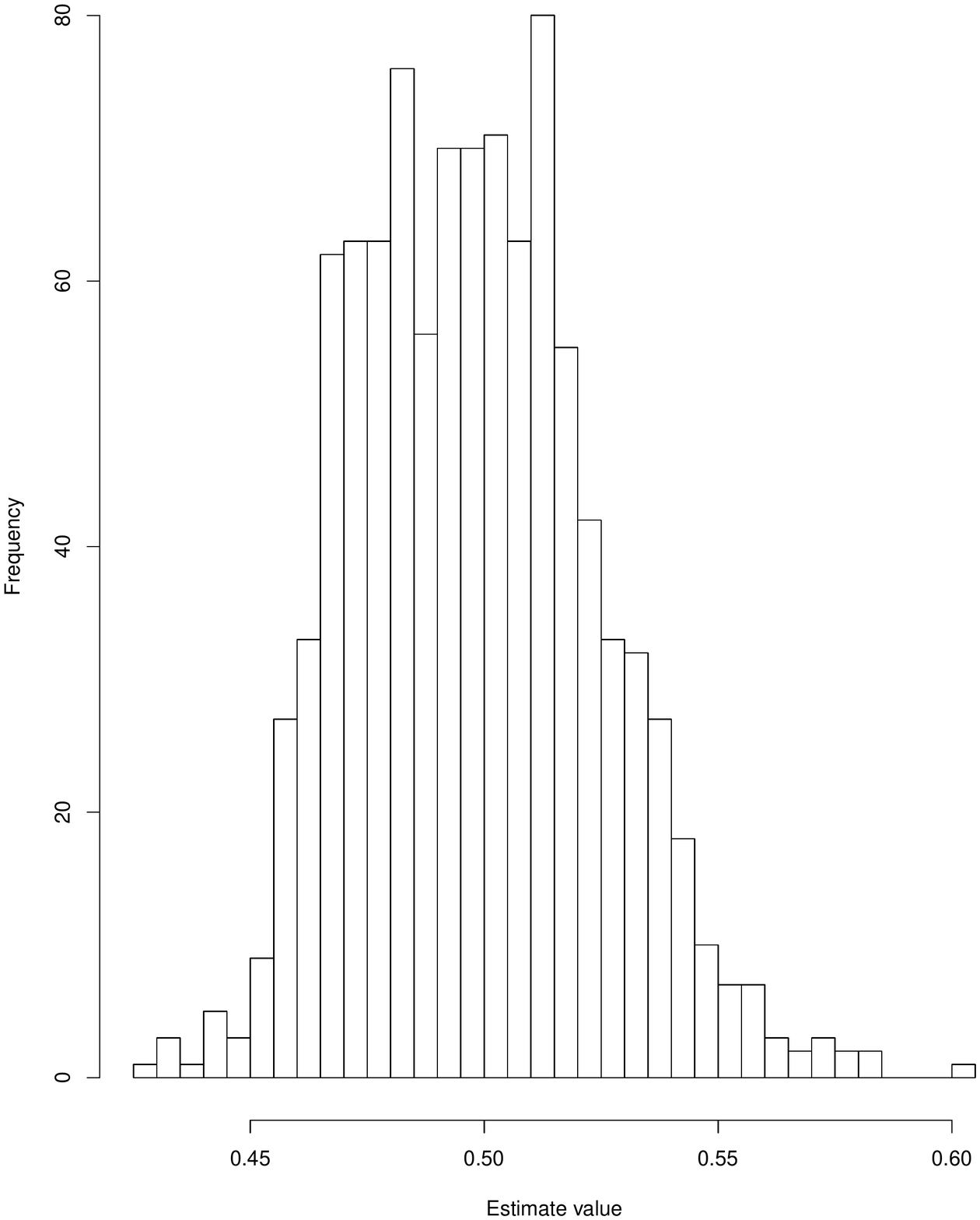}
    \caption{Estimates of $l_1$.}
   \label{fig:pois-N=100000-l1}
 \end{subfigure}
\caption{Compensated Poisson liquidity with $\lambda=1$ and $N=100000$.}
\label{fig:pois-N=100000}
\end{figure}

\appendix
\section{Tables}
\label{sec:A_tables}
In the following tables we have presented means and standard deviations of the estimates in different cases. In addition, we have provided tables demonstrating how the estimates match their theoretical intervals $0 \leq \alpha_0$, $0< \alpha_1 < \frac{1}{105^ \frac{1}{4}}$ and $0 < l_1$. We can see that multiplying the mean squared error (RMSE) provided by Tables 1-4 with $N^H$, the power $H$ of the sample size, gives us evidence of the convergence rates of the estimators.

\begin{table}[H]
\centering
\begin{tabular}{r|ccc}
  
  $N$ & $\alpha_0$ & $\alpha_1$ & $l_1$  \\
  \hline

100 & 1.063 (0.374) & 0.074 (0.124) & 0.541 (0.346)  \\
1000& 1.037 (0.169) & 0.100 (0.051) & 0.465 (0.167)   \\
10000 & 1.007 (0.061) & 0.100 (0.018) & 0.494 (0.060)  \\ 
100000 & 1.001 (0.019) &  0.100 (0.005) & 0.499 (0.019) \\
   \hline
\end{tabular}
\caption{Table of means and standard deviations corresponding to fractional Brownian motion liquidity with $H=\frac{1}{3}$.}
\end{table}

\begin{table}[H]
\centering
\begin{tabular}{r|ccc}
  
  $N$ & $\alpha_0$ & $\alpha_1$ & $l_1$  \\
  \hline

100 & 1.086 (0.359) & 0.080 (0.127) & 0.538 (0.358)  \\
1000& 1.029 (0.181) & 0.097 (0.052) & 0.479 (0.184)   \\
10000 & 1.005 (0.057) & 0.099 (0.017) & 0.497 (0.059)  \\ 
100000 & 1.000 (0.019) &  0.100 (0.005) & 0.500 (0.019) \\
   \hline
\end{tabular}
\caption{Table of means and standard deviations corresponding to fractional Brownian motion liquidity with $H=\frac{2}{3}$.}
\end{table}

\begin{table}[H]
\centering
\begin{tabular}{r|ccc}
  
  $N$ & $\alpha_0$ & $\alpha_1$ & $l_1$  \\
  \hline

100 & 1.068 (0.322) & 0.090 (0.139) & 0.535 (0.336)  \\
1000& 1.042 (0.163) & 0.098 (0.052) & 0.467 (0.180)   \\
10000 & 1.009 (0.059) & 0.099 (0.018) & 0.491 (0.064)  \\ 
100000 & 1.001 (0.020) &  0.100 (0.006) & 0.499 (0.022) \\
   \hline
\end{tabular}
\caption{Table of means and standard deviations corresponding to fractional Brownian motion liquidity with $H=\frac{4}{5}$.}
\end{table}

\begin{table}[H]
\centering
\begin{tabular}{r|ccc}
  
  $N$ & $\alpha_0$ & $\alpha_1$ & $l_1$  \\
  \hline

100 & 1.110 (0.346) & 0.071 (0.128) & 0.508 (0.384)  \\
1000& 1.057 (0.190) & 0.095 (0.050) & 0.452 (0.207)   \\
10000 & 1.011 (0.075) & 0.098 (0.020) & 0.493 (0.081)  \\ 
100000 & 1.001 (0.025) &  0.100 (0.007) & 0.498 (0.026) \\
   \hline
\end{tabular}
\caption{Table of means and standard deviations corresponding to compensated Poisson liquidity with $\lambda=1$.}
\end{table}


\begin{table}[H]
\centering
\begin{tabular}{r|ccc}
  
  $N$ & $\alpha_0$ & $\alpha_1$ & $l_1$  \\
  \hline

100 &  55.1&  65.4 &55.8  \\
1000&  96.5& 98.8 &96.5    \\
10000 & 100 & 100 & 100  \\ 
100000 &  100& 100  & 100 \\
   \hline
\end{tabular}
\caption{Table of percentages of the estimates lying on their theoretical intervals corresponding to fBm liquidity with $H=\frac{1}{3}$.}
\end{table}

\begin{table}[H]
\centering
\begin{tabular}{r|ccc}
  
  $N$ & $\alpha_0$ & $\alpha_1$ & $l_1$  \\
  \hline

100 &  54.0& 66.6 &54.5   \\
1000&  97.1& 99.0 & 97.1  \\
10000 &  100&100 &100   \\ 
100000 & 100 & 100& 100 \\
   \hline
\end{tabular}
\caption{Table of percentages of the estimates lying on their theoretical intervals corresponding to fBm liquidity with $H=\frac{2}{3}$.}
\end{table}

\begin{table}[H]
\centering
\begin{tabular}{r|ccc}
  
  $N$ & $\alpha_0$ & $\alpha_1$ & $l_1$  \\
  \hline

100 &  52.0& 65.2 & 52.1  \\
1000& 95.7& 98.5 & 95.7   \\
10000 &  100&100&100   \\ 
100000 & 100 & 100  & 100 \\
   \hline
\end{tabular}
\caption{Table of percentages of the estimates lying on their theoretical intervals corresponding to fBm liquidity with $H=\frac{4}{5}$.}
\end{table}

\begin{table}[H]
\centering
\begin{tabular}{r|ccc}
  
  $N$ & $\alpha_0$ & $\alpha_1$ & $l_1$  \\
  \hline

100 & 54.9& 61.8& 55.3  \\
1000& 96.6& 98.7 & 96.9  \\
10000 & 100 & 100& 100 \\ 
100000 & 100 & 100 & 100 \\
   \hline
\end{tabular}
\caption{Table of percentages of the estimates lying on their theoretical intervals corresponding to compensated Poisson liquidity with $\lambda = 1$.}
\end{table}

\end{document}